\documentclass[a4paper]{amsart}

\usepackage[top=1.5in,right=1in,left=1in,bottom=1in,a4paper]{geometry}
\usepackage{amsmath}
\usepackage{amsthm}
\usepackage{amssymb}
\usepackage{amsfonts}
\usepackage{stmaryrd}
\usepackage{enumerate}
\usepackage{tikz}
\usetikzlibrary{trees}
\usepackage[all]{xy}
\usepackage[titletoc,toc,title,page]{appendix}
\usepackage{longtable}
\usepackage{pgfplots}
\pgfplotsset{compat=1.17}
\usepackage{attachfile}

\newtheoremstyle{theoremstyle}
  {10pt}      
  {5pt}       
  {\itshape}  
  {}          
  {\bfseries} 
  {:}         
  {.5em}      
  {}          

\newtheoremstyle{examplestyle}
  {10pt}      
  {5pt}       
  {}          
  {}          
  {\bfseries} 
  {:}         
  {.5em}      
  {}          

\newtheoremstyle{sub-style}
  {10pt}      
  {5pt}       
  {}          
  {}          
  {\bfseries} 
  {.}         
  {.5em}      
  {}          

\theoremstyle{theoremstyle}
\newtheorem{theorem}{Theorem}[section]
\newtheorem*{theorem*}{Theorem}
\newtheorem{lemma}[theorem]{Lemma}
\newtheorem{proposition}[theorem]{Proposition}
\newtheorem*{proposition*}{Proposition}
\newtheorem{corollary}[theorem]{Corollary}
\newtheorem*{corollary*}{Corollary}

\theoremstyle{examplestyle}
\newtheorem{example}[theorem]{Example}
\newtheorem*{example*}{Example}
\newtheorem{definition}[theorem]{Definition}
\newtheorem*{definition*}{Definition}
\newtheorem{remark}[theorem]{Remark}
\newtheorem*{remarks*}{Remarks}

\newtheorem*{remark*}{Remark}

\theoremstyle{sub-style}
\newtheorem{sub}[theorem]{}

\newcommand{\Z}{\mathbb{Z}}
\newcommand{\C}{\mathbb{C}}
\newcommand{\Q}{\mathbb{Q}}
\newcommand{\R}{\mathbb{R}}

\begin{document}

\title[Markoff equation and T-Singularities]{Unfocused notes on the Markoff equation and T-Singularities}

\subjclass[2020]{Primary: 14J17, 11J06; Secondary: 28A78, 28A80, 14F08}

\author{Markus Perling}

\begin{abstract}
We consider minimal resolutions of the singularities for weighted projective planes
of type $\mathbb{P}(e^2, f^2, g^2)$, where $e, f, g$ satisfy the Markoff equation
$ e^2 + f^2 + g^2 = 3efg$. We give a complete classification of such resolutions in
terms of continued fractions similar to classical work of Frobenius.
In particular, we investigate the behaviour of resolutions under mutations and describe
a Cantor set emerging as limits of continued fractions.
\end{abstract}

\maketitle


\section{Introduction}

The main aim of these notes is to describe the minimal resolutions of singularities for weighted
projective planes $\mathbb{P}(e^2, f^2, g^2)$, where $e, f, g$ are integral
solutions of the Markoff equation: $$e^2 + f^2 + g^2 = 3efg.$$
It has been observed in \cite{HackingProkhorov10} that these singularities
are T-singularities in the sense of Wahl \cite{Wahl81}. Such singularities have been
connected to exceptional vector bundles on rational surfaces by results of
Hacking-Prokhorov \cite{HackingProkhorov10} and Hacking \cite{Hacking13}. This aspect
has also been investigated in our earlier work \cite{HillePerling08}, \cite{Perling18}.
We think that the complete understanding of these resolutions and the fractal
structure emerging via the associated Markoff tree are of independent interest, but our
hope is that this also may lead to further insights into the global structure of exceptional collections.

Consider a T-singularity of type $\frac{1}{n^2}(1, nk - 1)$, i.e. a cyclic quotient
singularity of order $n^2$ with weights $1, nk - 1$. The minimal resolution of such
a singularity can be desribed via Hirzebruch-Jung continued fractions $\llbracket a_1,
\dots, a_t \rrbracket = a_1 - 1 / (a_2 - 1 / (a_3 - \cdots (a_{t - 1} - 1 / a_t) \cdots ))$.
These continued fractions have been classified for T-singularities in \cite{KollarShepherdbarron88}
(see also Section \ref{tsingsection}). Our first basic tool will be a simplified representation
of these Hirzebruch-Jung continued fractions in terms of integer sequences
$c_1, \dots, c_s$ (see Section \ref{tsingsection}), which is based on the
Koll\'ar--{Shepherd-Barron} classification. It turns out that the $c_i$ can be
recovered from the {\em regular} continued fraction expansion $n^2 / (n k - 1) =
[c_s + 1, c_{s - 1}, \dots, c_2, c_1 + x, c_1 + 2 - x, c_2, \dots, c_{s - 1}, c_s + 1]$,
where $x = 2$ if $s$ is even and $x = 0$ if $s$ is odd (Theorem \ref{LEProposition}).

Now, for any Markoff triple $(e, f, g)$ with $1 < e, f < g$, $\mathbb{P}(e^2, f^2, g^2)$ has
the three T-singularities $\frac{1}{e^2}(1, e w_e - 1)$, $\frac{1}{f^2}(1, f w_f - 1)$
and $\frac{1}{g^2}(1, g w_g - 1)$, respectively, where we call $w_e, w_f, w_g$ the
{\em T-weights} of this triple. When passing from one Markoff triple to the next,
say $(e, g, E)$ with $E = 3fg - e$, we will show in Section \ref{finalchapter}
how to construct the continued fraction
of $E^2 / (E w_E - 1)$ from those of $\frac{1}{e^2}(1, e w_e - 1)$, $\frac{1}{f^2}(1, f w_f - 1)$
and $\frac{1}{g^2}(1, g w_g - 1)$, respectively, and thereby obtain a complete classification
of the minimal resolutions of T-singularities associated to Markoff triples in terms of
their mutations along the Markoff tree.

A strongly related question is the distribution of quotients $g / w_g \approx g^2 / (g w_g - 1)$
and whether limits of
such quotients exists when mutating Markoff triples downward paths in the Markoff tree.
This is very closely related to classical work by Frobenius \cite{Frobenius}. In loc. cit.,
Frobenius derives a complete description of continued fraction expansions of quotients
$g / r_g$, where $r_g$ is associated to a triple that contains $g$ such that
$r_g^2 \equiv -1 \mod g$. Using Frobenius' description, it is straightforward to see
that limits $\lim_{i \rightarrow \infty} r_{g_i} / g_i$ exist, for any sequence $g_i$
that follows a path downward the Markoff tree (Proposition \ref{limitsexist}).
It turns out that the set of these limit points forms a Cantor set. We will show that
this set has measure and Hausdorff dimension zero (Theorems \ref{lebesguetheorem} and
\ref{hausdorfftheorem}). Moreover, we will characterize this Cantor set by its natural
intervals and we will show that the boundary points of these intervals are quadratic
numbers, which are explicitly described in Proposition \ref{quadraticspectrum1}.

The quotients $w_g / g$ are related to the $r_g / g$ simply by the affine transform
$w_g / g = 3 r_g / g - 1$ for every $g$ (Corollary \ref{affinetransform}).
Therefore, their limit points via the Markoff tree exist
as well and the resulting Cantor set is an affine transform of the one
associated to the $r_g / g$. Their dimension and measure therefore coincide (Theorem \ref{tcantor}).
The boundary points are quadratic as well and their explicit formulas and periodic continued
fraction expansions are computed in Secton \ref{finalchapter}.

We point out that the quotients $r_g/g$ have shown up in earlier work as the slopes of
exceptional vector bundles on $\mathbb{P}^2$, see \cite{Rudakov89a}, \cite{DrezetLePotier},
and their fractal nature has already been observed in \cite{DrezetLePotier}.

{\bf Overview.}
In Section \ref{treesection} we fix some general conventions regarding the Markoff tree that
we will follow for the rest of these notes.
In Section \ref{frobenius}, we collect all fundamental facts and computations related to
Markoff triples that will be needed in the subsequent sections. This includes some well-known
formulas, but also some new formulas for which we could not find any reference in the literature.
Also, we review Frobenius' results on continued fractions related to Markoff triples.
In Section \ref{cantor} we use the results of Section \ref{frobenius} to analyze the
Cantor set that emerges from the quotients $r_g / g$.
In Section \ref{equationsection}, we introduce T-weights and perform some calculations
similar as in Section \ref{frobenius}.
In Section \ref{tsingsection}, we analyze the continued fraction expansions of T-singularities.
We then apply these results in the final Section \ref{finalchapter} in order to classify
the minimal resolutions for T-singularities associated to the Markoff equation and to
describe the Cantor set that emerges from the quotients $w_g / g$.

In Appendix \ref{continuants} we collect some basic facts about continued fractions
and continuants that we will need throughout the text.
In Appendix \ref{tcontinuants}
we introduce a new class of polynomials related to T-singularities. These are not
needed in the main body of these notes, but they may be of independent interest.
In Appendix \ref{cantorappendix}, we collect some material that we need in order
to deal with the Hausdorff dimension.
In Appendix \ref{firstfewlevels}, we display the first few levels of the Markoff tree,
both only the maximal elements and the triples. We found it quite useful to be able to
occasionally pinpoint a specific example and its location in the tree in one peek.
Appendix \ref{threehundred} has a computer generated list of the first 300 Markoff numbers
together with their weights and T-weights. If needed, the table can be extracted from
this documents \LaTeX{} source.
In Appendix \ref{experiments} we present some simplistic numerical
experiments regarding the growth of Markoff numbers and the uniqueness conjecture.

\section{The Markoff Tree}\label{treesection}

In this section we will state some general facts and fix our general conventions regarding the Markoff tree
(we refer to \cite{Aigner13} for an overview).
We are exclusively interested
in positive integral solutions $(e, f, g)$ of the Markoff equation $e^2 + f^2 + g^2 = 3efg$, which we will call
{\em Markoff triples}. If for any given Markoff triple $(e, f, g)$ we denote $E = 3fg - e$, $F = 3eg - f$, and $G = 3ef - g$,
then  we obtain three more Markoff triples $(E, f, g)$, $(e, F, g)$, and $(e, f, G)$. We call these substitutions
{\em mutations} of the triple $(e, f, g)$. By rearranging the formulas
$3eg = f + F$, $3fg = e + E$, and $3ef = g + G$, it becomes obvious that mutations are reversible.
All Markoff triples can be obtained by mutation, starting from the fundamental solution $(1, 1, 1)$.
So we can consider
the set of all Markoff triples as vertices of a trivalent acyclic graph whose edges are labeled by
the mutations, where the fundamental solution is the natural root. By the symmetry of the Markoff equation,
we have a natural action of the symmetric group $S_3$ on the set of Markoff triples. This action has an almost
everywhere trivial stabilizer with exceptions $(1, 1, 1)$ and the permutations of $(1, 2, 1)$. These triples,
which are the only ones where $e, f, g$ are not pairwise distinct, are called {\em singular} solutions of the
Markoff equation. All other Markoff triples are called {\em regular}.
The set of regular Markoff triples up to permutation forms a binary tree, whose root is represented by the
triple $(1, 5, 2)$. Figure \ref{treefig1} shows the first few levels of this tree.
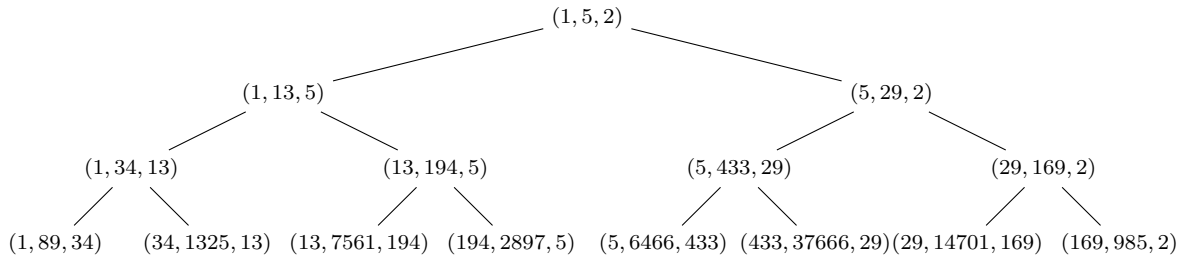
\begin{figure}[ht]
\footnotesize
\tikzstyle{level 1}=[level distance = 1cm, sibling distance=8cm]
\tikzstyle{level 2}=[sibling distance=4cm]
\tikzstyle{level 3}=[sibling distance=2cm]
\begin{tikzpicture}[]
  \node {$(1, 5, 2)$}
    child {node {$(1, 13, 5)$}
      child {node {$(1, 34, 13)$}
        child {node[align=right] {$(1, 89, 34)$}}
        child {node[align=right] {$(34, 1325, 13)$}}
      }
      child {node {$(13, 194, 5)$}
        child {node[align=right] {$(13, 7561, 194)$}}
        child {node[align=right] {$(194, 2897, 5)$}}
      }
    }
    child {node {$(5, 29, 2)$}
      child {node {$(5, 433, 29)$}
        child {node[align=right] {$(5, 6466, 433)$}}
        child {node[align=right] {$(433, 37666, 29)$}}
      }
      child {node {$(29, 169, 2)$}
        child {node[align=right] {$(29, 14701, 169)$}}
        child {node[align=right] {$(169, 985, 2)$}}
      }
    };
\end{tikzpicture}
\caption{A part of the Markoff tree.}\label{treefig1}
\end{figure}
As in Figure \ref{treefig1}, we will use from now on the convention that every regular Markoff
triple will be represented by $(e, g, f)$, such that $0 < e, f < g$. Because $g < E, F$, mutation
of $e$ or $f$ does mean that we descend downwards
the Markoff tree, where we fix the direction as indicated in Figure \ref{mutationfig}.
\begin{figure}[ht]
\footnotesize
\tikzstyle{level 1}=[level distance = 1cm, sibling distance=4cm]
\begin{tikzpicture}[]
  \node {$(e, g, f)$}
    child {node {$(e, F, g)$}}
    child {node {$(g, E, f)$}};
\end{tikzpicture}
\caption{Left and right mutation.}\label{mutationfig}
\end{figure}
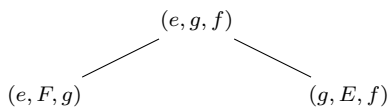
We will in call the triple $(e, F, g)$ the {\em left mutation} and $(g, E, f)$ the {\em right mutation} of $(e, f, g)$,
respectively. With this convention, if we go upwards by mutating $g$, then we will either end up with the triple
$(e, f, G)$ (if $e < f$) or $(G, e, f)$ (if $e > f$). By requiring that the root is represented by $(1, 5, 2)$,
the ordering of every representative is thus uniquely fixed. In particular, if we traverse any level below the root
from left to right then we find that alternatingly $e < f$ or $e > f$, with $1 = e < f$ for the leftmost vertex
and $e > f = 2$ for the rightmost vertex.

It follows that, given a triple $(e, g, f)$, then starting from this triple, there are precisely four branches 
where at least one of $e, f, g$ is preserved, as indicated by Figure \ref{chains}.
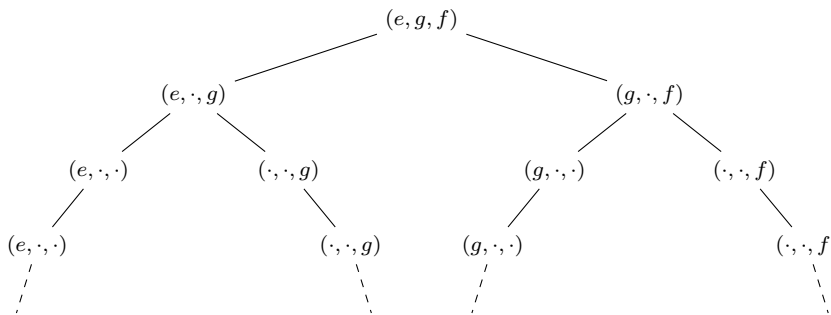
\begin{figure}[ht]
\footnotesize
\tikzstyle{level 1}=[level distance = 1cm, sibling distance=6cm]
\tikzstyle{level 2}=[level distance = 1cm, sibling distance=2.5cm]
\begin{tikzpicture}[]
  \node {$(e, g, f)$}
    child {node {$(e, \cdot, g)$}
      child {node {$(e, \cdot, \cdot)$}
    child {node[left=.3cm] {$(e, \cdot, \cdot)$}
      child[dashed] {node[left=.2cm] {\ }}
    }
      }
      child {node {$(\cdot, \cdot, g)$}
        child {node[right=0.3cm] {$(\cdot, \cdot, g)$}
            child[dashed] {node[right=.2cm] {\ }}
    }
      }
    }
    child {node {$(g, \cdot, f)$}
      child {node {$(g, \cdot, \cdot)$}
    child {node[left=.3cm] {$(g, \cdot, \cdot)$}
      child[dashed] {node[left=.2cm] {\ }}
    }
      }
      child {node {$(\cdot, \cdot, f)$}
    child {node[right=.3cm] {$(\cdot, \cdot, f)$}
      child[dashed] {node[right=.2cm] {\ }}
    }
      }
    };
\end{tikzpicture}
\caption{How the elements in a Markoff triple propagate downwards.}\label{chains}
\end{figure}

The two outmost branches of the Markoff tree are of the form
$$
(1, F_{2n + 1}, F_{2n - 1}) \text{ for } n \geq 2 \quad \text  \and \quad
(P_{2n -1}, P_{2n + 1}, 2) \text{ for } n \geq 1,
$$
where $F_i$ and $P_i$ denote the Fibonacci and Pell numbers. Note that $5 = F_5 = P_3$.
We call these the {\em Fibonacci} and {\em Pell branches}, respectively.

\section{The Markoff tree and finite continued fractions \`a la Frobenius}\label{frobenius}

In this section we give an expository overview around some aspects of Frobenius' classical account \cite{Frobenius}
on the Markoff equation which are related to finite continued fractions. We also refer to
\cite[Chapter II]{Cassels57}, \cite{CusickFlahive}, \cite{Aigner13}, \cite{Reutenauer19}. We will
paraphrase some of the material to suit our needs and complement it by a few observations of our own which, though probably well-known
to specialists, we couldn't locate in the literature.

\subsection{Weights and coweights}

Denote $(e, g, f)$ a regular Markoff triple. Then we set:
\begin{align*}
r_e & := 
\begin{cases}
0 & \text{ if } e = 1, \\
f^{-1}g \bmod e & \text{ else},
\end{cases}\\
r_f & := g^{-1}e \bmod f, \\
r_g & := e^{-1} f \bmod g.
\end{align*}
Note that the choice $r_e = 0$ for $e = 1$ is made for consistency, as we will see below.
Then the Markoff equation implies that
$$
r_e^2 \equiv -1 \bmod e \,\, (\text{for } e > 1), \quad r_f^2 \equiv -1 \bmod f, \quad \text{ and } \quad r_g^2 \equiv -1 \bmod g.
$$
We choose to call the triple $(r_e, r_g, r_f)$ the {\em weights} associated to the triple $(e, g, f)$. Moreover, we define the
{\em coweights} $(s_e, s_g, s_f)$ via the following equations:
$$
r_e^2 = -1 + s_e e, \quad r_f^2 = -1 + s_f f, \quad r_g^2 = -1 + s_g g.
$$
Note that in particular $s_e = 1$ for $e = 1$. The weights and coweights satisfy a number of nice equations, in particular
\begin{align}
g r_f - f r_g & = e, \nonumber \\
e r_g - g r_e & = f, \label{cofactorequations1} \\
e r_f - f r_e & = G \nonumber
\end{align}
(see \cite[\S II.3]{Cassels57} or \cite[p. 602]{Frobenius}).

\begin{remark*}
In \cite{Rudakov89a} the weights have been identified with the first Chern classes of the
vector bundles in an exceptional collection on $\mathbb{P}^2$ and the ratios $r_e / e, r_g / g, r_f / f$
with their slopes.
\end{remark*}

\begin{lemma}\label{cofactorlemmaf}
Let $(e, g, f)$ be a regular Markoff triple.
\begin{enumerate}[(i)]
\item\label{cofactorlemmafi}
The weights and coweights transform under mutation as follows:
\begin{center}
\tikzstyle{level 1}=[level distance = 2cm, sibling distance=4cm]
\begin{tikzpicture}[]
  \node{\begin{tabular}{r@{\hskip 1pt}rrr@{\hskip 1pt}l}( & $e$, & $g$, & $f$ & )\\ (& $r_e$, & $r_g$, & $r_f$ & )\\ ( & $s_e$, & $s_g$, & $s_f$ & )\end{tabular}}
    child {node {\begin{tabular}{r@{\hskip 1pt}rrr@{\hskip 1pt}l}( & $e$, & $F$, & $g$ & )\\ (& $r_e$, & $r_F$, & $r_g$ & )\\ ( & $s_e$, & $s_F$, & $s_g$ & )\end{tabular}}}
    child {node {\begin{tabular}{r@{\hskip 1pt}rrr@{\hskip 1pt}l}( & $g$, & $E$, & $f$ & )\\ (& $r_g$, & $r_E$, & $r_f$ & )\\ ( & $s_g$, & $s_E$, & $s_f$ & )\end{tabular}}};
\end{tikzpicture}
\end{center}
where
$$
r_E = 3f r_g - r_e, \quad s_E = 3f s_g - s_e, \quad r_F = 3e r_g - r_f, \ \ \text{ and } \ \ s_F = 3e s_g - s_f.
$$
\item\label{cofactorlemmafii}
The {\em slopes} of a regular Markoff triple are strictly increasing:
$$
\frac{r_e}{e} < \frac{r_g}{g} < \frac{r_f}{f}
$$
(see also \cite[Prop. 3.1]{Rudakov89a}).
\item\label{cofactorlemmafiii}
We have  $r_e < e - r_e$, $r_f < f - r_f$, and $r_g < g - r_g$ for all $e, g$ and all $f > 2$.
\end{enumerate}
\end{lemma}

\begin{proof}
(\ref{cofactorlemmafi})
We only consider the left mutation; the right mutation follows analogously. With $F = 3eg - f$ and using the Markoff equation we get
$g^{-1} F \bmod e = - f^{-1}g \bmod e = g^{-1}f \bmod e = r_e$. Similarly,
we have $F^{-1} e \bmod g = -f^{-1}e \bmod g = e^{-1}f \bmod g = r_g$. Hence, $r_e$ and $r_g$ (and thus
$s_e$ and $s_g$) don't change under left mutation. To determine $r_F$, we write first
$$
g = 3ef - G =  3ef + f r_e - e r_f = 3e(e r_g - gr_e) + f r_e - e r_f = e(3er_g - r_f) - F r_e,
$$
where we use (\ref{cofactorequations1}) twice. On the other hand, using (\ref{cofactorequations1}) for the
triple $(e, F, g)$, we have
$$
g = e r_F - F r_e.
$$
Comparing terms then yields $r_F = 3er_g - r_f$. Similarly, the equality $s_F = 3es_g - s_f$ follows via evaluation
of $r_F^2 = -1 + F s_F$ by
direct computation, using (\ref{cofactorequations1}) and the Markoff equation.

(\ref{cofactorlemmafii})
The statement is trivially true for the triple $(1, 5, 2)$. Then we use (\ref{cofactorlemmafi}) and induction over the Markoff tree.

(\ref{cofactorlemmafiii})
The statement is obviously true for $g = 5$. By (\ref{cofactorlemmafii}) we get for left- and
right mutation:
$$
\frac{r_e}{e} < \frac{r_F}{F} < \frac{r_g}{g} < \frac{1}{2} \quad \text{ and } \quad \frac{r_g}{g} < \frac{r_E}{E} < \frac{r_f}{f} < \frac{1}{2}
$$
and by induction over the Markoff tree, the assertion follows for all $g$ and therefore for all $e > 1$ and $f > 2$.
For $e = 1$ the assertion is trivially true.
\end{proof}

\begin{remark*}
Lemma \ref{cofactorlemmaf} (\ref{cofactorlemmafi}) can also be inferred from Cohn's matrices \cite{Cohn55}
(see \cite[\S 4.3]{Aigner13}).
\end{remark*}

The slopes indeed contain some partial information about a triple's position in the Markoff tree.

\begin{corollary}\label{cofactorcorollary}\ 
\begin{enumerate}[(i)]
\item\label{cofactorcorollaryi}
Consider the $n$-th level of the Markoff tree (where the $0$-th level given by the triple $(1, 5, 2)$)
and denote $m_1, \dots, m_{2^n}$ the maximal elements of the Markoff triples on this level, enumerated
from left to right, with weights $r_{m_1}, \dots, r_{m_{2^n}}$. Then
the slopes $r_{m_i}/m_i$ are strictly increasing from left to right:
$$
0 < \frac{r_{m_1}}{m_1} < \cdots < \frac{r_{m_{2^n}}}{m_{2^n}} < \frac{1}{2}.
$$
\item\label{cofactorcorollaryii}
Let $(m_1, m_2, m_3)$ be a regular solution of the Markoff equation with $2 < m_1 < m_2 < m_3$ and set $r'_i = m_{i + 1}^{-1} m_{i + 2} \bmod m_i$
for $i = 1, 2, 3$ (where we read the indices modulo $3$). Then this solution shows up as a left mutation in the Markoff tree
(i.e. as regular Markoff triple $(m_1, m_3, m_2)$) iff one (and therefore all) of the inequalities $r'_i < m_i - r'_i$ holds.
\end{enumerate}
\end{corollary}

Figure \ref{treefig0} shows the Markoff triples, their weights and their coweights of
the first four levels of the Markoff tree.
\setlength{\tabcolsep}{1pt}
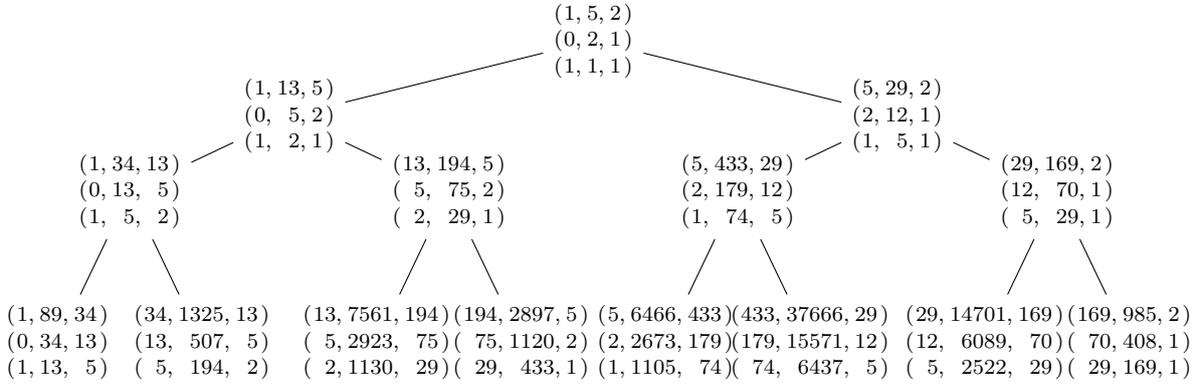
\begin{figure}[ht]
\footnotesize
\tikzstyle{level 1}=[level distance = 1cm, sibling distance=8cm]
\tikzstyle{level 2}=[sibling distance=4.2cm]
\tikzstyle{level 3}=[sibling distance=1.9cm, level distance = 2cm]
\begin{tikzpicture}[]
  \node {\begin{tabular}{r@{\hskip .5pt}rrr@{\hskip 1pt}l}( & 1, & 5, & 2 & )\\ (& 0, & 2, & 1 & )\\ ( & 1, & 1, & 1 & )\end{tabular}}
    child {node {\begin{tabular}{r@{\hskip .5pt}rrr@{\hskip 1pt}l}( & 1, & 13, & 5 & )\\ (& 0,& 5, & 2 & )\\ ( & 1, & 2, & 1 & )\end{tabular}}
      child {node {\begin{tabular}{r@{\hskip .5pt}rrr@{\hskip 1pt}l}( & 1, & 34, & 13 & )\\ (& 0, & 13, & 5 & )\\ ( & 1, & 5, & 2 & )\end{tabular}}
        child {node {\begin{tabular}{r@{\hskip .5pt}rrr@{\hskip 1pt}l}( & 1, & 89, & 34 & )\\ (& 0, & 34, & 13 & )\\ ( & 1, & 13, & 5 & )\end{tabular}}}
        child {node {\begin{tabular}{r@{\hskip .5pt}rrr@{\hskip 1pt}l}( & 34, & 1325, & 13 & )\\ (& 13, & 507, & 5 & )\\ ( & 5, & 194, & 2 & )\end{tabular}}}
      }
      child {node {\begin{tabular}{r@{\hskip .5pt}rrr@{\hskip 1pt}l}( & 13, & 194, & 5 & )\\ (& 5, & 75, & 2 & )\\ ( & 2, & 29, & 1 & )\end{tabular}}
        child {node {\begin{tabular}{r@{\hskip .5pt}rrr@{\hskip 1pt}l}( & 13, & 7561, & 194 & )\\ (& 5, & 2923, & 75 & )\\ ( & 2, & 1130, & 29 & )\end{tabular}}}
        child {node {\begin{tabular}{r@{\hskip .5pt}rrr@{\hskip 1pt}l}( & 194, & 2897, & 5 & )\\ (& 75, & 1120, & 2 & )\\ ( & 29, & 433, & 1 & )\end{tabular}}}
      }
    }
    child {node {\begin{tabular}{r@{\hskip .5pt}rrr@{\hskip 1pt}l}( & 5, & 29, & 2 & )\\ (& 2, & 12, & 1 & )\\ ( & 1, & 5, & 1 & )\end{tabular}}
      child {node {\begin{tabular}{r@{\hskip .5pt}rrr@{\hskip 1pt}l}( & 5, & 433, & 29 & )\\ (& 2, & 179, & 12 & )\\ ( & 1, & 74, & 5 & )\end{tabular}}
        child {node {\begin{tabular}{r@{\hskip .5pt}rrr@{\hskip 1pt}l}( & 5, & 6466, & 433 & )\\ (& 2, & 2673, & 179 & )\\ ( & 1, & 1105, & 74 & )\end{tabular}}}
        child {node {\begin{tabular}{r@{\hskip .5pt}rrr@{\hskip 1pt}l}( & 433, & 37666, & 29 & )\\ (& 179, & 15571, & 12 & )\\ ( & 74, & 6437, & 5 & )\end{tabular}}}
      }
      child {node {\begin{tabular}{r@{\hskip .5pt}rrr@{\hskip 1pt}l}( & 29, & 169, & 2 & )\\ (& 12, & 70, & 1 & )\\ ( & 5, & 29, & 1 & )\end{tabular}}
        child {node {\begin{tabular}{r@{\hskip .5pt}rrr@{\hskip 1pt}l}( & 29, & 14701, & 169 & )\\ (& 12, & 6089, & 70 & )\\ ( & 5, & 2522, & 29 & )\end{tabular}}}
        child {node {\begin{tabular}{r@{\hskip .5pt}rrr@{\hskip 1pt}l}( & 169, & 985, & 2 & )\\ (& 70, & 408, & 1 & )\\ ( & 29, & 169, & 1 & )\end{tabular}}}
      }
    };
\end{tikzpicture}
\caption{The first four levels of the Markoff tree with triples
$(e, g, f)$, $(r_e, r_g, r_f)$, $(s_e, s_g, s_f)$.}\label{treefig0}
\end{figure}
In the Fibonacci branch, any triple is of the form $(e, g, f) = (1, F_{2n + 1}, F_{2n - 1})$ for $n \geq 2$ and correspondingly, we have
$$
(r_e, r_g, r_f) = (0, F_{2n - 1}, F_{2n - 3}) \quad \text{ and } \quad (s_e, s_g, s_f) = (1, F_{2n - 3}, F_{2n - 5}).
$$
Similarly, in the Pell branch, any triple is of the form $(e, g, f) = (P_{2n - 1}, P_{2n + 1}, 2)$ for $n \geq 1$ with
$$
(r_e, r_g, r_f) = (P_{2n - 2}, P_{2n}, 1) \quad \text{ and } \quad (s_e, s_g, s_f) = (P_{2n - 3}, P_{2n - 1}, 1).
$$

\begin{definition}
For any real numbers $x, y \geq 2/3$ we set
\begin{enumerate}[(i)]
\item $\Delta_x := 9x^2 - 4$ 
\item $\Delta_{x, y} := \frac{1}{2}(x \sqrt{\Delta_y} + y \sqrt{\Delta_x})$.
\end{enumerate}
\end{definition}

For a regular Markoff triple $(e, g, f)$, the inequalities $2ef < g < 3ef$ hold trivially.
The following lemma sharpens these inequalities.

\begin{lemma}\label{somewhatsharpbound}
Let $(e, g, f)$ be a regular Markoff triple.
\begin{enumerate}[(i)]
\item\label{somewhatsharpboundi}
The following inequalities hold:
$$
\Delta_{e, f} - \frac{ef}{g} < g < \Delta_{e, f}.
$$
In particular,
$$
g = \lfloor\Delta_{e, f}\rfloor.
$$
\item\label{somewhatsharpboundii}
The first inequality can be sharpened to:
$$
\Delta_{e, f} - \frac{2}{3g} < \Delta_{e, f} - \frac{2}{9ef} < g.
$$
\item\label{somewhatsharpboundiii}
The second to:
$$
g + \frac{2ef}{g \sqrt{\Delta_e} \sqrt{\Delta_f}} < \Delta_{e, f}.
$$
\end{enumerate}
\end{lemma}

\begin{proof}
(\ref{somewhatsharpboundi})
The last assertion follows from the inequalities because $ef/g < 1$.

Squaring both sides and rearranging terms, the second inequality is equivalent to
$$
2e^2 + 2f^2 + 2g^2 - 9e^2f^2 < ef \sqrt{9e^2 - 4} \sqrt{9f^2 - 4}.
$$
Using the Markoff equation and dividing by $ef$ this is equivalent to:
$$
3 (2g - 3ef) < \sqrt{9e^2 - 4} \sqrt{9f^2 - 4}.
$$
Again squaring both sides, we obtain equivalently:
$$
9(2g - 3ef)^2 < (9e^2 - 4)(9f^2 - 4),
$$
which, using again the Markoff equation, is trivially verified.

For the first inequality, squaring both sides and using the Markoff equation leads to the expression
$$
9(2g - 3ef)^2 + 12(2g - 3ef)(2 + \frac{ef}{g}) + 4(2 + \frac{ef}{g})^2 < (9e^2 - 4)(9f^2 - 4),
$$
which, by using the Markoff equation, reduces to:
$$
4 < 3(2g - 3ef)(2 + \frac{ef}{g}) + (2 + \frac{ef}{g})^2.
$$
This inequality follows from $\frac{ef}{g} > 0$ and $2g - 3ef = g - G > 1$.

(\ref{somewhatsharpboundii})
We need only to show the second inequality. For this, it suffices to show the squared inequality
$\Delta_{e, f}^2 < (g + 2 / 9ef)^2$, for which we get:
$$
ef \sqrt{\Delta_e} \sqrt{\Delta_f} < \left(g + \frac{2}{9ef}\right)^2 - \frac{1}{2}(e^2 \Delta_f + f^2 \Delta_e).
$$
Squaring again and plugging in the following formula which can easily be derived from the Markoff equation:
$$
9e^2f^2 - 4(e^2 + f^2) = (g - G)^2,
$$
we get:
$$
e^2f^2(9(g - G)^2 + 16) < \left(2(g + \frac{2}{9ef})^2 - (g - G)^2 - 2(e^2 + f^2)\right)^2.
$$
Now we divide both sides by $e^2f^2(g - G)$ and first observe for the l.h.s.:
$$
9 + \frac{16}{e^2f^2(g - G)} \leq 10
$$
for every Markoff number $\geq 5$.
For the r.h.s., we get by using the Markoff equation and rearranging terms:
\begin{align*}
\frac{1}{e^2f^2(g - G)} \left(2(g + \frac{2}{9ef})^2 - (g - G)^2 - 2(e^2 + f^2)\right)^2
	& = \frac{1}{e^2f^2(g - G)} \left(3ef(g - G) + R\right)^2 \\
	& = 9ef(g - G) + \frac{6R}{ef} + \frac{R^2}{e^2f^2(g - G)} \\
	& > 10
\end{align*}
for every regular Markoff triple, where we denote $R := 8g / 9ef + 8/81e^2f^2 > 0$.
So, the second inequality follows.

(\ref{somewhatsharpboundiii})
We show that the inequality holds for all positive real solutions $(e, g, f) \in \R^3$ of the Markoff equation.
By considering the Markoff equation a quadratic equation in $g$, we get the following expression:
$$
g = \frac{ef}{2}\left(3 \pm \sqrt{9 - \frac{4}{e^2} - \frac{4}{f^2}}\right).
$$
The boundary of the projection of the positive solutions of the Markoff equation to the $e$-$f$-plane
coincides with the branching points for this equation given by the curve $9e^2f^2 = 4(e^2 + f^2)$.
With
$$
\inf_{e \to \infty} \sqrt{\frac{4e^2}{9e^2 - 4}} = \inf_{f \to \infty} \sqrt{\frac{4f^2}{9f^2 - 4}} = \frac{2}{3}
$$
we see that both functions
$$
g + \frac{2ef}{g \sqrt{\Delta_e} \sqrt{\Delta_f}} = \frac{ef}{2}\left(3 + \sqrt{9 - \frac{4}{e^2} - \frac{4}{f^2}}\right)
+ 4 / \left( (3 + \sqrt{9 - \frac{4}{e^2} - \frac{4}{f^2}}) \sqrt{\Delta_e} \sqrt{\Delta_f} \right)
$$ and $\Delta_{e, f}$
are well-defined functions over the projection of the positive component of the Markoff variety to the
$e$-$f$-plane (note that without loss
of generality we choose the bigger solutions for $g$).

We first consider the case $e = f$. Then the formulas simplify to:
$$
g + \frac{2e^2}{g \Delta_e} < e \sqrt{\Delta_e}
\quad \text{ and } \quad
g = \frac{e}{2}\left(3e + \sqrt{9 e^2 - 8}\right),
$$
where we claim that the inequality holds for $e \geq \frac{2\sqrt{2}}{3}$.
Eliminating $g$, we get:
$$
\frac{\left(3e + \sqrt{9 e^2 - 8}\right)}{2 \sqrt{\Delta_e}} + \frac{4}{\Delta_e^{3/2} (3e + \sqrt{9e^2 - 8})} < 1.
$$
As $\underset{e \rightarrow \infty}{\lim}(\text{l.h.s.}) = 1$, it suffices to show that the l.h.s is
strictly monotonously increasing for $e \geq \frac{2\sqrt{2}}{3}$. This means that for the first derivative of the l.h.s.,
after rearranging terms and clearing some denominators, it suffices to show:
$$
2(3e - \sqrt{9e^2-8}) > \frac{1}{3e + \sqrt{9e^2 - 8}}\left(\frac{9e\sqrt{9e^2 - 8}}{\Delta_e} + 1\right)
$$
for $e \geq 2\sqrt{2}/3$, which can straightforwardly be verified.

For the general case, we fix $C := 9 - \frac{4}{e^2} - \frac{4}{f^2}$, where we note that $0 \leq C < 9$.
Then, after eliminating $g$ and
$f^2 = 4(9 - C - \frac{4}{e^2})$ and some rearrangements, the inequality takes the following form:
$$
\sqrt{9 - \frac{4}{e^2}}\sqrt{C + \frac{4}{e^2}}(3 + \sqrt{C})\left(\sqrt{9 - \frac{4}{e^2}} + \sqrt{C + \frac{4}{e^2}}
- \sqrt{C} - 3\right) - \frac{2}{e^2}(9 - C - \frac{4}{e^2}) > 0
$$
It is straightforward to verify that the l.h.s. of this inequality has four local extrema at $e^2 = \frac{4}{9 - C}$
and $e^2 = \frac{8}{9 - C}$ within its domain of definition, and moreover
$\underset{e \rightarrow \infty}{\lim}(\text{l.h.s.}) = 0$.
As $e^2 = \frac{8}{9 - C}$ iff $e^2 = f^2$, we have seen before that the inequality holds in that case and therefore
we have maxima at $e^2 = \frac{8}{9 - C}$. For degree reasons it follows that there are no further local extrema and
hence the inequality holds for all $e^2 \geq \frac{8}{9 - C}$ and we are done.
\end{proof}

Consider a subtree of the Markoff tree whose root is a regular Markoff triple $(e, g, f)$. Then the two outer
branches of this tree consist of triples of the form $(e, F_i, F_{i - 1})$ and $(E_i, E_{i - 1}, f)$ with
$E_{i + 1} = 3f E_i - E_{i - 1}$ and $F_{i + 1} = 3e F_i - F_{i - 1}$ for $i \geq 0$, where we set $E_{-1} = e$,
$F_{-1} = f$, and $F_0 = E_0 = g$. Standard methods for solving recurrences of the form $x_{i + 1} = 3n x_i - x_{i - 1}$
for $i \geq 0$ yield:
$$
x_i = \lambda_+ \left(\frac{3n + \sqrt{\Delta_n}}{2}\right)^i + \lambda_- \left(\frac{3n - \sqrt{\Delta_n}}{2}\right)^i
$$
with
$$
\lambda_\pm = \frac{1}{2} \left(x_0 \pm \frac{x_1 - x_{-1}}{\sqrt{\Delta_n}} \right).
$$
Using Lemma \ref{cofactorlemmaf} (\ref{cofactorlemmafi}), we get the following formulas (compare also \cite[\S 10.2]{Aigner13})
from which we can see that the corresponding weights and coweights follow the same growth as the Markoff numbers.

\begin{lemma}\label{growthlemma}
Denote $P_\pm := (3f \pm \sqrt{\Delta_f})/2$ and $Q_\pm := (3e \pm \sqrt{\Delta_e})/2$. Then for $i \geq 0$ we have:
\begin{enumerate}[(i)]
\item\label{growthlemmai}
\begin{align*}
E_i & = \lambda_+ P_+^i + \lambda_- P_-^i,
\quad \text{where} \quad \lambda_\pm = \frac{1}{2} \left(g \pm \frac{E - e}{\sqrt{\Delta_f}} \right),\\
F_i & = \kappa_+ Q_+^i + \kappa_- Q_-^i,
\quad \text{where} \quad  \kappa_\pm = \frac{1}{2} \left(g \pm \frac{F - f}{\sqrt{\Delta_e}} \right).
\end{align*}
\item\label{growthlemmaii}
\begin{align*}
r_{E_i} & = \lambda_+ P_+^i + \lambda_- P_-^i,
\quad \text{where} \quad \lambda_\pm = \frac{1}{2} \left(r_g \pm \frac{r_E - r_e}{\sqrt{\Delta_f}} \right),\\
r_{F_i} & = \kappa_+ Q_+^i + \kappa_- Q_-^i,
\quad \text{where} \quad  \kappa_\pm = \frac{1}{2} \left(r_g \pm \frac{r_F - r_f}{\sqrt{\Delta_e}} \right).
\end{align*}
\item\label{growthlemmaiii}
\begin{align*}
s_{E_i} & = \lambda_+ P_+^i + \lambda_- P_-^i,
\quad \text{where} \quad \lambda_\pm = \frac{1}{2} \left(s_g \pm \frac{s_E - s_e}{\sqrt{\Delta_f}} \right),\\
s_{F_i} & = \kappa_+ Q_+^i + \kappa_- Q_-^i,
\quad \text{where} \quad  \kappa_\pm = \frac{1}{2} \left(s_g \pm \frac{s_F - s_f}{\sqrt{\Delta_e}} \right).
\end{align*}
\end{enumerate}
\end{lemma}

\subsection{Continued fractions}

Recall the following construction of the {\em Stern-Brocot tree}. We start with a binary tree that contains triples of the form
$(\frac{\lambda}{\lambda'}, \frac{\nu}{\nu'}, \frac{\mu}{\mu'})$ with $\lambda, \lambda', \mu,$ $\mu', \nu, \nu' \in \Z$,
where for the moment we formally consider the fractions as pairs of integers rather than
rational numbers. Its root is $(\frac{1}{0}, \frac{1}{1}, \frac{0}{1})$ and descendants are formed by the rule
\begin{center}
\tikzstyle{level 1}=[level distance = 1cm, sibling distance=4cm]
\begin{tikzpicture}[]
  \node {$(\frac{\lambda}{\lambda'}, \frac{\nu}{\nu'}, \frac{\mu}{\mu'})$}
    child {node {$(\frac{\lambda}{\lambda'}, \frac{\lambda + \nu}{\lambda' + \nu'}, \frac{\nu}{\nu'})$}}
    child {node {$(\frac{\nu}{\nu'}, \frac{\mu + \nu}{\mu' + \nu'}, \frac{\mu}{\mu'})$}};
\end{tikzpicture}
\end{center}
The middle entry of each triple is a pair of coprime positive integers and in particular
a well-defined rational number. The {\em Stern-Brocot tree} now is the binary tree consisting of these
middle entries, containing every positive rational number precisely once.
Note that by abuse of notion, we refer to the tree formed by above triples as Stern-Brocot tree as well.
Now, starting with
$$
(m_{\frac{1}{0}}, m_{\frac{1}{1}}, m_{\frac{0}{1}}) := (1, 5, 2)
$$
and extending to a bijection of trees,
we define a bijection between triples $(\frac{\lambda}{\lambda'}, \frac{\nu}{\nu'}, \frac{\mu}{\mu'})$ and regular Markoff triples
$(m_{\frac{\lambda}{\lambda'}}, m_{\frac{\nu}{\nu'}}, m_{\frac{\mu}{\mu'}})$. In particular,
we get a surjective map $\frac{\mu}{\nu} \mapsto m_{\frac{\mu}{\nu}}$
from $\Q_{\geq 0} \cup \{\frac{1}{0}\}$ to the set of Markoff numbers (one version of the uniqueness conjecture
states that this map is a bijection). This indexing of Markoff numbers has been introduced by Frobenius \cite{Frobenius}.

Now it is a natural question whether a positive rational number $\mu/\nu$ enough to construct the corresponding
Markoff triple (and vice versa) without explicitly referring to the binary trees?
To see that this is indeed possible, we start with the following observation. If we consider any triple of the form
$(\frac{\lambda}{\lambda'}, \frac{\nu}{\nu'}, \frac{\mu}{\mu'})$ and write the fractions as lattice vectors $(\lambda, \lambda'),
(\mu, \mu'), (\nu, \nu') \in \Z^2$, then by construction, every two of these lattice vectors form a basis of $\Z^2$.
Conversely, for any primitive lattice vector $(\mu, \nu) \in \Z^2$ it is easy to see that there exist precisely two
lattice vectors $(\mu_1, \nu_1), (\mu_2, \nu_2) \in [0, \mu] \times [0, \nu] \subset \Z^2$ such that $(\mu, \nu), (\mu_1, \nu_1)$
and $(\mu, \nu), (\mu_2, \nu_2)$ form bases of $\Z$, respectively. Also it is immediately clear that $(\mu, \nu) =
(\mu_1, \nu_1) + (\mu_2, \nu_2)$ and hence $(\mu_1, \nu_1), (\mu_2, \nu_2)$ form a basis as well. It follows that
either $(\mu_1/\nu_1, \mu/\nu, \mu_2/\nu_2)$ or $(\mu_2/\nu_2, \mu/\nu, \mu_1/\nu_1)$ represents a vertex in the
Stern-Brocot tree.

Now, consider any fraction $\mu / \nu$ with corresponding Stern-Brocot triple $(\mu_1 / \nu_1, \mu / \nu, \mu_2 / \nu_2)$.
We denote $(m_{\mu_1 / \nu_1}, m_{\mu/\nu}, m_{\mu_2 / \nu_2})$ the corresponding Markoff triple and
$(r_{\mu_1 / \nu_1}, r_{\mu/\nu}, r_{\mu_2 / \nu_2})$, $(s_{\mu_1 / \nu_1}, s_{\mu/\nu}, s_{\mu_2 / \nu_2})$ the associated
weights and coweights. Frobenius has determined the continued fraction expansion of $m_{\mu/\nu}/r_{\mu/\nu}$ as follows
(regarding the definition of length of a continued fraction, please check the discussion right after Formula
\ref{Cassini2} in Appendix \ref{continuants}).

\begin{theorem}[\cite{Frobenius}, \S's 10, 11, \cite{CusickFlahive}, Ch. 2]\label{frobeniustheorem}
With above notation, assume that $m_{\mu/\nu} > 2$. Then the continued fraction expansion of
$\frac{m_{\mu/\nu}}{r_{\mu/\nu}}$ has the following properties.
\begin{enumerate}[(i)]
\item\label{frobeniustheoremi}
If we write $\frac{m_{\mu/\nu}}{r_{\mu/\nu}} = [a_1, \dots, a_n]$ with $a_1, a_n > 1$ then $n = 2s$ is even.
\item\label{frobeniustheoremii}
The continued fraction expansion $\frac{m_{\mu/\nu}}{r_{\mu/\nu}} = [a_1, \dots, a_{2s}]$ is symmetric,
i.e  $a_i = a_{2s + 1 - i}$ for every $1 \leq i < 2s$.
\item\label{frobeniustheoremiii}
The continued fraction expansion of $\frac{m_{\mu/\nu}}{r_{\mu/\nu}}$ begins and ends in $2$, i.e.
$$
\frac{m_{\mu/\nu}}{r_{\mu/\nu}} = [2, S(\mu, \nu), 2].
$$
\item\label{frobeniustheoremiv}
For the case $\nu = 1$ with $\frac{m_{\mu/1}}{r_{\mu/1}} = \frac{F_{2\mu + 3}}{F_{2\mu + 1}}$ we have:
$$
S(\mu, 1) = 1_{2\mu - 2},
$$
where $1_{2\mu - 2}$ denotes the $2\mu-2$-fold repetition of $1$.
\item\label{frobeniustheoremv}
For $\nu > 1$ and $1 \leq i < \nu$ we set
$$
\kappa(i) := \lfloor i\frac{\mu}{\nu} \rfloor - \lfloor (i - 1)\frac{\mu}{\nu} \rfloor.
$$
Then:
$$
S(\mu, \nu) = 1_{2\kappa(1)}, 2, 2, 1_{2\kappa(2)}, 2, 2, \dots, 2, 2, 1_{2\kappa(\nu - 1)}, 2, 2, 1_{2\kappa(1)}.
$$
\item\label{frobeniustheoremvi}
Consider a Markoff triple $(m_{\mu_1/\nu_1}, m_{\mu/\nu}, m_{\mu_2/\nu_2})$ with $\mu, \nu > 1$. Then
$$
S(\mu, \nu) = S(\mu_1, \nu_1), 2, 2, 1, 1, S(\mu_2, \nu_2) = S(\mu_2, \nu_2), 1, 1, 2, 2, S(\mu_1, \nu_1).
$$
\end{enumerate}
\end{theorem}

\begin{example}
For $\mu/\nu = 8/5$ we get $\kappa(1) = \kappa(3) = 1$ and $\kappa(2) = \kappa(4) = 2$, so
$$
\frac{m_{8/5}}{r_{8/5}} = [2, 1, 1, 2, 2, 1, 1, 1, 1, 2, 2, 1, 1, 2, 2, 1, 1, 1, 1, 2, 2, 1, 1, 2] = \frac{4400489}{1701181},
$$
and thus $s_{8/5} = 657658$. The vector $(8, 5)$ splits into $(3, 2) + (5, 3)$ such that every two of $(8, 5)$, $(3, 2)$, $(5, 3)$
form a basis of $\Z^2$.
For $(5, 3)$ we get $\kappa(1) = 1$, $\kappa(2) = 2$, hence
$$
\frac{m_{5/3}}{r_{5/3}} = [2, 1, 1, 2, 2, 1, 1, 1, 1, 2, 2, 1, 1, 2] = \frac{7561}{2923}
$$
with $s_{3/2} = 1130$.
For $(3, 2)$, we get $\kappa(1) = 1$, hence
$$
\frac{m_{3/2}}{r_{3/2}} = [2, 1, 1, 2, 2, 1, 1, 2] = \frac{194}{75}
$$
with $s_{3/2} = 29$. 
This data corresponds to the Markoff triple
$$
(7561, 4400489, 194).
$$
\end{example}

We observe that we can partition the continued fraction expansion of $m_{\mu/\nu} / r_{\mu/\nu}$ into segments
of the form $2, 1_{2\kappa(i)}, 2$ and
$$
[2, 1_{2\kappa(i)}, 2] = \frac{F_{2\kappa(i) + 5}}{F_{2\kappa(i) + 3}},
$$
so, in a sense, the continued fraction expansions yield a decomposition of a Markoff number into Fibonacci numbers. For
sake of exposition, we choose to call a sequence of the form $2, 1_k, 2$ for $k \geq 0$ a {\em Fibonacci segment}.
The following statement then follows by induction and $\sum_{i = 1}^{\nu - 1} \kappa(i) = \mu - 1$.

\begin{corollary}\label{frobeniuscorollary}\ 
\begin{enumerate}[(i)]
\item The length of the continued fraction expansion of $\frac{m_{\mu/\nu}}{r_{\mu/\nu}}$ is $2(\mu + \nu - 1)$.
\item The continued fraction expansion of $\frac{m_{\mu/\nu}}{r_{\mu/\nu}}$ consists of $\nu$ Fibonacci segments.
In particular, it contains $2(\mu - 1)$ ones and $2\nu$ twos.
\end{enumerate}
\end{corollary}

It follows that just from the pair $(m_{\mu/\nu}, r_{\mu/\nu})$ (or even just $\mu/ \nu$) we can reconstruct
the complete Markoff triple.
It has already been pointed out by Frobenius \cite[\S 11]{Frobenius}, the composition of the continued fractions of
$\frac{m_{\mu/\nu}}{r_{\mu/\nu}}$ can be interpreted in terms of Christoffel words and their standard factorizations.
We want to add the observation that this also allows another very nice graphical interpretation. For this, consider
the set $[0, \mu + \nu] \times [0, \nu] \subset \Z^2$ as the set of $(\mu + \nu + 1) (\nu + 1)$ corner points of a collection of
$1 \times 1$-boxes
in the plane. Then there are precisely $(\mu - 1) + 2\nu$ boxes that touch the diagonal from $(0, 0)$ to $(\mu + \nu, \nu)$.
Now we fill the first and the last box, as well as every corner box with $2$'s, and the remaining boxes with pairs $1, 1$.
Then, if we read off the numbers from left to right, we obtain the continued fraction expansion of $\frac{m_{\mu/\nu}}{r_{\mu/\nu}}$.
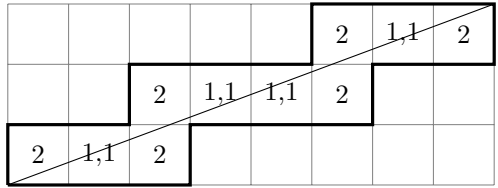
\begin{figure}[ht]
\begin{tikzpicture}[]
\draw[step=0.8cm,gray,very thin] (0, 0) grid (6.4,2.4);
\draw (0, 0) -- (6.4, 2.4);
\draw[very thick] (0, 0) -- (2.4, 0) -- (2.4, 0.8) -- (4.8, 0.8) -- (4.8, 1.6) -- (6.4, 1.6) --
(6.4, 2.4) -- (4.0, 2.4) -- (4.0, 1.6) -- (1.6, 1.6) -- (1.6, 0.8) -- (0, 0.8) -- (0, 0);
\node at (0.4, 0.4) {2};
\node at (1.2, 0.4) {1,1};
\node at (2.0, 0.4) {2};
\node at (2.0, 1.2) {2};
\node at (2.8, 1.2) {1,1};
\node at (3.6, 1.2) {1,1};
\node at (4.4, 1.2) {2};
\node at (4.4, 2.0) {2};
\node at (5.2, 2.0) {1,1};
\node at (6.0, 2.0) {2};
\end{tikzpicture}
\caption{The snake diagram and continued fraction expansion of $\frac{m_{5/3}}{r_{5/3}} = \frac{7561}{2923}$.}\label{snakefigure1}
\end{figure}
Also, every row in this this diagram contains precisely one Fibonacci segment.
Figure \ref{snakefigure1} shows the corresponding ``snake diagram'' for $(\mu, \nu) = (5, 3)$. For more details on Christoffel
words and snake diagrams we refer to \cite[\S 7]{Aigner13}.

\begin{remark*}
Note that in order to reconstruct a Markoff triple from a pair $(m_{\mu/\nu}, r_{\mu/\nu})$, we
need the Euclidean algorithm to compute $\mu_1 / \nu_1$ and $\mu_2 / \nu_2$. Therefore, while
we do not have to deal explictly with mutations, implicitly we still are required to backtrack the full path
to $(m_{\mu/\nu}, r_{\mu/\nu})$ via the Stern-Brocot tree.
\end{remark*}

\begin{sub}\label{complement}
A rational number $\mu/\nu > 0$ and its inverse $\nu / \mu$ correspond to opposite positions in the Stern-Brocot tree
with respect to reflection along the middle axis. So given a pair $(m, r)$ consisting of a Markoff number and
a weight $r$, we have seen above that we can determine $(\mu, \nu)$ such that $(m, r) = (m_{\mu/\nu}, r_{\mu/\nu})$ together
with the Markoff triple which gives rise to $(m, r)$. Moreover, once given $(\mu, \nu)$, we can also compute the opposite
Markoff triple corresponding to $(m_{\nu/\mu}, r_{\nu/\mu})$. This is nicely reflected in the corresponding continued fraction
expansions as follows. Both sequences $S(\mu, \nu) = x_1, \dots, x_{2(\mu + \nu - 2)}$ and $S(\nu, \mu) = y_1, \dots, y_{2(\mu + \nu - 2)}$
have the same length by Corollary \ref{frobeniuscorollary} and moreover, as has been observed by Frobenius \cite[p. 620]{Frobenius},
$x_i + y_i = 3$ for every $1 \leq i \leq 2(\mu + \nu - 2)$. In other words, we
obtain $S(\nu, \mu)$ from $S(\mu, \nu)$ by changing $2$'s to $1$'s and vice versa.
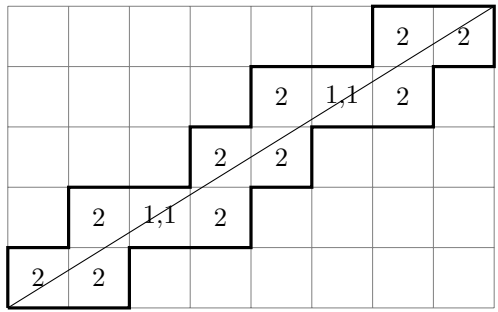
\begin{figure}[ht]
\begin{tikzpicture}[]
\draw[step=0.8cm,gray,very thin] (0, 0) grid (6.4,4.0);
\draw (0, 0) -- (6.4, 4.0);
\draw[very thick] (0, 0) -- (1.6, 0) -- (1.6, 0.8) -- (3.2, 0.8) -- (3.2, 1.6) -- (4.0, 1.6) --
(4.0, 2.4) -- (5.6, 2.4) -- (5.6, 3.2) -- (6.4, 3.2) -- (6.4, 4.0) -- (4.8, 4.0) -- (4.8, 3.2) --
(3.2, 3.2) -- (3.2, 2.4) -- (2.4, 2.4) -- (2.4, 1.6) -- (0.8, 1.6) -- (0.8, 0.8) -- (0, 0.8) -- (0, 0);
\node at (0.4, 0.4) {2};
\node at (1.2, 0.4) {2};
\node at (1.2, 1.2) {2};
\node at (2.0, 1.2) {1,1};
\node at (2.8, 1.2) {2};
\node at (2.8, 2.0) {2};
\node at (3.6, 2.0) {2};
\node at (3.6, 2.8) {2};
\node at (4.4, 2.8) {1,1};
\node at (5.2, 2.8) {2};
\node at (5.2, 3.6) {2};
\node at (6.0, 3.6) {2};
\end{tikzpicture}
\caption{The snake diagram and continued fraction expansion of $\frac{m_{3/5}}{r_{3/5}} = \frac{37666}{15571}$.}\label{snakefigure2}
\end{figure}
Figure \ref{snakefigure2} shows the case opposite of $(\mu, \nu) = (5, 3)$ from Figure \ref{snakefigure1}.
\end{sub}

\section{Cantor set}\label{cantor}

Consider the $n$-th level of the Markoff tree, where $n > 1$ and denote $m_1, \dots, m_{2^n}$ the maximal elements of the Markoff triples
in this level, enumerated from left to right. By Corollary \ref{cofactorcorollary} (\ref{cofactorcorollaryi}), we have
$$
\frac{F_{2n + 3}}{F_{2n + 1}} = \frac{m_1}{r_{m_1}} > \cdots > \frac{m_{2^n}}{r_{m_{2^n}}} = \frac{P_{2n + 1}}{P_{2n - 1}}.
$$
Moreover, both $F_{2n + 3}/F_{2n + 1}$ and $P_{2n + 1}/P_{2n - 1}$ are monotonous, bounded sequences which converge from below
and above, respectively, to their respective limits. Using
\begin{align*}
\lim_{i \to \infty} \frac{F_{2i + 1}}{F_{2i - 1}} & = \lim_{i \to \infty} [2, 1_{2(i - 1)}, 2]  = \lim_{i \to \infty} [2, 1_{2i}] = [2, \bar{1}] =
\frac{1}{2}(3 + \sqrt{5}), \\
\lim_{i \to \infty} \frac{P_{2i + 1}}{P_{2i - 1}} & = \lim_{i \to \infty} [2_{2i}] = [\bar{2}] = 1 + \sqrt{2},
\end{align*}
we see that the spectrum of fractions $m/r$ for Markoff numbers $m \geq 5$ is bounded:
\begin{equation}\label{cantorbound}
1 + \sqrt{2} < \frac{m}{r} < \frac{1}{2}(3 + \sqrt{5}).
\end{equation}
We can consider the Fibonacci- and Pell branch as particular infinite paths in the Markoff tree and above limits
as limits over these paths. More generally, we can encode every vertex of a binary tree by its path starting from
the root. We may write every such path as a finite word in the alphabet $\{L, R\}$, where the empty word stands
for the root and $L$ and $R$ designate left- and right mutation, respectively. Similarly, an infinite word in this
alphabet denotes an infinite path which we can consider as the limit over its finite starting segments in the
obvious sense.

\begin{definition}
For any $X \in \{L, R\}$ we denote by $X^n$ $n$-fold repetition for any $n > 0$ and by $\bar{X}$ we denote infinite repetition.
\end{definition}

In particular, words of the form $L^n$ and $R^n$ for $n \geq 0$ describe the entries in the Fibonacci- and Pell
branches, and $\bar{L} = \lim_{n \to \infty} L^n$, $\bar{R} = \lim_{n \to \infty} R^n$, respectively.

\begin{definition}
We denote by $\mathfrak{P}_f$ the set of finite paths, by $\mathfrak{P}$ the set of infinite paths, and by
$\bar{\mathfrak{P}}_f \subset \mathfrak{P}$ paths which are of the form $p\bar{X}$ for $p \in \mathfrak{P}_f$ and $X \in \{L, R\}$.
\end{definition}

Clearly, both $\mathfrak{P}_f$, $\bar{\mathfrak{P}}_f$ are countable; in particular, mapping a path $X_1 X_2 \cdots X_n$ with
$n > 0$ to $X_1 X_2 \dots X_{n - 1} \bar{X}_n$ is a surjection from $\mathfrak{P}_f \setminus \{\emptyset\}$ to $\bar{\mathfrak{P}}_f$.
The set $\mathfrak{P}$ has the same cardinality as the real numbers. Now let $p = \lim_{n \to \infty} p_n \in \mathfrak{P}$,
where $p_n$ denotes the starting segment of length $n$ of $p$ which also corresponds to a rational number $\mu_n / \nu_n$ in
the Stern-Brocot tree.

\begin{proposition}\label{limitsexist}
With above notation, let $p = \lim_{n \to \infty} p_n \in \mathfrak{P}$. Then the limit
$\lim_{n \to \infty} m_{\mu_n/\nu_n}/r_{\mu_n/\nu_n}$ $= \mathfrak{l}_p$ exists. In particular, every $m_{\mu_n/\nu_n}/r_{\mu_n/\nu_n}$ is
a rational convergent of the continued fraction expansion of $\mathfrak{l}_p$.
\end{proposition}

\begin{proof}
As we have considered the Fibonacci and Pell branches already above, it suffices to
consider any Stern-Brocot triple $(\mu_1 / \nu_1, \mu / \nu, \mu_2 / \nu_2)$ with $\mu, \nu > 1$.
Then by Theorem \ref{frobeniustheorem} we have
\begin{align*}
[2, S(\mu, \nu), 2] = [2, S(\mu_1, \nu_1), 2, 2, 1, 1, S(\mu_2, \nu_2), 2] = [2, S(\mu_2, \nu_2), 1, 1, 2, 2, S(\mu_1, \nu_1), 2],
\end{align*}
so we can view the continued fraction expansion of $m_{\mu/\nu}/r_{\mu/\nu}$ as continuation of its predecessors'
(note that $[2, S(\mu_2, \nu_2), 2] = [2, S(\mu_2, \nu_2), 1, 1]$). In particular, we can consider mutation as a process
of extending continued fractions to the right:
\begin{center}
\footnotesize
\tikzstyle{level 1}=[level distance = 1.3cm, sibling distance=8cm]
\begin{tikzpicture}[]
  \node[align=center] {$[2, S(\mu_1, \nu_1), 2]$, $[2, S(\mu, \nu), 2]$, $[2, S(\mu_2, \nu_2), 2]$}
    child {node[align=center] at (1, 0.5) {$[2, S(\mu_1, \nu_1), 2]$, $[2, S(\mu_1, \nu_1), 2, 2, 1, 1, S(\mu, \nu), 2]$, $[2, S(\mu, \nu), 2]$}}
    child {node[align=center]  at (-1, 0) {$[2, S(\mu, \nu), 2]$, $[2, S(\mu, \nu), 2, 2, 1, 1, S(\mu_2, \nu_2), 2]$, $[2, S(\mu_2, \nu_2), 2]$}};
\end{tikzpicture}
\end{center}
where:
\begin{align*}
[2, S(\mu_1, \nu_1), 2, 2, 1, 1, S(\mu, \nu), 2] & = [2, S(\mu_1, \nu_1), 2, 2, 1, 1, S(\mu_1, \nu_1), 2, 2, 1, 1, S(\mu_2, \nu_2), 2]\\
& = [2, S(\mu, \nu), 1, 1, 2, 2, S(\mu_1, \nu_1), 2].
\end{align*}
It follows that the limit infinite continued fraction exists for any path down the Markoff tree, hence the assertions are shown.
\end{proof}

It is easy to see that the $\mathfrak{l}_p$ are contained in the closed interval $[1 + \sqrt{2} ,\frac{1}{2}(3 + \sqrt{5})]$
and we are going to show that they form a Cantor set. Before we can do this, we consider one more special case.

\begin{proposition}\label{firstoffprop}
\begin{enumerate}[(i)]
\item\label{firstoffpropi}
Let $p = L^n\bar{R}$, then
$$\mathfrak{l}_p = [\overline{2, 1_{2n}, 2}] = \frac{1}{2}\frac{3F_{2n + 3} - 2F_{2n + 1} + \sqrt{9 F_{2n + 3}^2 - 4}}{3F_{2n + 1} - F_{2n - 1}}.$$
\item\label{firstoffpropii}
Let $p = R^n\bar{L}$, then
$$\mathfrak{l}_p = 1 + [\overline{1, 2_{2n}, 1}] = \frac{1}{2}\frac{3P_{2n + 1} + 2P_{2n - 1} + \sqrt{9 P_{2n + 1}^2 - 4}}{3 P_{2n - 1} + P_{2n - 3}}.$$
\end{enumerate}
\end{proposition}

\begin{proof}
(\ref{firstoffpropi})
The Markoff triple in position $L^n$ is of the form $(1, F_{2n + 5}, F_{2n + 3})$ and the continued fraction expansions
are $(-, [2, 1_{2n}, 2], [2, 1_{2n - 2}, 2])$. Using Theorem \ref{frobeniustheorem} (\ref{frobeniustheoremvi}), a straightforward
induction yields that the continued fractions at position $L^nR^m$ for $m > 0$ are
$$
([\langle 2, 1_{2n}, 2 \rangle_m],[\langle 2, 1_{2n}, 2 \rangle_{m + 1}],[2, 1_{2n - 2}, 2]),
$$
where $\langle \cdots \rangle_k$ denotes $k$-fold repetition. Hence
$$
\mathfrak{l}_p = \lim_{m \to \infty} [\langle 2, 1_{2n}, 2 \rangle_m] = [\overline{2, 1_{2n}, 2}].
$$

(\ref{firstoffpropii})
The Markoff triple in position $R^n$ is of the form $(P_{2n + 1}, P_{2n + 3}, 2)$ with continued fraction expansions
$([2_{2n}], [2_{2n + 2}], [2])$. Again, using Theorem \ref{frobeniustheorem} (\ref{frobeniustheoremvi}) and induction,
we get for the continued fractions at position $R^nL^m$ for $m > 0$:
$$
([2_{2n}],[2, \langle 2_{2n}, 1, 1 \rangle_m, 2_{2n + 1}], [2, \langle 2_{2n}, 1, 1 \rangle_{m - 1}, 2_{2n + 1}]),
$$
hence
$$
\mathfrak{l}_p = \lim_{m \to \infty} [2, \langle 2_{2n}, 1, 1 \rangle_m] = [2, \overline{2_{2n}, 1, 1}] = 1 + [\overline{1, 2_{2n}, 1}].
$$

In both cases, using \ref{makeperiodic}, the translation of the periodic continued fractions into quadratic irrational
numbers follows by straightforward computation (see also the proof of Proposition \ref{quadraticspectrum1} below).
\end{proof}

Given a triple $(m_{\mu_1/\nu_1}, m_{\mu/\nu}, m_{\mu_2/\nu_2})$ corresponding to a finite path $p$, then, as indicated in
Figure \ref{chains}, the set of triples below this triple wich contain $m_{\mu/\nu}$ are precisely those in the paths
$pL\bar{R}$ and $pR\bar{L}$.

\begin{proposition}\label{quadraticspectrum1}
Consider a Markoff triple $(e, g, f)$ with $e > 1$, $f > 2$ and denote $(\mu_1/\nu_1, \mu/\nu, \mu_2/\nu_2)$ and $p$ the corresponding Stern-Brocot
triple and finite path, respectively. Moreover, denote $p_1 = pL\bar{R}$, $p_2 = pR\bar{L}$. Then
\begin{align*}
\mathfrak{l}_{p_1} = [\overline{2, S(\mu, \nu), 1, 1, 2}] & = \frac{1}{2} \frac{3g - 2r_g + \sqrt{9g^2 - 4}}{3r_g - s_g}, \\
\mathfrak{l}_{p_2} = 1 + [\overline{1, S(\mu, \nu), 2, 2, 1}] & = \frac{1}{2} \frac{3g + 2r_g + \sqrt{9g^2 - 4}}{3r_g + s_g}.
\end{align*}
\end{proposition}

\begin{proof}
We denote $$[2, S(\mu_1, \nu_1), 2], [2, S(\mu, \nu), 2], [2, S(\mu_2, \nu_2), 2]$$ the continued fraction expansions of $e/r_e, g/r_g, f/r_f$.
Then it follows from Theorem \ref{frobeniustheorem} (\ref{frobeniustheoremvi}) that after one left- and one right mutation, we obtain
$$
[2, S(\mu_1, \nu_1), 2, 2, 1, 1, S(\mu, \nu), 2], [2, S(\mu_1, \nu_1), 2, 2, 1, 1, S(\mu, \nu), 2, 2, 1, 1, S(\mu, \nu), 2],
[2, S(\mu, \nu), 2]
$$
and by straightforward induction and Theorem \ref{frobeniustheorem} (\ref{frobeniustheoremvi}), we get at the position $pLR^n$ for $n > 0$:
$$
[2, S(\mu_1, \nu_1), 2, \langle 2, 1, 1,  S(\mu, \nu), 2\rangle_n],
[2, 1, 1, S(\mu_1, \nu_1), 2, \langle 2, 1, 1, S(\mu, \nu), 2 \rangle_{n + 1}], [2, S(\mu, \nu), 2].
$$
It follows that the limit of finite continued fractions over $pLR^n$ for $n \to \infty$ exists:
$$
\lim_{n \to \infty} [2, S(\mu_1, \nu_1), 2, \langle 2, 1, 1, S(\mu, \nu), 2 \rangle_{n + 1}] =
[2, S(\mu_1, \nu_1), 2, \overline{2, 1, 1, S(\mu, \nu), 2}].
$$
Using symmetry, we get moreover
\begin{gather*}
\mathfrak{l}_{p_1} = \lim_{n \to \infty} [2, S(\mu_1, \nu_1), 2, \langle 2, 1, 1, S(\mu, \nu), 2 \rangle_{n + 1}] \\
= \lim_{n \to \infty} [\langle 2, S(\mu, \nu), 1, 1, 2 \rangle_{n + 1}, 2, S(\mu_1, \nu_1), 2] \\
= [\overline{2, S(\mu, \nu), 1, 1, 2}].
\end{gather*}
In particular, $\mathfrak{l}_{p_1}$ is a quadratic number which we can determine using (\ref{makeperiodic}). For this, we first observe
$$
[2, 1, 1, S(\mu, \nu), 2] = 2 + \frac{1}{[1, 1, S(\mu, \nu), 2]} = \frac{K(2, S(\mu, \nu), 1, 1, 2)}{K(2, S(\mu, \nu), 1, 1)} = \frac{3g - r_g}{g},
$$
where we have used Lemma \ref{negativecontfrac} and the symmetry of continuants. Moreover, with
$$
(3r_g - s_g) g = 3g r_g - r_g^2 - 1 = (3g - r_g) r_g - 1
$$
and equation (\ref{reverseontfrac}) we get
$$
K(S(\mu, \nu), 1, 1, 2) = 3r_g - s_g.
$$
Finally, with $[2, S(\mu, \nu), 2] = g/r_g$ we get
$$
K(S(\mu, \nu), 1, 1) = K(S(\mu, \nu), 2) = r_g.
$$
Plugging these four continuants into (\ref{makeperiodic}), we compute
$$
\mathfrak{l}_{p_1} = \frac{1}{2} \frac{3g - 2r_g + \sqrt{9g^2 - 4}}{3r_g - s_g}.
$$

In a similar fashion, we obtain at the position $pRL^n$ for $n > 0$:
$$
[2, S(\mu, \nu), 2], 1 + [\langle 1, S(\mu, \nu), 2, 2, 1 \rangle_{n + 1}, 1, S(\mu_2, \nu_2), 2],
[\langle 1, S(\mu, \nu), 2, 2, 1 \rangle_n, 1, S(\mu_2, \nu_2), 2].
$$
Using symmetry again, we get:
$$
\mathfrak{l}_{p_2} = 1 + [\overline{1, S(\mu, \nu), 2, 2, 1}] =: 1 + \xi.
$$
In order to compute the purely periodic continued fraction, we use
$$
\sigma := [\overline{2, 1, 1, S(\mu, \nu), 2}] = [2, 1, \overline{1, S(\mu, \nu), 2, 2, 1}].
$$
By symmetry, we can reuse the continuants we have computed above and compute with \ref{makeperiodic}:
\begin{align*}
\sigma = \frac{1}{2g} (3g - 2r_g + \sqrt{9g^2 - 4}).
\end{align*}
Then, with $\sigma = 2 + 1/(1 + 1/\xi)$, we get:
$$
\xi = [\overline{1, S(\mu, \nu), 2, 2, 1}] = \frac{\sigma - 2}{3 - \sigma} = \frac{1}{2} \frac{3g - 4r_g -2s_g + \sqrt{9g^2 - 4}}{3r_g + s_g}
$$
(see also Formula (\ref{periodshift2})) and the assertion follows.
\end{proof}

\begin{sub}
\begin{figure}[ht]
\tikzstyle{level 1}=[level distance=1cm, sibling distance=8cm]
\tikzstyle{level 2}=[level distance=1cm, sibling distance=4cm]
\tikzstyle{level 3}=[level distance=3cm, sibling distance=2cm]
\tikzstyle{level 4}=[level distance=5cm, sibling distance=1cm]
\begin{tikzpicture}[grow=right, sloped]
  \node[align=left] {$\frac{1}{2}\frac{4+\sqrt{32}}{2}$, $2, 2$\\ $\frac{1}{2}(1+\sqrt{5})$, $1, 1$}
    child {node[align=left]{$\frac{1}{2}\frac{11+\sqrt{221}}{5}$\\ $2, 1, 1, 2$}
      child {node[align=left] {$\frac{1}{2}\frac{29+\sqrt{1517}}{13}$\\ $2, 1, 1, 1, 1, 2$}
        child {node[align=left] {$\frac{1}{2}\frac{68+\sqrt{10400}}{34}$\\ $2, 1, 1, 1, 1, 1, 1, 2$}
           child {node[align=left] {$\frac{1}{2}\frac{199+\sqrt{71285}}{89}$\\ $2, 1, 1, 1, 1, 1, 1, 1, 1, 2$}}
           child {node[align=left] {$\frac{1}{2}\frac{105+\sqrt{71285}}{115}$\\ $1, 1, 1, 1, 1, 1, 1, 2, 2, 1$}}
        }
        child {node[align=left] { $\frac{1}{2}\frac{40+\sqrt{10400}}{44}$\\ $1, 1, 1, 1, 1, 2, 2, 1$}
           child {node[align=left] {$\frac{1}{2}\frac{2961+\sqrt{15800621}}{1327}$\\ $2, 1, 1, 1, 1, 2, 2, 1, 1, 1, 1, 1, 1, 2$}}
           child {node[align=left] {$\frac{1}{2}\frac{1559+\sqrt{15800621}}{1715}$\\ $1, 1, 1, 1, 1, 2, 2, 1, 1, 1, 1, 2, 2, 1$}}
        }
      }
      child {node[align=left] {$\frac{1}{2}\frac{15+\sqrt{1517}}{17}$\\ $1, 1, 1, 2, 2, 1$}
        child {node[align=left] { $\frac{1}{2}\frac{432+\sqrt{338720}}{196}$\\ $2, 1, 1, 2, 2, 1, 1, 1, 1, 2$}
         child {node[align=left] { $\frac{1}{2}\frac{16837+\sqrt{514518485}}{15278}$\\ $2, 1, 1, 2, 2, 1, 1, 1, 1, 2, 2, 1, 1, 1, 1, 2$}}
         child {node[align=left] {$\frac{1}{2}\frac{8731+\sqrt{514518485}}{19798}$\\ $1, 1, 1, 2, 2, 1, 1, 1, 1, 2, 2, 1, 1, 2, 2, 1$}}
        }
        child {node[align=left] {$\frac{1}{2}\frac{224+\sqrt{338720}}{254}$\\ $1, 1, 1, 2, 2, 1, 1, 2, 2, 1$}
         child {node[align=left] { $\frac{1}{2}\frac{6451+\sqrt{75533477}}{2927}$\\ $2, 1, 1, 2, 2, 1, 1, 2, 2, 1, 1, 1, 1, 2$}}
         child {node[align=left] {$\frac{1}{2}\frac{3345+\sqrt{75533477}}{3793}$\\ $1, 1, 1, 2, 2, 1, 1, 2, 2, 1, 1, 2, 2, 1$}}
        }
      }
    }
    child {node[align=left] {$\frac{1}{2}\frac{5+\sqrt{221}}{7}$ \\ $1, 2, 2, 1$}
      child {node[align=left] {$\frac{1}{2}\frac{63+\sqrt{7565}}{41}$ \\ $2, 2, 2, 1, 1, 2$}
        child {node[align=left] { $\frac{1}{2}\frac{941+\sqrt{1687397}}{463}$\\ $2, 2, 2, 1, 1, 2, 2, 1, 1, 2$}
         child {node[align=left] { $\frac{1}{2}\frac{14052+\sqrt{376282400}}{6914}$\\ $2, 2, 2, 1, 1, 2, 2, 1, 1, 2, 2, 1, 1, 2$}}
         child {node[align=left] {$\frac{1}{2}\frac{6496+\sqrt{376282400}}{9124}$\\ $1, 2, 2, 1, 1, 2, 2, 1, 1, 2, 2, 2, 2, 1$}}
        }
        child {node[align=left] {$\frac{1}{2}\frac{435+\sqrt{1687397}}{611}$\\ $1, 2, 2, 1, 1, 2, 2, 2, 2, 1$}
         child {node[align=left] {$\frac{1}{2}\frac{81856+\sqrt{12768548000}}{40276}$\\ $2, 2, 2, 1, 1, 2, 2, 2, 2, 1, 1, 2, 2, 1, 1, 2$}}
         child {node[align=left] {$\frac{1}{2}\frac{37840+\sqrt{12768548000}}{53150}$\\ $1, 2, 2, 1, 1, 2, 2, 2, 2, 1, 1, 2, 2, 2, 2, 1$}}
        }
      }
      child {node[align=left] {$\frac{1}{2}\frac{29+\sqrt{7565}}{31}$ \\ $1, 2, 2, 2, 2, 1$}
        child {node[align=left] {$\frac{1}{2}\frac{367+\sqrt{257045}}{181}$\\ $2, 2, 2, 2, 2, 1, 1, 2$}
          child {node[align=left] {$\frac{1}{2}\frac{31925+\sqrt{1945074605}}{31490}$\\ $2, 2, 2, 2, 2, 1, 1, 2, 2, 2, 2, 1, 1, 2$}}
          child {node[align=left] { $\frac{1}{2}\frac{14703+\sqrt{1945074605}}{41578}$\\ $1, 2, 2, 2, 2, 1, 1, 2, 2, 2, 2, 2, 2, 1$}}
        }
        child {node[align=left] { $\frac{1}{2}\frac{169+\sqrt{257045}}{239}$\\ $1, 2, 2, 2, 2, 2, 2, 1$}
          child {node[align=left] {$\frac{1}{2}\frac{2139+\sqrt{8732021}}{1055}$\\ $2, 2, 2, 2, 2, 2, 2, 1, 1, 2$}}
          child {node[align=left] { $\frac{1}{2}\frac{985+\sqrt{8732021}}{1393}$\\ $1, 2, 2, 2, 2, 2, 2, 2, 2, 1$}}
        }
      }
    };
\end{tikzpicture}
\caption{The purely periodic spectrum.}\label{quadraticfig1}
\end{figure}
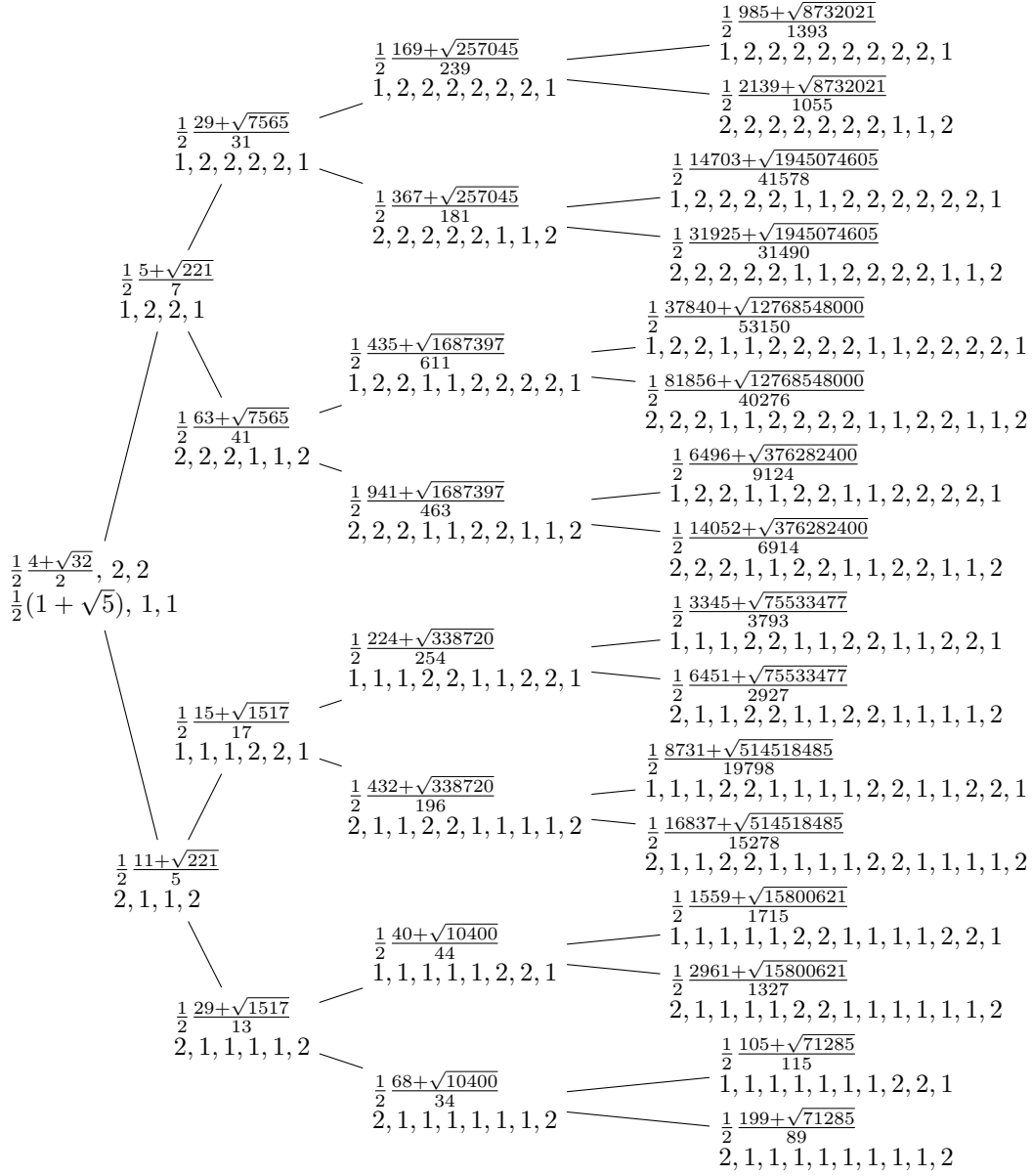

Figure \ref{quadraticfig1} shows the purely periodic values $\mathfrak{l}_{p_1}$ and
$\mathfrak{l}_{p_2} - 1 = \frac{1}{2} \frac{3g - 4r_g -2s_g + \sqrt{9g^2 - 4}}{3r_g + s_g}$, together with their periods,
for the first five levels of the Markoff tree. More precisely, the root shows $\mathfrak{l}_{p_2} - 1$ for $g = 1$ and $\mathfrak{l}_{p_1} - 1$ for $g = 2$,
and any pair of siblings corresponding to mutations $(e, F, g)$ and $(g, E, f)$, show $\mathfrak{l}_{p_1}$ and $\mathfrak{l}_{p_2} - 1$ corresponding to $g$,
respectively. As remarked in \ref{complement}, for any $\mu/\nu$, the sequences $S(\mu, \nu)$ from $S(\nu, \mu)$ are symmetric
to each other in the sense that they correspond to opposite locations in the Markoff tree and can be transformed into each other
by exchanging $1$'s and $2$'s. As can easily be seen from Proposition \ref{quadraticspectrum1}, in the representation of
Figure \ref{quadraticfig1} this symmetry extends to the purely periodic spectrum.
\end{sub}

In order to arrive at nicer representations, we pass now from the spectrum of quotients $m/r$ to the slopes $r/m$. By Formula
(\ref{cantorbound}), this spectrum is contained in the interval
$$
\frac{1}{2}(3 - \sqrt{5}) < \frac{r}{m} < \sqrt{2} - 1.
$$

\begin{definition}
Let $(e, g, f)$ be a regular Markoff triple located at a path $p$ and denote $p_1 = pL\bar{R}$, $p_2= pR\bar{L}$. Then we denote
$$
A_g := 1 / \mathfrak{l}_{p_1}, \qquad B_g := 1 / \mathfrak{l}_{p_2}.
$$
\end{definition}

We compute
$$
\left(\frac{1}{2} \frac{3g \pm 2 r_g + \sqrt{9g^2-4}}{3r_g \pm s_g}\right)^{-1} = \frac{r_g}{g} \pm \frac{1}{2}\left(3 - \sqrt{9 - \frac{4}{g^2}}\right),
$$
hence $A_g < B_g$ such that the closed interval $[A_g, B_g]$ is centered at $r_g/g$ and has length $3 - \sqrt{9 - 4/g^2}$.
The four branches of Figure \ref{chains}, from left to right, are given by $p\bar{L}$, $p_1$, $p_2$, and $p\bar{R}$,
respectively, and it follows that
$$
1 / \mathfrak{l}_{p\bar{L}} = B_e < A_g < B_g < A_f = 1 / \mathfrak{l}_{p\bar{R}}.
$$
Moreover, for any other infinite path of the form $pq$ it follows that either
$$
B_e < 1 / \mathfrak{l}_{pq} < A_g \quad \text{ or } \quad B_g < 1 / \mathfrak{l}_{pq} < A_f.
$$
For any regular Markoff triple $(e, g, f)$, we denote the intervals:
\begin{align*}
I_g & = [B_e, A_f], \\
J_g & = (A_g, B_g),
\end{align*}
and a short computation yields for their lengths:
\begin{align*}
| I_g | & = \frac{1}{ef}\left(\Delta_{e, f} - g\right), \\
| J_g | & = \frac{1}{g}(3g - \sqrt{\Delta_g}).
\end{align*}

With this, we can construct a Cantor set in the interval $[\frac{1}{2}(3 - \sqrt{5}), \sqrt{2} - 1]$ as follows. For $n \geq 0$,
in the $n$-th step we denote $(e^n_1, g^n_1, f^n_1), \dots, (e^n_{2^n}, g^n_{2^n}, f^n_{2^n})$ the Markoff triples on the $n$-th
level of the Markoff tree, enumerated from left to right. Then for any $n \geq 0$, the union $C_{n + 1} := \bigcup_{i = 1}^{2^{n + 1}}
I_{g_i^{n + 1}}$ coincides with
$\coprod_{i = 1}^{2^n} I_{g_i^n} \setminus J_{g^n_i}$
and, in particular, is disjoint. It follows that
$$
\bigcap_{n \geq 0} C_n = \bigcap_g I_g = [\frac{1}{2}(3 - \sqrt{5}), \sqrt{2} - 1] \setminus \bigcup_{g} J_g
$$
forms a Cantor set in the interval $[\frac{1}{2}(3 - \sqrt{5}), \sqrt{2} - 1]$, where $g$ runs over
the maximal elements of every regular Markoff triple.

\begin{definition}
We denote by $\mathfrak{M}$ the set of values $\{1 / \mathfrak{l}_{p}\}$, where $p$
runs over all infinite paths in the Markoff tree that start at the root
\end{definition}

\begin{theorem}\label{lebesguetheorem}
$\mathfrak{M}$ forms a Cantor
set of Lebesgue measure zero in the interval $[\frac{1}{2}(3 - \sqrt{5}), \sqrt{2} - 1]$.
\end{theorem}

\begin{proof}
It only remains to show that our Cantor set has measure zero.
With above notation, it equals $\bigcap_{n \geq 0} C_n$, where every $C_n$ is a finite disjoint union of intervals in $\R$.
By a standard argument, it suffices to show that
$| I_g | < 3 | J_g |$ for every regular Markoff triple $(e, g, f)$. So we want to show that
$$
\frac{1}{ef} \Delta_{e, f} - \frac{g}{ef} < \frac{3}{g}(3g - \sqrt{\Delta_g}).
$$
With $(3g + \sqrt{\Delta_g})/2 < 3g$ we show the sharper bound:
$$
\frac{1}{ef}  \Delta_{e, f} < \frac{g}{ef} + \frac{1}{2g^2} (3g + \sqrt{\Delta_g})(3g - \sqrt{\Delta_g}) =
\frac{g}{ef} + \frac{2}{g^2}.
$$
Using $g < 3ef$, it suffices to show:
$$
\frac{1}{ef} \Delta_{e, f} < \frac{g}{ef} + \frac{2}{3efg},
$$
or, equivalently:
$$
\Delta_{e, f} < g + \frac{2}{3g}.
$$
This inequality has been shown in Lemma \ref{somewhatsharpbound} (\ref{somewhatsharpboundii}).
\end{proof}

With the Cantor set being of measure zero, it follows that the lengths of the complementary open intervals add
up to the length of the interval $I_5$:

\begin{corollary}[see also {\cite[\S 4.5]{LangTan07}}]\label{gapsum}
It follows that
$$
(3 - \sqrt{5}) + (3 - \sqrt{8}) + 2 \sum_{(e, g, f)} \left(3 - \sqrt{9 - \frac{4}{g^2}}\right) = 1,
$$
where the sum runs over the regular Markoff triples.
\end{corollary}

\begin{remark}
Theorem \ref{lebesguetheorem} and Corollary \ref{gapsum} are well-known to specialists, though \cite{LangTan07}
is the only prior reference we could find. According to \cite{LangTan07}, Corollary \ref{gapsum} can be derived from
McShane's identity (see \cite{Bowditch96}), whereas our proof is elementary.
\end{remark}

Next we are going to determine the Hausdorff dimension $\dim_H \mathfrak{M}$ of $\mathfrak{M}$
(see Appendix \ref{cantorappendix} for details and references). Consider any
regular Markoff triple $(e, g, f)$ corresponding to a path $W \in \mathfrak{P}_f$. Then we denote $I_W := I_g$
and we set $d_{WX} = |I_{WX}| / |I_W|$ for $X \in \{L, R\}$.
Now, for any $p \in \mathfrak{P}$ and its corresponding point $x_p \in \mathfrak{M}$, we define the {\em local Hausdorff measure}
$$
h^s(x) = \liminf_{k \rightarrow \infty} \prod_{i = 0}^k (d_{p|_i L}^s + d_{p|_i R}^s),
$$
where $p|_i$ denotes the prefix of length $i$ of $p$.
We define the {\em local Hausdorff dimension} of $\mathfrak{M}$ at $x$ as:
$$
h(x) = \inf_{s \geq 0}  \{ h^s(x) > 0 \}.
$$

\begin{theorem}\label{hausdorfftheorem}
$\dim_H \mathfrak{M} = 0$ and $h(x) = 0$ for any $x \in \mathfrak{M}$.
\end{theorem}

\begin{proof}
Using Corollary \ref{hausdorffzerodim}, $\dim_H \mathfrak{M} = 0$ is implied
by $h(x) = 0$ for every $x \in \mathfrak{M}$. Therefore, it suffices to show the latter. We consider three cases
of points $x_p$, depending on the structure of $p$.

{\em The case $p = p|_k\bar{L}$ for some $k$.}

In order to show that $h^s(x_p) = 0$ for any $s > 0$, we are going to show that $d^s_{p|_j} + d^s_{p|_j} < 1$
for sufficiently large $j \gg 0$ for any $s > 0$. For this, we show that both $d_{p|_jL}$ and $d_{p|_j}$ are
bounded from above by monotonously decreasing series, where for $d_{p|_jL}$ the bounding series has limit $< 1$
and for $d_{p|_jR}$ we obtain $\lim_{j \rightarrow \infty} d_{p|_jR} = 0$.

Using the same notation as in Lemma \ref{growthlemma}, assume that $p = p|_k \bar{L}$ and denote $(e, g, f)$ the
Markov triple corresponding to the path $p|_k$. Then we get:
$$
d_{p|_{k + i + 1}} = \frac{|I_{p|_{k + i + 1}}|}{|I_{p|_{k + i}}|} = \frac{F_{i - 1}}{F_i}
\frac{\Delta_{e, F_i} - F_{i + 1}}{\Delta_{e, F_{i - 1}} - F_i}
\quad \text{and} \quad
d_{p|_{k + i}R} = \frac{|I_{p|_{k + i}R}|}{|I_{p|_{k + i}}|} = \frac{e}{F_i}
\frac{\Delta_{F_i, F_{i - 1}} - H_{i + 1}}{\Delta_{e, F_{i - 1}} - F_i}
$$
for $i \geq 0$, where $H_i := 3F_iF_{i - 1} - e$.
Using first Lemma \ref{somewhatsharpbound} (\ref{somewhatsharpboundiii}), then
\ref{somewhatsharpbound} (\ref{somewhatsharpboundii}), then $\sqrt{\Delta_{F_{i - 1}}} < 3F_{i - 1}$, we get:
$$
d_{p|_{k + i + 1}} < \frac{\sqrt{\Delta_e} \sqrt{\Delta_{F_i}}}{2e} (\Delta_{e, F_{i + 1}} - F_{i + 2}) <
\frac{\sqrt{\Delta_e}}{9e^2} \frac{\sqrt{\Delta_{F_{i - 1}}}}{F_i} <
\frac{\sqrt{\Delta_e}}{3e^2} \frac{F_{i - 1}}{F_i}.
$$
From Lemma \ref{growthlemma} (\ref{growthlemmai}),
it is straightforward to see that $F_i / F_{i + 1} < F_{i - 1} / F_i$ for every $i$ and moreover,
$$
\lim_{i \rightarrow \infty} \frac{\sqrt{\Delta_e}}{3e^2} \frac{F_{i - 1}}{F_i} =
\frac{2}{3e^2} \frac{\sqrt{\Delta_e}}{3e + \sqrt{\Delta_e}} < 1.
$$
Similarly, we get
$$
d_{p|_{k + i}R} < \frac{\sqrt{\Delta_e}}{3 H_{i + 1}} \xrightarrow[\ i \rightarrow \infty \ ]{} 0.
$$
Note that these inequalities also show that $\mathfrak{M}$ satisfies the conditions for Corollary \ref{hausdorffzerodim}.

{\em The case $p = p|_k\bar{R}$ for some $k$.}

Follows analogously.

{\em The remaining paths.}

Let $p$ be any infinite path not in $\mathfrak{P}_f$. Then $p$ has infinitely many continuous segments
consisting of repetitions of $L$. Let denote $p|_{k_j}, \dots, p|_{k_j + t_j}$ the $j$-th such segment,
i.e. the last letter of every $p|_{k_j + i}, 0 \leq i \leq t_j$ is $L$. Denote $(e_j, g_j, f_j)$ the
root of this segment. Then we have seen above that $d_{p|_{k_j + i}}$ is bounded by values
(very close to) $1/3e_j^2$.
Therefore, for $j \rightarrow \infty$,  $d_{p|_n} \rightarrow 0$ for $k_j \leq n \leq k_j + t_j$.
Similarly, the $d_{p|_{n - 1}R}$ are bounded even stronger. We can argue analogously for the segments
consisting of $R$s. Therefore $\lim_{i \rightarrow \infty}  (d^s_{p|_iL} + d^s_{p|_iR}) = 0$
for every $s > 0$, hence $h(x_p) = 0$.
\end{proof}

\section{Markoff triples and T-singularities}\label{equationsection}

According to \cite{HackingProkhorov10}, given a Markoff triple $(e, g, f)$, 
the singular points of $\mathbb{P}(e^2, g^2, f^2)$ are cyclic quotient singularities of orders $e^2$, $g^2$,
and $f^2$, respectively. These singularities can locally be described as the diagonal action of, say, $\Z / g^2 \Z$
on $\C^2$ with weights $e^2$ and $f^2$ (shorthand: $\frac{1}{g^2}(e^2, f^2)$). By choosing an appropriate
$g^2$-th root of unity, this is equivalent to $\frac{1}{g^2}(1, e^{-2}f^2 \mod g^2)$ and, using the Markoff equation, to:
$$
\frac{1}{g^2}(1, g w_g - 1),
$$
where $w_g = 3 e^{-1}f \mod g$. This is a particular example of a class of surface singularities called {T-singularities}
which we will review in more detail in Section \ref{tsingsection}. Repeating this calculation for the singularities of
orders $e^2$ and $f^2$ leads to the following definition.

\begin{definition}\label{weightdefinition}
Denote $(e, g, f)$ be a regular Markoff triple. Then we set:
\begin{align*}
w_e & := 
\begin{cases}
-1 & \text{ if } e = 1, \\
 3f^{-1}g \mod e & \text{ else},
\end{cases}\\
w_f & := 3g^{-1}e \mod f, \\
w_g & := 3e^{-1} f \mod g.
\end{align*}
We call $(w_e, w_g, w_f)$ the {\em T-weights} of the triple $(e, g, f)$.
\end{definition}

The Markoff equation implies that $(e^{-1}f)^2 \equiv -1 \mod g$ and hence $w_g^2 \equiv -9 \mod g$.
Similarly, $w_e^2 \equiv -9 \mod e$ and $w_f^2 \equiv -9 \mod f$.

\begin{definition}
	We define the {\em T-coweights} $(v_e, v_g, v_f)$ of $(e, g, f)$ to be the
	unique integers such that the following equations are satisfied:
	\begin{align*}
		w_e^2 & = -9 + v_e e, \\
		w_f^2 & = -9 + v_f f, \\
		w_g^2 & = -9 + v_g g.
	\end{align*}
\end{definition}

As the reader has noticed, the $T$-weights and $T$-coweights are very similar to the weights and coweights
we have considered in Section \ref{frobenius}. The following Lemma is proved in the same way as Lemma
\ref{cofactorlemmaf}.

\begin{lemma}\label{tcofactorlemmaf}
Let $(e, g, f)$ be a regular Markoff triple.
\begin{enumerate}[(i)]
	\item\label{tcofactorlemmafi}
	The weights and coweights transform under mutation as follows:
	\begin{center}
		\tikzstyle{level 1}=[level distance = 2cm, sibling distance=4cm]
		\begin{tikzpicture}[]
			\node{\begin{tabular}{r@{\hskip 1pt}rrr@{\hskip 1pt}l}( & $e$, & $g$, & $f$ & )\\ (& $w_e$, & $w_g$, & $w_f$ & )\\ ( & $v_e$, & $v_g$, & $v_f$ & )\end{tabular}}
			child {node {\begin{tabular}{r@{\hskip 1pt}rrr@{\hskip 1pt}l}( & $e$, & $F$, & $g$ & )\\ (& $w_e$, & $w_F$, & $w_g$ & )\\ ( & $v_e$, & $v_F$, & $v_g$ & )\end{tabular}}}
			child {node {\begin{tabular}{r@{\hskip 1pt}rrr@{\hskip 1pt}l}( & $g$, & $E$, & $f$ & )\\ (& $w_g$, & $w_E$, & $w_f$ & )\\ ( & $v_g$, & $v_E$, & $v_f$ & )\end{tabular}}};
		\end{tikzpicture}
	\end{center}
	where
	$$
	w_E = 3f w_g - w_e, \quad v_E = 3f v_g - v_e, \quad w_F = 3e w_g - w_f, \ \ \text{ and } \ \ v_F = 3e v_g - v_f.
	$$
	\item\label{tcofactorlemmafii}
	The {\em slopes} of a regular Markoff triple are strictly increasing:
	$$
	\frac{w_e}{e} < \frac{w_g}{g} < \frac{w_f}{f}.
	$$
	\item\label{tcofactorlemmafiii}
	We have  $w_e < e - w_e$, $w_f < f - w_f$, and $w_g < g - w_g$ for all $e, g$ and all $f > 2$.
\end{enumerate}
\end{lemma}

Figure \ref{treefig2} shows the Markoff triples, their weights and coweights of
the first four levels of the Markoff tree.
\setlength{\tabcolsep}{1pt}
\begin{figure}[ht]
\footnotesize
\tikzstyle{level 1}=[level distance = 1cm, sibling distance=8cm]
\tikzstyle{level 2}=[sibling distance=4.2cm]
\tikzstyle{level 3}=[sibling distance=1.9cm, level distance = 2cm]
\begin{tikzpicture}[]
  \node {\begin{tabular}{r@{\hskip .5pt}rrr@{\hskip 1pt}l}( & 1, & 5, & 2 & )\\ (& -1, & 1, & 1 & )\\ ( & 10, & 2, & 5 & )\end{tabular}}
    child {node {\begin{tabular}{r@{\hskip .5pt}rrr@{\hskip 1pt}l}( & 1, & 13, & 5 & )\\ (& -1,& 2, & 1 & )\\ ( & 10, & 1, & 2 & )\end{tabular}}
      child {node {\begin{tabular}{r@{\hskip .5pt}rrr@{\hskip 1pt}l}( & 1, & 34, & 13 & )\\ (& -1, & 5, & 2 & )\\ ( & 10, & 1, & 1 & )\end{tabular}}
        child {node {\begin{tabular}{r@{\hskip .5pt}rrr@{\hskip 1pt}l}( & 1, & 89, & 34 & )\\ (& -1, & 13, & 5 & )\\ ( & 10, & 2, & 1 & )\end{tabular}}}
        child {node {\begin{tabular}{r@{\hskip .5pt}rrr@{\hskip 1pt}l}( & 34, & 1325, & 13 & )\\ (& 5, & 196, & 2 & )\\ ( & 1, & 29, & 1 & )\end{tabular}}}
      }
      child {node {\begin{tabular}{r@{\hskip .5pt}rrr@{\hskip 1pt}l}( & 13, & 194, & 5 & )\\ (& 2, & 31, & 1 & )\\ ( & 1, & 5, & 2 & )\end{tabular}}
        child {node {\begin{tabular}{r@{\hskip .5pt}rrr@{\hskip 1pt}l}( & 13, & 7561, & 194 & )\\ (& 2, & 1208, & 31 & )\\ ( & 1, & 193, & 5 & )\end{tabular}}}
        child {node {\begin{tabular}{r@{\hskip .5pt}rrr@{\hskip 1pt}l}( & 194, & 2897, & 5 & )\\ (& 31, & 463, & 1 & )\\ ( & 5, & 74, & 2  & )\end{tabular}}}
      }
    }
    child {node {\begin{tabular}{r@{\hskip .5pt}rrr@{\hskip 1pt}l}( & 5, & 29, & 2 & )\\ (& 1, & 7, & 1 & )\\ ( & 2, & 2, & 5 & )\end{tabular}}
      child {node {\begin{tabular}{r@{\hskip .5pt}rrr@{\hskip 1pt}l}( & 5, & 433, & 29 & )\\ (& 1, & 104, & 7 & )\\ ( & 2, & 25, & 2 & )\end{tabular}}
        child {node {\begin{tabular}{r@{\hskip .5pt}rrr@{\hskip 1pt}l}( & 5, & 6466, & 433 & )\\ (& 1, & 1553, & 104 & )\\ ( & 2, & 373, & 25 & )\end{tabular}}}
        child {node {\begin{tabular}{r@{\hskip .5pt}rrr@{\hskip 1pt}l}( & 433, & 37666, & 29 & )\\ (& 104, & 9047, & 7 & )\\ ( & 25, & 2173, & 2 & )\end{tabular}}}
      }
      child {node {\begin{tabular}{r@{\hskip .5pt}rrr@{\hskip 1pt}l}( & 29, & 169, & 2 & )\\ (& 7, & 41, & 1 & )\\ ( & 2, & 10, & 5 & )\end{tabular}}
        child {node {\begin{tabular}{r@{\hskip .5pt}rrr@{\hskip 1pt}l}( & 29, & 14701, & 169 & )\\ (& 7, & 3566, & 41 & )\\ ( & 2, & 865, & 10 & )\end{tabular}}}
        child {node {\begin{tabular}{r@{\hskip .5pt}rrr@{\hskip 1pt}l}( & 169, & 985, & 2 & )\\ (& 41, & 239, & 1 & )\\ ( & 10, & 58, & 5 & )\end{tabular}}}
      }
    };
\end{tikzpicture}
\caption{The first four levels of the Markoff tree with triples
$(e, g, f)$, $(w_e, w_g, w_f)$, $(v_e, v_g, v_f)$.}\label{treefig2}
\end{figure}
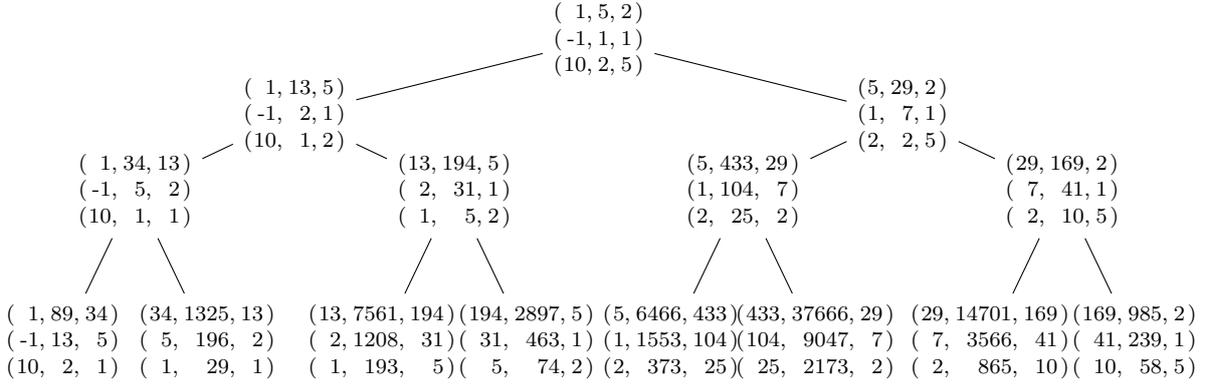

With Lemmas \ref{cofactorlemmaf} (\ref{cofactorlemmafiii}), \ref{cofactorlemma} (\ref{cofactorlemmai}), we get:

\begin{corollary}\label{affinetransform}
	For every regular Markoff triple, $w_e = 3 r_e - e$, $w_g = 3 r_g - g$, and $w_f = 3 r_f - f$.
\end{corollary}

For later use, we state the following formulas (compare \cite[\S II.3, Lemma 7]{Cassels57}).

\begin{lemma}\label{cofactorlemma}
Let $(e, g, f)$ be a regular Markoff triple. Then the following equations hold:
\begin{enumerate}[(i)]
\item\label{cofactorlemmai}
\begin{align*}
g w_f - f w_g & = 3e,\\
e w_g - g w_e & = 3f, \\
e w_f - f w_e & = 3G.
\end{align*}
\item\label{cofactorlemmaiii}
$$Gw_g = ew_e + fw_f.$$
\item\label{cofactorlemmav}
\begin{align*}
f v_e & = w_e w_f + 3(w_g - 3f w_e)\\
e v_f & = w_e w_f + 3(3e w_f - w_g).
\end{align*}
\end{enumerate}
\end{lemma}

\begin{proof}
(\ref{cofactorlemmai})
We will only prove the the last equation and leave the others for the reader.
First we observe $ew_f - fw_e \equiv -3g \mod ef$, hence $ew_f - fw_e = -3g + xef$
for some integer $x$. From $ew_f - fw_e + 3g < ew_f - fw_e + 3g + 3G = e w_f - fw_e + 9ef < 10ef$
we get $x < 10$. Moreover, we have $ew_f - fw_e + 3g = ew_f - fw_e + 9ef - 3G = ew_f + f(e - w_e) + 8ef -3G$.
The latter term is greater than $8ef$ (implying x = 9 and therefore the assertion) iff $ew_f + f(e - w_e) > 3G$.
Because $G < f$, this inequality will be satisfied if $3G < f(e - w_e)$ holds and for this it is enough to consider
the cases $e - w_e = 1, 2$. If $e - w_e = 1$ then $(e - w_e)^2 \equiv -9 \equiv 1 \mod e$, hence
$e \in \{1, 2, 5\}$, but for $e = 1$ we have $e - w_e = 2$ and for $e = 5$ we have $e - w_e = 4$, hence $e = 2$.
In that case, $(e, g, f) = (P_{2n - 1}, P_{2n + 1}, 2)$.
If $G = P_{2n - 3}$ then $3G = 3P_{2n-3} < P_{2n + 1}$ by the well-known properties of Pell numbers. Otherwise,
if $G < e$ then $G = 1$ and $(e, g, f) = (5, 29, 2)$, thus $3G < f$.

If $e - w_e = 2$ then $(e - w_e)^2 \equiv -9 \equiv 4 \mod e$, hence $e \in \{1, 13\}$, but for $e = 13$ we
have $e - w_e = 11$, hence $e = 1$. But then $(e, g, f) = (1, F_{2n + 1}, F_{2n - 1})$ and we conclude analogously
as for the Pell numbers.

(\ref{cofactorlemmaiii})
By (\ref{cofactorlemmai}) we have $ew_fw_g + fw_ew_g - gw_ew_f = ew_fw_g - 3ew_e = fw_ew_g + 3fw_f$
which, again by (\ref{cofactorlemmai}),  implies the assertion.

(\ref{cofactorlemmav})
We only show the first equality, the other follows analogously.
\begin{align*}
efv_e & = f(9 + w_e^2) = 9f + fw_e^2 - e w_e w_f + e w_e w_f\\
& = 9f - 3Gw_e + e w_e w_f \quad \text{ by (\ref{cofactorlemmai}) }\\
& = e w_e w_f + 3e(w_g - 3f w_e) \quad \text{ again by (\ref{cofactorlemmai}). }
\end{align*}
\end{proof}

The following observations regarding the Fibonacci and Pell branches may be of independent
interest:

\begin{lemma}\label{fibonaccipellweights}\
\begin{enumerate}[(i)]
\item
Let $(e, g, f)$ be a Markoff triple of the form $(1, F_{2n + 1}, F_{2n - 1})$ for $n \geq 2$. Then:
\begin{align*}
(w_e, w_g, w_f) & = (-1, F_{2n - 3}, F_{2n - 5}), \\
(e - w_e, g - w_g, f - f_w) & = (2, L_{2n - 1}, L_{2n - 3}), \\
(v_e, v_g, v_f) & = (10, F_{2n - 7}, F_{2n - 9}),
\end{align*}
where $L_i$ denotes the $i$-th Lucas number.
\item
Let $(e, g, f)$ be a Markoff triple of the form $(P_{2n - 1}, P_{2n + 1}, 2)$ for $n \geq 2$. Then:
\begin{align*}
(w_e, w_g, w_f) & = (S_{2n - 3}, S_{2n - 1}, 1), \\
(e - w_e, g - w_g, f - w_f) & = ( Q_{2n - 3}, Q_{2n - 1}, 1), \\
(v_e, v_g, v_f) & = (R_{2n - 5}, R_{2n - 3}, 5),
\end{align*}
where  $Q_i = 2Q_{i - 1} + Q_{i - 2}$, $R_i = 2R_{i - 1} + R_{i - 2}$, $S_i = 2S_{i - 1} + S_{i - 2}$
with initial values $Q_0 = 1$, $Q_1 = 4$, $R_1 = 2$, $R_2 = 4$, and $S_0 = 1$, $S_1 = 1$, respectively.
\end{enumerate}
\end{lemma}

\begin{remark}
The sequence $2S_i$ are the Pell-Lucas numbers and the sequence $S_{2n+1}$ whose first few terms are
$1, 7, 41, 239, 1393, 8119, 47321, \dots$ is known as the  sequence of NSW numbers. Together with
some interesting interpretations it can be found in the Online Encyclopedia of Integer Sequences
under reference A002315 (see \cite{OEIS}). The $R_i$ can be found under A052542, and the odd entries of that
series under A075870. The $Q_i$ are referenced by A048654, their odd entries by A100525. The sequences
$(L_{2n - 1}^2 + 9) / F_{2n + 1}$ and $(Q_{2n - 1}^2 + 9)/P_{2n+1}$ are referenced by A106729 and
A254759, respectively.
\end{remark}

\section{T-singularities and continued fractions}\label{tsingsection}

In this section we consider T-singularities of type 
$$ \frac{1}{n^2}(1, nk - 1), $$
where $1 < n$ is not necessarily a Markoff number and $0 < k < n$ is coprime to $n$.
General T-singularities have been introduced by Wahl \cite{Wahl81} and play an important role in the
three-dimensional minimal model program.
Our main source is \cite[§ 3]{KollarShepherdbarron88}.

Given such a singularity, its minimal desingularization can be described with help of Hirzebruch-Jung
continued fractions, i.e. if we write
$$ \frac{n^2}{nk-1} = \llbracket a_1, a_2, \dots, a_t\rrbracket := a_1 - 1 / (a_2 - 1 / (a_3 - \cdots (a_{t - 1} - 1 / a_t) \cdots )), $$
then the intersection graph of the minimal resolution is a chain with $t$ vertices which are labelled by $-a_1,
\dots, -a_t$, where each $a_i \geq 2$. The continued fraction expansions corresponding to the T-singularities of above type have been
described in \cite[Proposition 3.11]{KollarShepherdbarron88} as follows:
\begin{itemize}
\item $\llbracket 4 \rrbracket$ corresponds to a T-singularity (which obviously is given by the pair $(n, k) = (2, 1)$),
\item If $\llbracket a_1, \dots, a_t\rrbracket$ corresponds to a T-singularity, then so do $\llbracket a_1 + 1, a_2, \dots, a_t, 2\rrbracket$
and $\llbracket 2, a_1, \dots, a_{t - 1},$ $a_t + 1\rrbracket$.
\end{itemize}

This classification suggests a more efficient representation of a T-singularity as follows. 
\begin{definition*}
The {\em length encoding} (LE) of a T-singularity of type $\frac{1}{n^2}(1, nk - 1)$ is defined as:
\begin{enumerate}[(i)]
\item The empty sequence $\emptyset$ for $(n, k) = (2,1)$.
\item If a sequence $c_1, \dots, c_m$ encodes a T-singularity corresponding to $\llbracket a_1, \dots, a_t\rrbracket$ with
$a_1 \neq 2$, then for any $c_{m + 1} > 0$, $c_1, \dots, c_{m + 1}$ encodes the T-singularity corresponding to
$\llbracket 2, \dots, 2, a_1, \dots, a_{t - 1}, a_t + c_{m + 1}\rrbracket$.
\item Similarly, if a sequence $c_1, \dots, c_m$ encodes a T-singularity corresponding to $\llbracket a_1, \dots, a_t\rrbracket$ with
$a_t \neq 2$, then for any $c_{m + 1} > 0$, $c_1, \dots, c_{m + 1}$ encodes the T-singularity corresponding to
$\llbracket a_1 + c_{m + 1}, a_2, \dots,$ $a_t, 2, \dots, 2\rrbracket$.
\end{enumerate}
\end{definition*}
Note that a continued fraction $\llbracket a_1, \dots, a_t\rrbracket$ corresponds to a T-singularity iff $\llbracket a_t, \dots, a_1\rrbracket$
does. More precisely, if  $\llbracket a_1, \dots, a_t\rrbracket$ corresponds to $ \frac{1}{n^2}(1, nk - 1)$ then
$\llbracket a_t, \dots, a_1\rrbracket$ corresponds to $\frac{1}{n^2}(1, n(n - k) - 1)$. This reflects the fact that
$n(n - k) - 1 \equiv (nk - 1)^{-1} \mod n^2$ and that, from a geometric
point of view, there is no natural preference whether to read the terms in our minimal resolutions from left to right or from right to left.
The LE coincides for both representations.

\begin{lemma}\label{lelemma}\ 
\begin{enumerate}[(i)]
\item\label{lelemmai} A positive integer $c$ is the LE of the pair $(c + 2, 1)$.
\item\label{lelemmaii} 
For $m > 1$ let $c_1, \dots, c_m$ be a sequence of positive integers. If $c_1, \dots, c_{m - 1}$
is the LE  for the
pair $(n, k)$, where $k < n - k$, then $c_1, \dots, c_m$ is the LE of the pair $(n + c_m (n -k), n - k)$.
\end{enumerate}
\end{lemma}

\begin{proof}
(\ref{lelemmai})
Consider first the continued fraction $\llbracket 2, \dots, 2\rrbracket$ (with $c$ entries). Then it is easy to check that
$\llbracket 2, \dots, 2\rrbracket = \frac{c + 1}{c}$. The LE $c$ represents $\llbracket c + 4, 2, \dots, 2\rrbracket
= c + 4 - 1 / \llbracket 2, \dots, 2\rrbracket = c + 4 - c/(c + 1) = (c + 2)^2/(c + 1)$.

(\ref{lelemmaii})
We first consider a pair $(n, k)$ with $n^2/(nk - 1) = \llbracket a_1, \dots, a_t\rrbracket$, with $a_1 \neq 2$. Then
$\llbracket 2, a_t, \dots, a_1\rrbracket = 2 - 1 / \llbracket a_t, \dots, a_1\rrbracket = 2 - 1 / ( n^2 / (n(n - k) - 1))) = (n^2 + nk + 1) / n^2.$
It is easy to verify that $k^2 + nk - 1$ is inverse to $n^2$ modulo $n^2 + nk + 1$. Hence, we get
$\llbracket a_1, \dots, a_t, 2\rrbracket = (n^2 + nk + 1) / (k^2 + nk - 1) = ((n + k) n + 1) / ((n + k) k - 1)$. Then it follows
that $\llbracket a_1 + 1, a_2, \dots, a_t, 2\rrbracket = 1 + \llbracket a_1, \dots, a_t, 2\rrbracket = (n + k)^2 / ((n + k) k - 1)$.

We apply this result to an induction on $c_m$. For simplicity, we will allow the redundant case $c_m = 0$
for which the statement is obviously true. Now assume that $c_m > 0$ and $c_1, \dots, c_{m - 1}, c_m - 1$
encodes the pair $(n + (c_m - 1)(n - k), n - k)$. Then the statement follows immediately from above computation.
\end{proof}

Conversely, by the lemma we see that if we are given a pair $(n, k)$ then we can recover its LE by means of
the Euclidean algorithm. Namely, we set $n_0 := n - k$, $n_1 := k$ and inductively, for every $i > 1$ we define
$n_i$ via division with remainder $n_{i - 2} = p_{i - 1} n_{i - 1} + n_i$ such that $0 \leq n_i < n_{i - 1}$.
Then from the resulting sequence $p_m, \dots, p_1$ we get the LE $c_1, \dots, c_m$ for $(n, k)$, where
$c_i = p_{m + 1 - i}$ for $2 \leq i \leq m$ and $c_1= n_{m - 1} - 1$. Moreover, for each $0 \leq i < m$ the partial
sequence $c_1, \dots, c_{m - i}$ is the LE of $(n_i + n_{i + 1}, n_{i + 1})$. As a consequence, we can recover
the length encoding of a T-singularity by means of an ordinary continued fraction:

\begin{corollary}\label{tcontfrac}
Let $m > 1$ and $c_1, \dots, c_m$ be the LE of the pair $(n, k)$, where $k < n - k$. Then $\frac{n}{k} =
[ c_m + 1,, c_{m- 1}, \dots, c_2,$ $c_1 + 1] =
(c_m + 1) + 1/(c_{m - 1} + \cdots (c_{2} + 1/(c_1+ 1)) \cdots ))$.
\end{corollary}

Intuitively, one could consider Corollary \ref{tcontfrac} as a version of
$$
\frac{n^2}{nk - 1} \approx \frac{n}{k}.
$$
A strengthening of this sentiment is the observation below that the continued fraction expansion of $n^2/(nk - 1)$ is
almost symmetric and we can obtain the $c_i$ above from either half. For the proof of this claim, we start by reminding
the reader that $[a_1, \dots, a_t, 1] = [a_1, \dots, a_{t - 1}, a_t + 1]$. So, if we refer to the length of a continued
fraction, we have to make explicit whether the last element is $> 1$.

\begin{lemma}\label{contfracevenodd}
Consider a pair $(n, k)$ with $k < n - k$. Then:
\begin{enumerate}[(i)]
\item\label{contfracevenoddi}
$\frac{n^2}{nk - 1}$ has a continued fraction expansion of even length
$\frac{n^2}{nk - 1} = [a_1, \dots, a_{2s}]$ for some $s \geq 1$ such that $a_{2s} > 1$;
\item\label{contfracevenoddii}
$\frac{n^2}{nk + 1} = [a_{2s}, \dots, a_1]$;
\item\label{contfracevenoddiv}
$\frac{n^2}{n(n - k) + 1} = [1, a_1 - 1, a_2, \dots, a_{2s}]$;
\item\label{contfracevenoddiii}
$\frac{n^2}{n(n - k) - 1} = [1, a_{2s} - 1, a_{2s - 1}, \dots, a_1]$.
\end{enumerate}
\end{lemma}

\begin{proof}
If we write $\frac{n^2}{nk - 1} = [a_1, \dots, a_t]$ with $a_t > 1$ then by equation (\ref{reverseontfrac}), we have
$[a_t, \dots, a_1] = \frac{n^2}{X}$
for some $0 < X < g^2$ such that $(nk - 1)X \equiv (-1)^{t + 1} \mod n^2$. As remarked earlier, we have $(nk - 1)(n(n - k) - 1) \equiv
1 \mod n^2$, so either $X = n(n - k) - 1$ and $t$ is odd, or $X = nk + 1$ and $t$ is even. Because $a_t > 1$, we necessarily have
$\lfloor n^2 / X \rfloor > 1$,
hence $X = nk + 1$ and we conclude $(\ref{contfracevenoddi})$ and $(\ref{contfracevenoddii})$.

$(\ref{contfracevenoddiv})$ and $(\ref{contfracevenoddiii})$ then follow from $(\ref{contfracevenoddi})$, $(\ref{contfracevenoddii})$,
and Lemma \ref{negativecontfrac}.
\end{proof}

The following statement proves the even-ness of the continued fraction expansion of $n^2/(nk - 1)$ and it also
shows that these expansions are almost symmetric.

\begin{proposition}\label{palisplit}
Consider a pair $(n, k)$ with $k < n - k$ and assume that $n/k = [c_1, \dots, c_s]$ with $c_1, c_s > 1$. Then:
\begin{align*}
\frac{n^2}{nk - 1} =
\begin{cases}
[c_1, \dots, c_{s - 1}, c_s - 1, c_s + 1, c_{s - 1}, \dots, c_1] & \text{ if $s$ is even}, \\
[c_1, \dots, c_{s - 1}, c_s + 1, c_s - 1, c_{s - 1}, \dots, c_1] & \text{ if $s$ is odd}.
\end{cases}
\end{align*}
\end{proposition}

\begin{proof}
We have $n / k = [c_1, \dots, c_s]$ and $n / k' = [c_s, \dots, c_1]$, where $k k' = (-1)^{1 + s} + un$ by equation (\ref{reverseontfrac}).
Then $K(c_1, \dots, c_s) = n$, $K(c_2, \dots, c_s) = k$, $K(c_1, \dots, c_{s - 1}) = k'$, and $K(c_2, \dots, c_{s - 1}) = u$
by Cassini's identity (\ref{Cassini}). We compute
\begin{align*}
[c_1, \dots, c_{s - 1}, c_s - (-1)^s] & = \frac{n - (-1)^s k'}{k - (-1)^s u}, \\
[c_s + (-1)^s, c_{s - 1}, \dots, c_1] & = \frac{n + (-1)^s k'}{k'}.
\end{align*}
Then:
\begin{gather*}
K(c_1, \dots, c_{s - 1}, c_{s} + (-1)^s) K(c_{s} - (-1)^s, c_{s - 1}, \dots, c_1) + K(c_1, \dots, c_{s - 1})^2 \\
= (n -(-1)^s k')(n + (-1)^s k') + (k')^2 = n^2
\end{gather*}
and
\begin{align*}
K(c_2, \dots, c_{s - 1}, c_{s} + (-1)^s) & K(c_{s} - (-1)^s, c_{s - 1}, \dots, c_1) + K(c_2, \dots, c_{s - 1}) K(c_1, \dots, c_{s - 1}) \\
& = (k - (-1)^s u)(n + (-1)^s k') + u k' \\
& = nk + (-1)^s(k k' - un) \\
& = nk - 1.
\end{align*}
By the concatenation rule \ref{concatenationlaw}, we conclude:
$$
\frac{n^2}{nk - 1} = [c_1, \dots, c_{s - 1}, c_s - (-1)^s, c_s + (-1)^s, c_{s - 1}, \dots, c_1],
$$
hence the assertion.
\end{proof}

\section{Continued fractions and the Markoff tree}\label{finalchapter}

In this section we are going to derive mutation formulas for the LEs corresponding to T-singularities
arising from the Markoff equation. Unfortunately, it seems not to be possible to do this directly via the
fractions $g / w_g$, as would be suggested by Corollory \ref{tcontfrac}. Instead, we will show that
the continued fraction expansions of $g^2 / (g w_g - 1)$ will exhibit mutation rules analogous to those
described in Section \ref{frobenius}. The LEs then can be derived from these continued fractions with
Proposition \ref{palisplit} and Corollary \ref{tcontfrac}.

Using the mutations for $g^2 / (g w_g - 1)$, we will see that the tree of fractions $w_g / g$ has limits
in a Cantor set, very similar to the Cantor set of Section \ref{cantor}, which emerges from the
fractions $r_g / g$. Indeed, we will see that these Cantor sets coincide up to an affine transformation.

\begin{proposition}\label{append8}
Consider a Markoff number $g$ with T-weight $w$ and T-coweight $v$ corresponding to some Markoff triple containing $g$. We
denote $W := g - w$, $V := v + g - 2w$ such that $w < W$ and $g^2/(gw - 1) = [a_1, \dots, a_{2s}]$ with $a_1, a_{2s} > 1$.
Then the following equations hold:
\begin{enumerate}[(i)]
\item\label{append8i}
$[a_1, \dots, a_{2s}, 8] = \frac{9g^2 - gW + 1}{8gw + gv - 17}$,
\item\label{append8ii}
$[a_1, \dots, a_{2s - 1}, a_{2s} - 1, 1, 8] = \frac{9g^2 - gw - 1}{9gw - gv}$,
\item\label{append8iii}
$[1, a_1 - 1, a_2 \dots, a_{2s}, 8] = \frac{9g^2 - gW + 1}{9gW - gV + 18}$,
\item\label{append8iv}
$[1, a_1 - 1, a_2, \dots, a_{2s - 1}, a_{2s} - 1, 1, 8] = \frac{9g^2 - gw - 1}{8gW + gV - 1}$,
\item\label{append8v}
$[a_{2s}, \dots, a_1, 8] = \frac{9g^2 - gW - 1}{8gw + gv - 1}$,
\item\label{append8vi}
$[a_{2s}, \dots, a_2, a_1 - 1, 1, 8] = \frac{9g^2 - gw + 1}{9gw - gv + 18}$,
\item\label{append8vii}
$[1, a_{2s} - 1, a_{2s - 1}, \dots, a_1, 8] = \frac{9g^2 - gW - 1}{9gW-gV}$,
\item\label{append8viii}
$[1, a_{2s} - 1, a_{2s - 1}, \dots, a_2, a_1 - 1, 1, 8] = \frac{9g^2 - gw + 1}{8gW + gV - 17}$.
\end{enumerate}
\end{proposition}

\begin{proof}
We first remark that (\ref{append8iii}), (\ref{append8iv}), (\ref{append8vii}), and (\ref{append8viii}) follow from
(\ref{append8i}), (\ref{append8ii}), (\ref{append8v}), and (\ref{append8vi}), respectively, by Lemma \ref{negativecontfrac}.

(\ref{append8ii})
By Lemma \ref{contfracevenodd}, we have
$$
[8, 1, a_{2s} - 1, a_{2s - 1}, \dots, a_1] = 8 + \cfrac{1}{\frac{g^2}{gW - 1}} = \frac{9g^2 - g w - 1}{g^2}.
$$
Moreover, we compute:
$$
g^2(9gw-gv) = -1 + (gw - 1)(9g^2 - gw - 1),
$$
hence $$
[a_1, \dots, a_{2s - 1}, a_{2s} - 1, 1, 8] = \frac{9g^2 - gw - 1}{9gw - gv}
$$
by equation \ref{reverseontfrac}. Cases (\ref{append8i}), (\ref{append8v}), (\ref{append8vi}) follow from similar computations.
\end{proof}

\begin{proposition}\label{juxtapose}
Let $(e, g, f)$ be a regular Markoff triple with $e > 1$ and $f > 2$.
and denote $e^2/(ew_e - 1) = [a_1, \dots, a_{2s}]$, $f^2/(f w_f - 1) = [b_1, \dots, b_{2t}]$
with $a_1, a_{2s}, b_1, b_{2t} > 1$ then:
$$
\frac{g^2}{g w_g - 1} = [a_{2s}, \dots, a_1, 8, 1, b_{2t} - 1, b_{2t - 1}, \dots, b_1].
$$
\end{proposition}

\begin{proof}
By the concatenation rule \ref{concatenationlaw}, it suffices to show that
\begin{align*}
g^2 & = K(8, a_1, \dots, a_{2s}) K(b_1, \dots, b_{2t}) + K(a_1, \dots, a_{2s}) K(b_1, \dots, b_{2t - 1}, b_{2t} - 1), \\
g w_g - 1 & = K(8, a_1, \dots, a_{2s - 1}) K(b_1, \dots, b_{2t}) + K(a_1, \dots, a_{2s - 1}) K(b_1, \dots, b_{2t - 1}, b_{2t} - 1).
\end{align*}
By Lemma \ref{contfracevenodd} and Proposition \ref{append8} (\ref{append8v}) these are equivalent to
\begin{align*}
g^2 & = (9e^2 - e(e - w_e) - 1)f^2 + e^2(f(f - w_f) - 1) \\
g w_g - 1 & = (8ew_e + ev_e - 1)f^2 + (e w_e + 1) (f(f - w_f) - 1)
\end{align*}
The following computation proves the first equality:
\begin{align*}
(9e^2 - e(e - w_e) - 1)f^2 + e^2(f(f - w_f) - 1) & = 9e^2f^2 - e^2 - f^2 - ef(e w_f - f w_e) \\
& = 9 e^2 f^2 - e^2 - f^2 - 3efG \quad \text{ by Lemma \ref{cofactorlemma} (\ref{cofactorlemmai})} \\
& = g^2 \quad \text{ by the Markoff equation}.
\end{align*}
For the second equality:
\begin{align*}
(8ew_e + ev_e - 1)f^2 & + (e w_e  + 1)(f^2-fw_f - 1) \\
&= 9ef^2w_e + ef^2v_e - ef w_e w_f - e w_e - f w_f - 1 \\
& = gw_g - 1 + ef (fv_e + 3(3fw_e - w_g) - w_ew_f) \quad \text{ by Lemma \ref{cofactorlemma} (\ref{cofactorlemmaiii})} \\
& = gw_g - 1 \quad \text{ by Lemma \ref{cofactorlemma} (\ref{cofactorlemmav})}.
\end{align*}
\end{proof}

By Proposition \ref{juxtapose}, the mutations $(e, F, g) \longleftarrow (e, g, f) \longrightarrow (g, E, f)$
for $e > 1$ and $f > 2$ behave as follows:
\begin{center}
	\footnotesize
	\tikzstyle{level 1}=[level distance = 2cm, sibling distance = 8cm]
	\begin{tikzpicture}[]
		\node[align=center] {$[a_1, \dots, a_{2s}]$, $[c_1, \dots, c_{2(s + t + 1)}]$, $[b_1, \dots, b_{2t}]$}
		child {node[align=center] at (0.5, 0.5) {$[a_1, \dots, a_{2s}]$\\
				$[a_{2s}, \dots, a_1, 8, 1, c_{2(r + s + 1)} - 1, c_{2(r + s) + 1}, \dots, c_1]$ \\
				$[c_1, \dots, c_{2(s + t + 1)}]$}}
		child {node[align=center]  at (-0.5, 0.5) {$[c_1, \dots, c_{2(s + t + 1)}]$ \\
				$[c_{2(r + s + 1)}, \dots, c_1, 8, 1, b_{2t} - 1, b_{2t - 1}, \dots, b_1]$ \\
				$[b_1, \dots, b_{2t}]$}};
	\end{tikzpicture}
\end{center}

Now it remains to establish the initial data. For this and later use, we introduce
the following notation.

\begin{definition}
Let $a_0, \dots, a_{m - 1}$ be a sequence of integers.
\begin{enumerate}[(i)]
\item 
For $n > 0$ we denote $(a_0, \dots, a_{m - 1})_n$ the sequence $b_0, \dots, b_{n - 1}$ such that $b_k = a_{k \bmod m}$.
\item For any integer $x$, we denote $x \overset{k}{\curvearrowright} a_0, \dots, a_{m - 1} = a_0, \dots, a_{k - 2}, x, a_{k - 1},
\dots, a_{m - 1}$.
\end{enumerate}
\end{definition}

\begin{example}
The expression $[3 \overset{5}{\curvearrowright} 6, (1,5)_5,6]$ denotes the continued fraction
$[6, 1, 5, 1, 3, 5, 1, 6] = 89^2/(89 \cdot 13 - 1)$. Similarly, we write $[6 \overset{4}{\curvearrowright} (4, 8)_7]$
for $[4,8,4,6,8,4,8,4] = 985^2/(985 \cdot 239 - 1)$.
\end{example}

\begin{proposition}\label{fibpellcontfrac}\ 
\begin{enumerate}[(i)]
\item\label{fibpellcontfraci}
Consider a Markoff triple of the form $(1, F_{2n + 1}, F_{2n - 1})$ for $n \geq 3$ with T-weights $(-1, F_{2n - 3}, F_{2n - 5})$, then
$$
\frac{F_{2n + 1}^2}{F_{2n + 1}F_{2n - 3} - 1} = [3 \overset{n}{\curvearrowright} 6, (1,5)_{2n - 5},6]
= [3 \overset{n}{\curvearrowright} 6, (1,5)_{2n - 3}].
$$
\item\label{fibpellcontfracii}
Consider a Markoff triple of the form $(P_{2n - 1}, P_{2n + 1}, 2)$ for $n \geq 1$ with T-weights $(S_{2n - 3}, S_{2n - 1}, 1)$, then
$$
\frac{P_{2n + 1}^2}{P_{2n + 1}S_{2n - 1} - 1} = [6 \overset{n}{\curvearrowright} (4, 8)_{2n - 1}].
$$
\end{enumerate}
\end{proposition}

\begin{proof}
Recall that the T-weights and T-coweights for the Fibonacci and Pell branches have been described
in Lemma \ref{fibonaccipellweights}.

(\ref{fibpellcontfraci})
For the first triple $(1, 13, 5)$ with T-weights $(-1, 2, 1)$ we compute:
$$
\frac{5^2}{5 \cdot 1 - 1} = [6, 4], \quad \frac{13^2}{13 \cdot 2 - 1} = [6,1,3,6] = [3\overset{3}{\curvearrowright} 6, (1,5)_1, 6].
$$
Now assume for $n > 3$ that $F_{2n - 1}^2/(F_{2n - 1}F_{2n - 5} - 1) = [a_1, \dots, a_{2s}]
= [3\overset{n - 1}{\curvearrowright} 6, (1,5)_{2n - 7}, 6]$. We claim that
$$
\frac{F_{2n + 1}^2}{F_{2n + 1}F_{2n - 3} - 1} = [6, 1, a_{2s} - 1, a_{2s - 1}, \dots, a_1] = [3 \overset{n}{\curvearrowright} 6, (1,5)_{2n - 5},6].
$$
By the concatenation rule Proposition \ref{concatenationlaw}, it suffices to show
\begin{align*}
F_{2n + 1}^2 & = K(6)K(1, a_{2s} - 1, a_{2s - 1}, \dots, a_1) + K(a_{2s} - 1, a_{2s - 1}, \dots, a_1) \\
F_{2n + 1}F_{2n - 3} - 1 & = K(1, a_{2s} - 1, a_{2s - 1}, \dots, a_1)
\end{align*}
for $n > 3$. By Lemma \ref{contfracevenodd} these are equivalent to
\begin{align*}
F_{2n + 1}^2 &= 6F_{2n - 1}^2 + F_{2n - 1}(F_{2n - 1} - F_{2n - 5}) \\
F_{2n + 1}F_{2n - 3} - 1 & = F_{2n - 1}^2,
\end{align*}
both of which follow from standard properties of the Fibonacci numbers.

(\ref{fibpellcontfracii})
For the first triple $(5, 29, 2)$ with T-weights $(1, 7, 1)$ we compute::
$$
\frac{5^2}{5 \cdot 1 - 1} = [6, 4] = [6 \overset{1}{\curvearrowright} (4, 8)_1], \quad
\frac{29^2}{29 \cdot 7 - 1} = [4, 6, 8, 4] = [6 \overset{2}{\curvearrowright} (4, 8)_3].
$$
We assume for $n > 2$ that $P_{2n - 1}^2/(P_{2n - 1}S_{2n - 3} - 1) = [a_1, \dots, a_{2n - 2}]
= [6 \overset{n - 1}{\curvearrowright} (4, 8)_{2n - 3}]$.
We claim that
$$
\frac{P_{2n + 1}^2}{P_{2n + 1}S_{2n - 1} - 1} = [a_{2n - 2}, \dots, a_1] = [6 \overset{n}{\curvearrowright} (4, 8)_{2n - 1}].
$$
Note that $P_{2n + 1}Q_{2n - 1} - 1$ and $P_{2n + 1}S_{2n - 1} - 1$ are inverse to each other modulo $P_{2n + 1}^2$.
It therefore is sufficient to show that
$$
\frac{P_{2n + 1}^2}{P_{2n + 1}Q_{2n - 1} - 1} = [1, 3, 8, a_1, \dots, a_{2n - 2}] = [6 \overset{n + 2}{\curvearrowright} 1, 3, 8, (4, 8)_{2n - 3}].
$$

Invoking again the concatenation rule, it suffices to show that
\begin{align*}
P_{2n + 1}^2 = K(1, 3, 8) K(a_1, \dots, a_{2n - 2}) + K(1, 3)K(a_1, \dots, a_{2n - 3})
= 33 P_{2n - 1}^2 + 4 (P_{2n - 1} S_{2n- 3} - 1), \\
P_{2n + 1}Q_{2n - 1} - 1 = K(3, 8) K(a_1, \dots, a_{2n - 2}) + K(3) K(a_2, \dots, a_{2n - 2})
= 25 P_{2n - 1}^2 + 3(P_{2n - 1} S_{2n - 3} - 1)
\end{align*}
for $n > 2$. These equations follow from successive application of well-known relations among the Pell and Pell-Lucas
numbers, plenty of which can be found, e.g., in \cite{koshy14}, which we leave as an exercise for the reader.
\end{proof}

\begin{corollary}\label{fibpellmutcontfrac}
\begin{enumerate}[(i)]
\item
Consider a Markoff triple of the form $(F_{2n + 1}, g, F_{2n - 1})$ (with $g =$ $3F_{2n - 1}F_{2n + 1} - 1$), then
$$
\frac{g^2}{gw_g - 1} = [3 \overset{n - 1}{\curvearrowright} 6, (1,5)_{2n - 5}, 6, 8, [3 \overset{n - 1}{\curvearrowright} (1, 5)_{2n - 5}, 6]] 
= [3 \overset{3n - 2}{\curvearrowright} 3 \overset{n - 1}{\curvearrowright} 6, (1,5)_{2n - 5}, 6, 8, (1, 5)_{2n - 5}, 6].
$$
\item
Consider a Markoff triple of the form $(P_{2n - 1}, g, P_{2n + 1})$ (with $g = 3P_{2n - 1}P_{2n + 1} - 2$), then
\begin{equation*}
\frac{g^2}{gw_g - 1} = [6 \overset{n}{\curvearrowright} (4, 8)_{2n - 2}, [6 \overset{n + 2}{\curvearrowright} 1, 3, (8, 4)_{2n - 2}]]
= [6 \overset{3n + 1}{\curvearrowright} 6 \overset{n}{\curvearrowright} (4,8)_{2n - 2}, 1, 3, (8, 4)_{2n - 2}].
\end{equation*}
\end{enumerate}
\end{corollary}

\begin{proof}
Follows from Proposition \ref{fibpellcontfrac} and direct computation.
\end{proof}

Here some small observations on the structure of the continued fractions:

\begin{corollary}
Let $(e, g, f)$ be a Markoff triple such that $g^2/(gw - 1) = [a_1, \dots, a_{2s}]$ with $a_1, a_{2s} > 1$. Then
\begin{enumerate}[(i)]
\item
$a_i \in \{1, 3, 4, 5, 6, 8\}$ for every $i$.
\item
$a_i \neq a_{i + 1}$ for every $i$.
\item
A triple $(e, g, f)$ with $g > 13$ is a descendant of the Fibonacci branch iff $a_i = 5$ for some $i$ iff
$a_i \in \{1, 3, 5, 6, 8\}$ for every $i$.
\item
A triple $(e, g, f)$ is a descendant of the Pell branch iff $a_i = 4$ for some $i$ iff
$a_i \in \{1, 3, 4, 6, 8\}$ for every $i$.
\end{enumerate}
\end{corollary}

Figure \ref{quadraticcfes} shows the continued fraction expansions of $g / (gw_g - 1)$ for the first five
levels of the Markoff tree.

\begin{figure}[ht]
	\tikzstyle{level 1}=[level distance=1cm, sibling distance=6cm]
	\tikzstyle{level 2}=[level distance=1cm, sibling distance=3cm]
	\tikzstyle{level 3}=[level distance=3cm, sibling distance=1.5cm]
	\tikzstyle{level 4}=[level distance=6cm, sibling distance=1cm]
	\begin{tikzpicture}[grow=right, sloped]
		\node[align=left] {$[6, 4]$}
		child {node[align=left]{$[6, 1, 3, 6]$}
			child {node[align=left] {$[6, 1, 5, 3, 1, 6]$}
				child {node[align=left] {$[6, 1, 5, 1, 3, 5, 1, 6]$}
					child {node[align=left] {$[6, 1, 5, 1, 5, 3, 1, 5, 1, 6]$}}
					child {node[align=left] {$[6, 1, 5, 3, 1, 5, 1, 6, 8, 1, 5, 1, 3, 5, 1, 6]$}}
				}
				child {node[align=left] {$[6, 1, 3, 5, 1, 6, 8, 1, 5, 3, 1, 6]$}
					child {node[align=left] {$[6, 1, 3, 5, 1, 6, 8, 1, 5, 1, 3, 5, 1, 8, 6, 1, 5, 3, 1, 6]$}}
					child {node[align=left] {$[6, 1, 3, 5, 1, 8, 6, 1, 5, 3, 1, 6, 8, 1, 5, 3, 1, 6]$}}
				}
			}
			child {node[align=left] {$[6, 3, 1, 6, 8, 1, 3, 6]$}
				child {node[align=left] {$[6, 3, 1, 6, 8, 1, 5, 3, 1, 8, 6, 1, 3, 6]$}
					child {node[align=left] {$[6, 3, 1, 6, 8, 1, 5, 3, 1, 6, 8, 1, 3, 5, 1, 8, 6, 1, 3, 6]$}}
					child {node[align=left] {$[6, 3, 1, 6, 8, 1, 3, 5, 1, 8, 6, 1, 3, 6, 8, 1, 5, 3, 1, 8, 6, 1, 3, 6]$}}
				}
				child {node[align=left] {$[6, 3, 1, 8, 6, 1, 3, 6, 8, 1, 3, 6]$}
					child {node[align=left] {$[6, 3, 1, 8, 6, 1, 3, 6, 8, 1, 5, 3, 1, 8, 6, 3, 1, 6, 8, 1, 3, 6]$}}
					child {node[align=left] {$[6, 3, 1, 8, 6, 3, 1, 6, 8, 1, 3, 6, 8, 1, 3, 6]$}}
				}
			}
		}
		child {node[align=left] {$[4, 6, 8, 4]$}
			child {node[align=left] {$[4, 6, 8, 1, 3, 8, 6, 4]$}
				child {node[align=left] {$[4, 6, 8, 1, 3, 6, 8, 3, 1, 8, 6, 4]$}
					child {node[align=left] {$[4, 6, 8, 1, 3, 6, 8, 1, 3, 8, 6, 3, 1, 8, 6, 4]$}}
					child {node[align=left] {$[4, 6, 8, 1, 3, 8, 6, 3, 1, 8, 6, 4, 8, 1, 3, 6, 8, 3, 1, 8, 6, 4]$}}
				}
				child {node[align=left] {$[4, 6, 8, 3, 1, 8, 6, 4, 8, 1, 3, 8, 6, 4]$}
					child {node[align=left] {$[4, 6, 8, 3, 1, 8, 6, 4, 8, 1, 3, 6, 8, 3, 1, 8, 4, 6, 8, 1, 3, 8, 6, 4]$}}
					child {node[align=left] {$[4, 6, 8, 3, 1, 8, 4, 6, 8, 1, 3, 8, 6, 4, 8, 1, 3, 8, 6, 4]$}}
				}
			}
			child {node[align=left] {$[4, 8, 6, 4, 8, 4]$}
				child {node[align=left] {$[4, 8, 6, 4, 8, 1, 3, 8, 4, 6, 8, 4]$}
					child {node[align=left] {$[4, 8, 6, 4, 8, 1, 3, 8, 6, 4, 8, 3, 1, 8, 4, 6, 8, 4]$}}
					child {node[align=left] {$[4, 8, 6, 4, 8, 3, 1, 8, 4, 6, 8, 4, 8, 1, 3, 8, 4, 6, 8, 4]$}}
				}
				child {node[align=left] {$[4, 8, 4, 6, 8, 4, 8, 4]$}
					child {node[align=left] {$[4, 8, 4, 6, 8, 4, 8, 1, 3, 8, 4, 8, 6, 4, 8, 4]$}}
					child {node[align=left] {$[4, 8, 4, 8, 6, 4, 8, 4, 8, 4]$}}
				}
			}
		};
	\end{tikzpicture}
	\caption{Continued fraction expansions of $g^2 / (gw_g - 1)$ for the first four levels of the Markoff tree.}\label{quadraticcfes}
\end{figure}
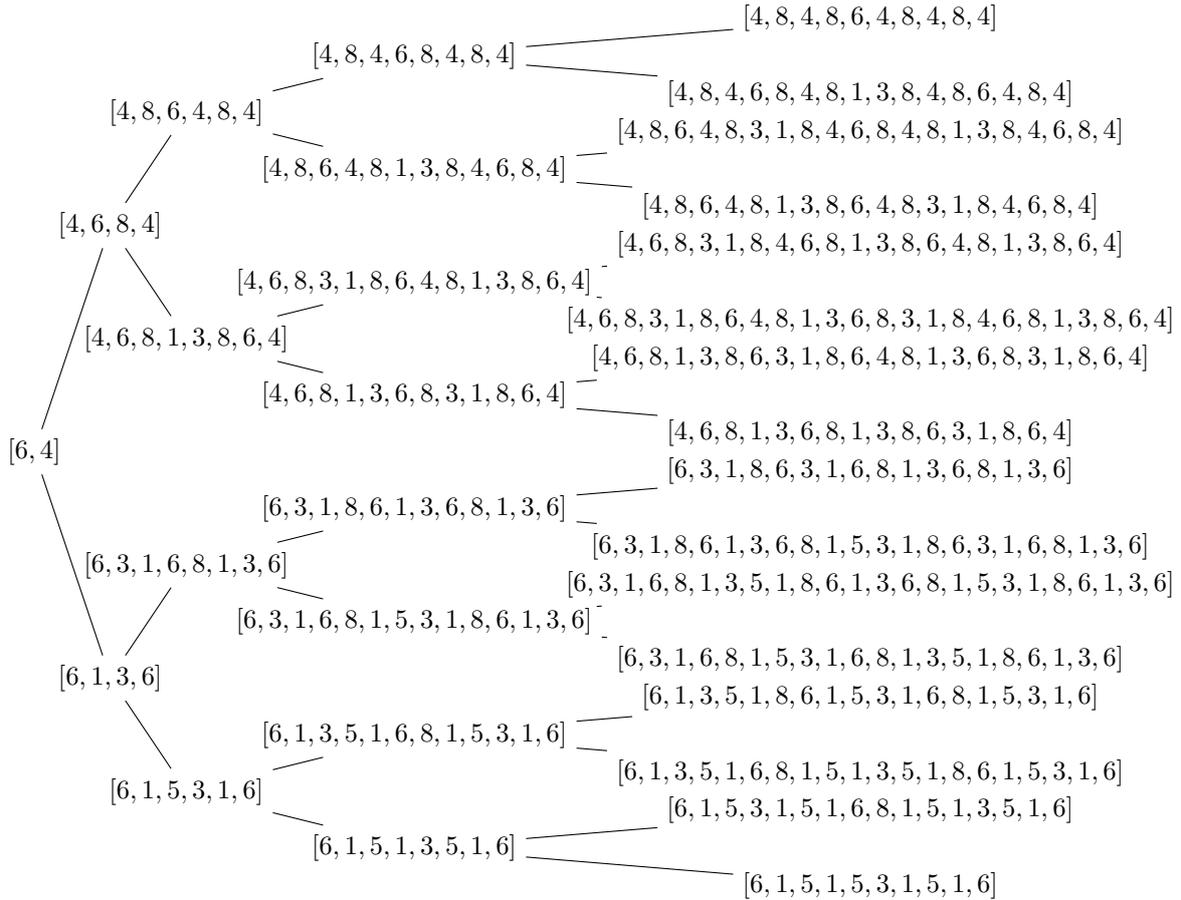

\begin{figure}[ht]
	\tikzstyle{level 1}=[level distance=1cm, sibling distance=4cm]
	\tikzstyle{level 2}=[level distance=1cm, sibling distance=2cm]
	\tikzstyle{level 3}=[level distance=3cm, sibling distance=1cm]
	\tikzstyle{level 4}=[level distance=6cm, sibling distance=.5cm]
	\begin{tikzpicture}[grow=right, sloped]
		\node[align=left] {$[3]$}
		child {node[align=left]{$[1, 5]$}
			child {node[align=left] {$[3, 1, 5]$}
				child {node[align=left] {$[1, 5, 1, 5]$}
					child {node[align=left] {$[3, 1, 5, 1, 5]$}}
					child {node[align=left] {$[6, 1, 5, 1, 3, 5, 1, 5]$}}
				}
				child {node[align=left] {$[6, 1, 5, 3, 1, 5]$}
					child {node[align=left] {$[1, 5, 1, 8, 6, 1, 5, 3, 1, 5]$}}
					child {node[align=left] {$[3, 1, 6, 8, 1, 5, 3, 1, 5]$}}
				}
			}
			child {node[align=left] {$[6, 1, 3, 5]$}
				child {node[align=left] {$[3, 1, 8, 6, 1, 3, 5]$}
					child {node[align=left] {$[6, 1, 3, 5, 1, 8, 6, 1, 3, 5]$}}
					child {node[align=left] {$[1, 6, 8, 1, 5, 3, 1, 8, 6, 1, 3, 5]$}}
				}
				child {node[align=left] {$[1, 6, 8, 1, 3, 5]$}
					child {node[align=left] {$[3, 1, 8, 6, 3, 1, 6, 8, 1, 3, 5]$}}
					child {node[align=left] {$[6, 1, 3, 6, 8, 1, 3, 5]$}}
				}
			}
		}
		child {node[align=left] {$[6, 3]$}
			child {node[align=left] {$[1, 8, 6, 3]$}
				child {node[align=left] {$[6, 3, 1, 8, 6, 3]$}
					child {node[align=left] {$[1,8,6,3,1,8,6,3]$}}
					child {node[align=left] {$[4, 8, 1, 3, 6, 8, 3, 1, 8, 6, 3]$}}
				}
				child {node[align=left] {$[4, 8, 1, 3, 8, 6, 3]$}
					child {node[align=left] {$[6, 3, 1, 8, 4, 6, 8, 1, 3, 8, 6, 3]$}}
					child {node[align=left] {$[1,8,6,4,8,1,3,8,6,3]$}}
				}
			}
			child {node[align=left] {$[4, 8, 3]$}
				child {node[align=left] {$[1, 8, 4, 6, 8, 3]$}
					child {node[align=left] {$[4, 8, 3, 1, 8, 4, 6, 8, 3]$}}
					child {node[align=left] {$[6, 4, 8, 1, 3, 8, 4, 6, 8, 3]$}}
				}
				child {node[align=left] {$[6, 4, 8, 3]$}
					child {node[align=left] {$[1, 8, 4, 8, 6, 4, 8, 3]$}}
					child {node[align=left] {$[4, 8, 4, 8, 3]$}}
				}
			}
		};
	\end{tikzpicture}
	\caption{The LEs for minimal resolutions of $\frac{1}{g^2}(1, g w_g - 1)$ for the first four levels of the Markoff tree.}\label{les}
\end{figure}
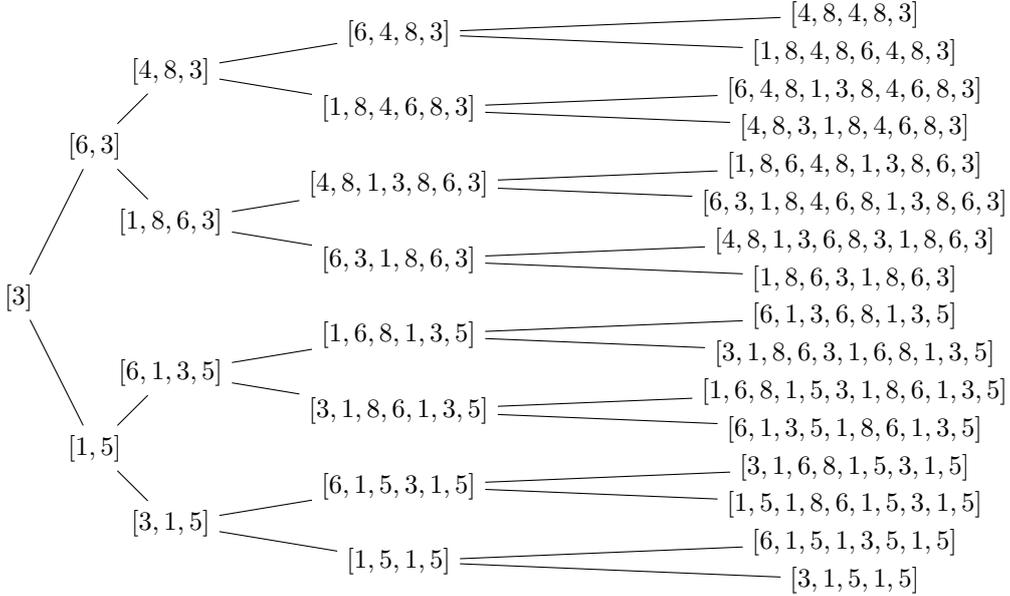

We can now give a complete description of the LEs correponding to Markoff triples.

\begin{theorem}\label{LEProposition}
	Let $(e, g, f)$ be a regular Markoff triple. Then the LE of the minimal resolution of the corresponding T-singularity
	$\frac{1}{g^2}(1, g w_g - 1)$ is given by
	\begin{align*}
		& a_s - 2, a_{s - 1}, \dots, a_2, a_1 - 1 = a_{s + 1}, \dots, a_{2s - 1}, a_{2s} - 1 & \text{ if $s$ is odd}, \\
		& a_s, \dots, a_2, a_1 - 1 = a_{s + 1} + 2, a_{s + 2}, \dots, a_{2s - 1}, a_{2s} - 1 & \text{ if $s$ is even}.
	\end{align*}
	where $[a_1, \dots, a_{2s}] = g^2/(gw_g - 1)$ with $a_1, a_{2s} > 1$
\end{theorem}

\begin{proof}
This is a direct consequence of Corollary \ref{tcontfrac} and Proposition \ref{palisplit}.
\end{proof}

For the Fibonacci and Pell branches, we obtain the following direct descriptions:

\begin{lemma}
\begin{enumerate}[(i)]
\item Consider a Markoff triple of the form $(1, F_{2n + 1}, F_{2n - 1})$ for $n \geq 3$ with weights $(-1, F_{2n - 3}, F_{2n - 5})$. Then
the LE of $F_{2n + 1}^2/(F_{2n + 1}F_{2n - 3} - 1)$ is given by
\begin{align*}
& 3, (1, 5)_{n - 2} & \text{ if $n$ is even}, \\
& (1, 5)_{n - 1} & \text{ if $n$ is odd}.
\end{align*}
\item Consider a Markoff triple of the form $(P_{2n - 1}, P_{2n + 1}, 2)$ for $n \geq 2$ with weights $(Q_{2n - 3}, Q_{2n - 1}, 1)$. Then
the LE of $P_{2n + 1}^2/(P_{2n + 1}Q_{2n - 1} - 1)$ is given by
\begin{align*}
& 6, (4, 8)_{n - 2}, 3 & \text{ if $n$ is even}, \\
& (4, 8)_{n - 1}, 3 & \text{ if $n$ is odd}.
\end{align*}
\end{enumerate}
\end{lemma}

Figure \ref{les} shows the LEs for the first four levels of the Markov trees. The following Corollary shows that the
LEs at least restore some of the symmetry that we have observed for Frobenius' description of continued fractions
in Section \ref{frobenius}.

\begin{corollary}
Similar as in \ref{complement}, we denote $(g_{\mu/\nu}, w_{g_{\mu/\nu}})$ and $(g_{\nu/\mu}, w_{g_{\nu/\mu}})$
two Markoff numbers $> 5$ that sit at opposite positions in the Markoff tree together with their T-weights, and let
$[a_1, \dots, a_s]$ and $[b_1, \dots, b_t]$ be their corresponding LEs. Then:
$$
s = t \quad \text{ and } \quad [a_1 + b_1, \dots a_s + b_s] = [7, 9, \dots, 9, 8].
$$
\end{corollary}

To conclude, we have now a look at fractions $w_g / g$ and their associated Cantor set. As in Section \ref{cantor}
we denote $p$ a finite path and $(e, g, f)$ the Markoff triple corresponding to that path. Denote $[a_1, \dots, a_{2s}]$
the continued fraction expansion of $g^2 / (gw_g - 1)$. Then by the mutation rules and induction we get for the
paths $p_1 = pL\bar{R}$ and $p_2 = pR\bar{L}$:
\begin{align*}
\mathfrak{l}_{p_1} & = [\overline{a_1, \dots, a_{2s - 1}, a_{2s} - 1, 1, 8}], \\
\mathfrak{l}_{p_2} & = 1 + [\overline{a_{2s} - 1, a_{2s - 1}, \dots, a_1, 8, 1}],
\end{align*} 
where the $\mathfrak{l}_i$ denote the limits of the fractions $g / w_g$ along the paths $p_i$.
By direct computation we obtain:
\begin{align*}
	\mathfrak{l}_{p_1} & = \frac{1}{2} \frac{9g-2w_g+3\sqrt{9g^2-4}}{9w_g-v_g}, \\
	\mathfrak{l}_{p_2} & = \frac{1}{2} \frac{9g + 2 w_g + 3 \sqrt{9g^2-4}}{9w_g + v_g}.
\end{align*} 
Figure \ref{quadraticfig2} shows the purely periodic values $\mathfrak{l}_{p_1}$ and
$\mathfrak{l}_{p_2} - 1$ together with their periods for the first four levels of the Markoff
tree. We may compare this spectrum that of Section \ref{cantor}, in particular Figure \ref{quadraticfig1}
and observe the symmetry between opposite positions of the tree such that the periods pointwise add up to 9.

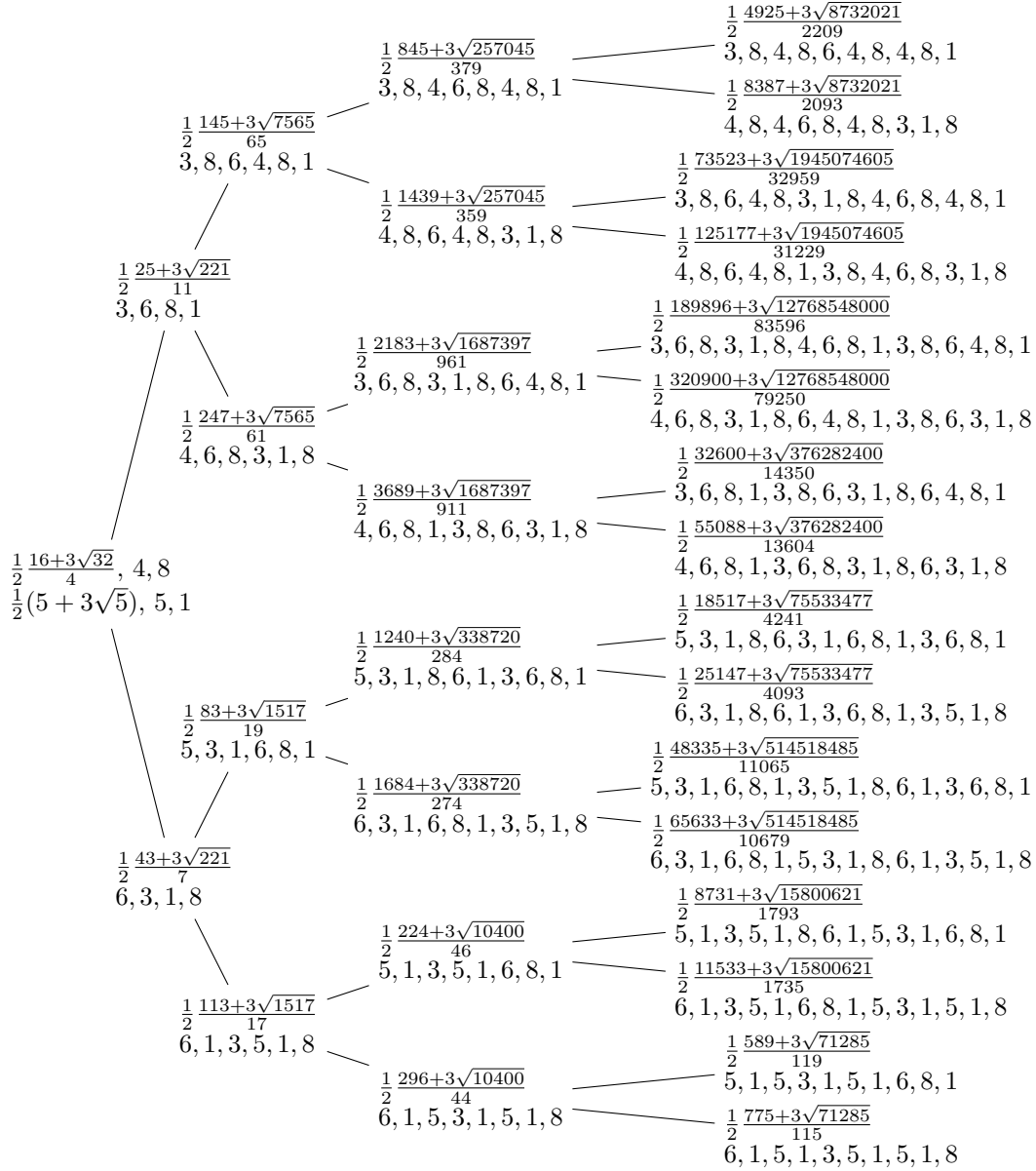
\begin{figure}[ht]
\tikzstyle{level 1}=[level distance=1cm, sibling distance=8cm]
\tikzstyle{level 2}=[level distance=1cm, sibling distance=4cm]
\tikzstyle{level 3}=[level distance=3cm, sibling distance=2cm]
\tikzstyle{level 4}=[level distance=5cm, sibling distance=1cm]
\begin{tikzpicture}[grow=right, sloped]
  \node[align=left] {$\frac{1}{2}\frac{16+3\sqrt{32}}{4}$, $4, 8$\\ $\frac{1}{2}(5+3\sqrt{5})$, $5, 1$}
    child {node[align=left]{$\frac{1}{2}\frac{43+3\sqrt{221}}{7}$\\ $6, 3, 1, 8$}
      child {node[align=left] {$\frac{1}{2}\frac{113+3\sqrt{1517}}{17}$\\ $6, 1, 3, 5, 1, 8$}
        child {node[align=left] {$\frac{1}{2}\frac{296+3\sqrt{10400}}{44}$\\ $6, 1, 5, 3, 1, 5, 1, 8$}
           child {node[align=left] {$\frac{1}{2}\frac{775+3\sqrt{71285}}{115}$\\ $6, 1, 5, 1, 3, 5, 1, 5, 1, 8$}}
           child {node[align=left] {$\frac{1}{2}\frac{589+3\sqrt{71285}}{119}$\\ $5, 1, 5, 3, 1, 5, 1, 6, 8, 1$}}
        }
        child {node[align=left] { $\frac{1}{2}\frac{224+3\sqrt{10400}}{46}$\\ $5, 1, 3, 5, 1, 6, 8, 1$}
           child {node[align=left] {$\frac{1}{2}\frac{11533+3\sqrt{15800621}}{1735}$\\ $6, 1, 3, 5, 1, 6, 8, 1, 5, 3, 1, 5, 1, 8$}}
           child {node[align=left] {$\frac{1}{2}\frac{8731+3\sqrt{15800621}}{1793}$\\ $5, 1, 3, 5, 1, 8, 6, 1, 5, 3, 1, 6, 8, 1$}}
        }
      }
      child {node[align=left] {$\frac{1}{2}\frac{83+3\sqrt{1517}}{19}$\\ $5, 3, 1, 6, 8, 1$}
        child {node[align=left] { $\frac{1}{2}\frac{1684+3\sqrt{338720}}{274}$\\ $6, 3, 1, 6, 8, 1, 3, 5, 1, 8$}
         child {node[align=left] { $\frac{1}{2}\frac{65633+3\sqrt{514518485}}{10679}$\\ $6, 3, 1, 6, 8, 1, 5, 3, 1, 8, 6, 1, 3, 5, 1, 8$}}
         child {node[align=left] {$\frac{1}{2}\frac{48335+3\sqrt{514518485}}{11065}$\\ $5, 3, 1, 6, 8, 1, 3, 5, 1, 8, 6, 1, 3, 6, 8, 1$}}
        }
        child {node[align=left] {$\frac{1}{2}\frac{1240+3\sqrt{338720}}{284}$\\ $5, 3, 1, 8, 6, 1, 3, 6, 8, 1$}
         child {node[align=left] { $\frac{1}{2}\frac{25147+3\sqrt{75533477}}{4093}$\\ $6, 3, 1, 8, 6, 1, 3, 6, 8, 1, 3, 5, 1, 8$}}
         child {node[align=left] {$\frac{1}{2}\frac{18517+3\sqrt{75533477}}{4241}$\\ $5, 3, 1, 8, 6, 3, 1, 6, 8, 1, 3, 6, 8, 1$}}
        }
      }
    }
    child {node[align=left] {$\frac{1}{2}\frac{25+3\sqrt{221}}{11}$ \\ $3, 6, 8, 1$}
      child {node[align=left] {$\frac{1}{2}\frac{247+3\sqrt{7565}}{61}$ \\ $4, 6, 8, 3, 1, 8$}
        child {node[align=left] { $\frac{1}{2}\frac{3689+3\sqrt{1687397}}{911}$\\ $4, 6, 8, 1, 3, 8, 6, 3, 1, 8$}
         child {node[align=left] { $\frac{1}{2}\frac{55088+3\sqrt{376282400}}{13604}$\\ $4, 6, 8, 1, 3, 6, 8, 3, 1, 8, 6, 3, 1, 8$}}
         child {node[align=left] {$\frac{1}{2}\frac{32600+3\sqrt{376282400}}{14350}$\\ $3, 6, 8, 1, 3, 8, 6, 3, 1, 8, 6, 4, 8, 1$}}
        }
        child {node[align=left] {$\frac{1}{2}\frac{2183+3\sqrt{1687397}}{961}$\\ $3, 6, 8, 3, 1, 8, 6, 4, 8, 1$}
         child {node[align=left] { $\frac{1}{2}\frac{320900+3\sqrt{12768548000}}{79250}$\\ $4, 6, 8, 3, 1, 8, 6, 4, 8, 1, 3, 8, 6, 3, 1, 8$}}
         child {node[align=left] {$\frac{1}{2}\frac{189896+3\sqrt{12768548000}}{83596}$\\ $3, 6, 8, 3, 1, 8, 4, 6, 8, 1, 3, 8, 6, 4, 8, 1$}}
        }
      }
      child {node[align=left] {$\frac{1}{2}\frac{145+3\sqrt{7565}}{65}$ \\ $3, 8, 6, 4, 8, 1$}
        child {node[align=left] {$\frac{1}{2}\frac{1439+3\sqrt{257045}}{359}$\\ $4, 8, 6, 4, 8, 3, 1, 8$}
          child {node[align=left] {$\frac{1}{2}\frac{125177+3\sqrt{1945074605}}{31229}$\\ $4, 8, 6, 4, 8, 1, 3, 8, 4, 6, 8, 3, 1, 8$}}
          child {node[align=left] { $\frac{1}{2}\frac{73523+3\sqrt{1945074605}}{32959}$\\ $3, 8, 6, 4, 8, 3, 1, 8, 4, 6, 8, 4, 8, 1$}}
        }
        child {node[align=left] { $\frac{1}{2}\frac{845+3\sqrt{257045}}{379}$\\ $3, 8, 4, 6, 8, 4, 8, 1$}
          child {node[align=left] {$\frac{1}{2}\frac{8387+3\sqrt{8732021}}{2093}$\\ $4, 8, 4, 6, 8, 4, 8, 3, 1, 8$}}
          child {node[align=left] { $\frac{1}{2}\frac{4925+3\sqrt{8732021}}{2209}$\\ $3, 8, 4, 8, 6, 4, 8, 4, 8, 1$}}
        }
      }
    };
\end{tikzpicture}
\caption{The quadratic T-spectrum}\label{quadraticfig2}
\end{figure}

Proceeding as in Section 4, we set
$$
A_g := 1 /  \mathfrak{l}_{p_1}, \quad B_g := 1 / \mathfrak{l}_{p_2}, \quad I_g = [B_e, A_f], \quad J_g := (A_g, B_g).
$$
This way, obtain a Cantor set $\mathfrak{T}$ by using the intervals $I_g$, $J_g$, where
\begin{align*}
	|I_g| & = \frac{3}{ef}(\Delta_{e, f} - g), \\
	|J_g| & = \frac{3}{g}(3g - \sqrt{\Delta_g}).
\end{align*}

\begin{theorem}\label{tcantor}
$\mathfrak{T}$ is a Cantor set which is given by the complement of the open intervals
$$
\left(\frac{w_g}{g} - \frac{3}{2}\left(3 - \sqrt{9 - \frac{4}{g^2}}\right), \frac{w_g}{g} + \frac{3}{2}\left(3 - \sqrt{9 - \frac{4}{g^2}}\right) \right)
$$
within the closed interval $[(7 - 3\sqrt{5})/2, 3\sqrt{2} - 4] \approx [0.145898, 0.242641]$, where $g$ runs over the
maxi\-mal elements of all regular Markoff triples. Lebesgue measure and Hausdorff dimension of $\mathfrak{T}$ both are zero.
\end{theorem}

\begin{proof}
It remains to show the assertions about measure and dimension. For this, we just observe that by Corollary \ref{affinetransform},
$\mathfrak{T}$ is an affine transformation
of $\mathfrak{M}$, given by $\mathfrak{T} = 3 \mathfrak{M} - 1$. Then we conclude with
Theorems \ref{lebesguetheorem} and \ref{hausdorfftheorem}.
\end{proof}

\numberwithin{equation}{section}
\renewcommand{\theequation}{\thesection-\arabic{equation}}

\appendix

\section{Continuants}\label{continuants}

In this appendix we state some general facts about continuants and continued fractions. For detailed expositions
we refer to \cite[\S 6.7]{GrahamKnuthPatashnik} and \cite[Ch I --- III]{Perron1}. Any sequence of real numbers
$a_1, \dots, a_n$ such that the continued fraction $[a_1, \dots, a_t] =
a_1 + 1/(a_2 + \cdots (a_{n - 1} + 1 / a_n) \cdots )$ exists can be written as a quotient
$$
\frac{K_n(a_1, \dots, a_n)}{K_{n - 1}(a_2, \dots, a_n)}
$$
where for any $t \geq 0$ and indeterminates $x_1, \dots, x_t$, $K_t(x_1, \dots, x_t)$ is the $t$-th {\em continuant} polynomial.
These can inductively be defined as follows:
\begin{align*}
K_0 & = 1, \\
K_1(x_1) & = x_1, \\
K_i(x_1, \dots, x_i) & = x_i K_{i - 1}(x_1, \dots, x_{i - 1}) + K_{i - 2}(x_1, \dots, x_{i - 2}) \text{ for } i > 1.
\end{align*}
For $t \leq 5$, we therefore get:
\begin{align*}
K_2(x_1, x_2) & = x_1 x_2 + 1, \\
K_3(x_1, x_2, x_3) & = x_1 x_2 x_3 + x_1 + x_3, \\
K_4(x_1, x_2, x_3, x_4) & = x_1 x_2 x_3 x_4 + x_1 x_2 + x_1 x_4 + x_3 x_4 + 1, \\
K_5(x_1, x_2, x_3, x_4, x_5) & = x_1 x_2 x_3 x_4 x_5 + x_1 x_2 x_3 + x_1 x_2 x_5 + x_1 x_4 x_5 + x_3 x_4 x_5 + x_1 + x_3 + x_5.
\end{align*}
Euler's rule states that in order to determine the monomial terms of $K_t(x_1, \dots, x_t)$, one starts with
$x_1 \cdots x_t$ and adds all possible monomials obtained from successively dividing $x_1 \cdots x_t$ by factors of
the form $x_i x_{i + 1}$. This description immediately implies that
continuants are invariant under reversing the order of their arguments:
$$
K_t(x_1, \dots, x_t) = K_t(x_t, \dots, x_1).
$$
For $s, t \geq 0$ the following factorization rule holds:
\begin{gather}
K_{s + t}(x_1, \dots, x_s, x_{s + 1}, \dots, x_{s + t})
= K_s(x_1, \dots, x_s) K_t(x_{s + 1}, \dots, x_{s + t}) \nonumber \\
+ K_{s - 1}(x_1, \dots, x_{s - 1})K_{t - 1}(x_{s + 2}, \dots, x_{s + t})\label{continuantfactorization}.
\end{gather}
For any $1 \leq k \leq t$ we moreover have Cassini's identity:
\begin{gather}
K_{s + t}(x_1, \dots, x_{s + t}) K_k(x_{s + 1}, \dots, x_{s + k})
= K_{s + k}(x_1, \dots, x_{s + k}) K_t(x_{s + 1}, \dots, x_{s + t}) \nonumber \\
+ (-1)^k K_{s - 1}(x_1, \dots, x_{s - 1})K_{t - k - 1}(x_{s + k + 2}, \dots, x_{s + t}).\label{Cassini}
\end{gather}
For the particular case $s = 1$, $k = t - 1$, Cassini's identity yields:
\begin{equation}
K_{t + 1}(x_1, \dots, x_{t + 1})K_{t - 1}(x_2, \dots, x_t) = K_t(x_1, \dots, x_t) K_t(x_2, \dots, x_{t + 1}) + (-1)^{t - 1}.\label{Cassini2}
\end{equation}

Any rational number admits a finite continued fraction expansion $[a_1, \dots, a_n]$ with integers
$a_1, \dots, a_n$, where $a_i > 0$ for $i > 1$. This representation is almost unique with the exception
that we can absorb the last element if it is equal to $1$, i.e. $[a_1, \dots, a_n, 1] = [a_1, \dots, a_{n - 1}, a_n + 1]$.
In the case $a_1 > 0$ we can use Cassini's identity
to describe the relation between reverse continued fractions: let $[a_1, \dots, a_n] = K_n(a_1, \dots, a_n)/
K_{n - 1}(a_2, \dots, a_n) = p/q$, then $p > q$ and:
\begin{equation}\label{reverseontfrac}
[a_n, \dots, a_1] = \frac{K(a_n, \dots, a_1)}{K(a_{n - 1}, \dots, a_1)} = \frac{K(a_1, \dots, a_n)}{K(a_1, \dots, a_{n - 1})}
= \frac{p}{q'},
\end{equation}
where $0 < q' < p$ such that $q q' \equiv (-1)^{n - 1} \mod p$ (see \cite[\S 11]{Perron1}). In particular, if $n$ is odd then
$q' \equiv q^{-1} \mod p$. For lighter notation we will drop the indices of the $K$'s from now on.
The following lemma gives us some control over the term $(-1)^{n - 1}$.

\begin{lemma}\label{negativecontfrac}
Let $0 < q < p$ be coprime integers. If $q > q - p$ and $p / (p - q) = [a_1, \dots, a_n]$
then $p / q = [1, a_1 - 1, a_2, \dots, a_n]$. Conversely, if $p / q = [1, a_1, \dots, a_n]$ then $q > p - q$ and
$p / (p - q) = [a_1 + 1, a_2, \dots, a_n]$.
\end{lemma}

\begin{proof}
The assertions follow from
$$
\frac{p}{q} = 1 + \cfrac{1}{\frac{q}{p - q}} \quad \text{ and } \quad \left\lfloor \frac{q}{p - q}
\right\rfloor = \left\lfloor \frac{p}{p - q} \right\rfloor - 1
$$
for $q > p - q$.
\end{proof}

Assume now that we have integers $a_1, \dots, a_m$, $b_1, \dots, b_n$ corresponding to continued fractions
$$
\frac{P}{Q} = [a_1, \dots, a_m], \quad \frac{U}{V} = [b_1, \dots, b_n].
$$
By the factorization rule (\ref{continuantfactorization}), we have
\begin{align*}
K(a_1, \dots, a_m, b_1, \dots, b_n) & = K(a_1, \dots, a_m) K(b_1, \dots, b_n) + K(a_1, \dots, a_{m - 1}) K(b_2, \dots, b_n), \\
K(a_2, \dots, a_m, b_1, \dots, b_n) & = K(a_2, \dots, a_m) K(b_1, \dots, b_n) + K(a_2, \dots, a_{m - 1}) K(b_2, \dots, b_n).
\end{align*}
By the following criterion it is enough to show that these factorizations are sufficient in order to construct
the concatenation of continued fractions.

\begin{proposition}[Concatenation Rule]\label{concatenationlaw}
Let $P, Q, U, V, a_i, b_j$ be as before and denote
$$
X = P U + Q' V, \quad Y = Q U + Q'' V, 
$$
where $Q' = K(a_1, \dots, a_{m - 1})$ and $Q'' = K(a_2, \dots, a_{m - 1})$. Then:
$$
\frac{X}{Y} = [a_1, \dots, a_m, b_1, \dots, b_n].
$$
\end{proposition}

\begin{proof}
We write
$$
[a_1, \dots, a_m, b_1, \dots, b_n] = [a_1, \dots, a_m, \frac{U}{V}] = \frac{K(a_1, \dots, a_m, \frac{U}{V})}{K(a_2, \dots, a_m, \frac{U}{V})}
= \frac{\frac{U}{V}K(a_1, \dots, a_m) + K(a_1, \dots, a_{m - 1})}{\frac{U}{V}K(a_2, \dots, a_m) + K(a_2, \dots, a_{m - 1})}.
$$
Rearranging and comparing terms then yields the assertion.
\end{proof}

By a theorem of Lagrange, an irrational number $\xi$ is a quadratic irrational if and only if its corresponding infinite continued
fraction expansion becomes periodic, i.e. we have $\xi = a + b\sqrt{D} \in \Q[\sqrt{D}]$ for some non-square integer
$D$ if and only if
$$
\xi = [a_1, \dots, a_k, \overline{a_{k + 1}, \dots, a_n}]
$$
for integers $a_1, \dots, a_n$. A quadratic irrational $\xi$ has a {\em purely periodic} continued fraction expansion
$\xi = [\overline{a_1, \dots, a_n}]$ if and only if it is {\em reduced}, i.e. if $\xi > 1$ and $-1 < \bar{\xi} < 0$. Here,
$\bar{\xi} = a - b \sqrt{D}$ is the conjugate of $\xi$. From now we assume that $\xi = [\overline{a_1, \dots, a_n}]$ is a
reduced quadratic irrational number. Then, following \cite[\S 22]{Perron1}, we abbreviate for $1 \leq i \leq n$:
\begin{align*}
A_i = K(a_1, \dots, a_i), \quad B_i = K(a_2, \dots, a_i)
\end{align*}
and get:
\begin{align*}
\xi = [a_1, \dots, a_n, \xi] = \frac{K(a_1, \dots, a_n, \xi)}{K(a_2, \dots, a_n, \xi)} =
\frac{\xi A_n + A_{n - 1}}{\xi B_n + B_{n - 1}}.
\end{align*}
Solving the resulting quadratic equation $B_n \xi^2 + (B_{n - 1} - A_n) \xi - A_{n - 1} = 0$ yields:
\begin{equation}\label{makeperiodic}
\xi = \frac{1}{2B_n} \left(A_n - B_{n -1} + \sqrt{(A_n - B_{n - 1})^2 + 4 A_{n - 1} B_n}\right).
\end{equation}

Consider now a cyclic shift of $\xi$, i.e. let
$$
\sigma = [\overline{a_{k + 1}, \dots, a_n, a_1, \dots, a_k}].
$$
Then by the relation $\xi = [a_1, \dots, a_k, \sigma]$ and the defining relations for the continuants we get:
$$
\xi = \frac{K(a_1, \dots, a_k, \sigma)}{K(a_2, \dots, a_k, \sigma)} = \frac{\sigma A_k + A_{k - 1}}{\sigma B_k  + B_{k - 1}}
$$
and by resolving for $\sigma$ we obtain:
\begin{equation}\label{periodshift1}
\sigma = [\overline{a_{k + 1}, \dots, a_n, a_1, \dots, a_k}] = \frac{A_{k - 1} - \xi B_{k - 1}}{\xi B_k - A_k}.
\end{equation}
Alternatively, we can write $\sigma = [a_{k + 1}, \dots, a_n, \xi]$ and therefore obtain
\begin{equation}\label{periodshift2}
\sigma = [\overline{a_{k + 1}, \dots, a_n, a_1, \dots, a_k}] = \frac{K(a_{k + 1}, \dots, a_n, \xi)}{K(a_{k + 2}, \dots, a_n, \xi)}.
\end{equation}

The reverse period of $\xi$ corresponds to $\eta := -1/\bar{\xi} = [\overline{a_n, \dots, a_1}]$ (see \cite[\S 23]{Perron1}).
Then for $\tau = [\overline{a_k, \dots, a_1, a_n, \dots, a_{k + 1}}]$ we obtain with (\ref{periodshift1}):
\begin{equation}\label{reverseperiodshift}
\tau = [\overline{a_k, \dots, a_1, a_n, \dots, a_{k + 1}}] = \frac{K(a_{k + 2}, \dots, a_n) -
\eta K(a_{k + 2}, \dots, a_{n - 1})}{\eta K(a_{k + 1}, \dots a_{n - 1}) - K(a_{k + 1}, \dots, a_n)} =
-\frac{K(a_{k + 2}, \dots, a_n, \bar{\xi})}{K(a_{k + 1}, \dots, a_n, \bar{\xi})}.
\end{equation}

\section{T-continuants}\label{tcontinuants}

We introduce a variant of continuants that are related to Hirzebruch-Jung continued fractions of
the type of T-singularities of Section \ref{tsingsection}. We start by defining the following sets
\begin{align*}
J_0 & = \{\emptyset\}, \\
J_1 & = \{\emptyset\}, \\
J_m & = J_{m - 2} \cup \{X \cup \{m - 1\} \mid X \in J_{m - 1}\} \text{ for } m > 1, \\
I_m & = J_m \sqcup J_{m + 1} \text{ for } m \geq 0.
\end{align*}
By construction, the cardinality of $J_m$ equals
the ${m + 1}$-st Fibonacci number $F_{m + 1}$ and the cardinality of $I_m$ is $F_{m + 3}$.
Note that $J_m$ and $J_{m + 1}$ are are almost disjoint sets, having only the empty set in common;
the disjoint union $I_m = J_m \sqcup J_{m + 1}$ therefore contains $\emptyset$ twice.

The sets $J_m$ and $I_m$ can also be nicely described in a non-inductive way. For this, we consider
$1 < \cdots < m$ as an ordered set and a subsequence
$i_1 < i_2 < \cdots < i_k$ of $1 < \cdots < m$ is called {\em even-odd} if its elements are alternatingly even
or odd (i.e. $0$ or $1 \mod$ 2; the smallest element can be either even or odd --- we do not mean
to distinguish between
even-odd and odd-even sequences). More precisely, $i_j \equiv i_{j + 2} \mod 2$ for every $1 \leq j < k - 1$
and, if $k >1$, then $i_1 \not \equiv i_2 \mod 2$.

\begin{lemma}\label{evenoddlemma}
The non-empty sets in $J_{m + 1}$ and $I_m$ are even-odd subsequences of $1, \dots, m$, where:
\begin{enumerate}[(i)]
\item\label{evenoddlemmai}
$i_1 < \cdots < i_k$ is in $J_{m + 1}$ iff $i_1 \equiv k + m + 1 \mod 2$.
\item\label{evenoddlemmaii} 
$I_m$ contains all even-odd subsequences of $\{1, \dots, m\}$.
\end{enumerate}
\end{lemma}

\begin{proof}
(\ref{evenoddlemmai})
We do induction by $m$. The statement is obviously true for $m = 0$. For $m > 0$, we have by definition
$J_{m + 1} = J_{m - 1} \cup \{X \cup \{m\} \mid X \in J_m)\}$. Consider any even-odd subsequence
$i_1 < \cdots < i_k$ of $\{1, \dots, m\}$. If the sequence is contained in $\{1, \dots, m - 2\}$ then
it is in $J_{m - 1}$ (and therefore in $J_{m + 1}$) iff $i_1 \equiv k + m - 1 \equiv k + m + 1 \mod 2$.
If the sequence is contained in $J_m$, then it follows immediately from the induction hypothesis that
$i_1 \equiv k + m \mod 2$ and $i_k \not \equiv m \mod 2$, hence $i_1 < \cdots < i_k < m$ is even-odd
and $i_1 \equiv (k + 1) + m + 1 \mod 2$. Hence, any subset of $J_{m + 1}$ is an even-odd sequence and satisfies
$i_1 \equiv k + m + 1 \mod 2$. Now, if $i_1 < \cdots < i_k$ has these properties and is not contained in
$J_{m - 1}$ then, again, we use $i_k \not \equiv m \mod 2$ and therefore $i_k = m$. Invoking again the
induction hypothesis, we get that $i_1 < \cdots < i_{k - 1}$ is contained in $J_m$ and the assertion follows
for $m$.

(\ref{evenoddlemmaii})
Follows immediately from (\ref{evenoddlemmai}).
\end{proof}

\begin{definition}
Let $m \geq 0$ and $x_1, \dots, x_m$ indeterminates. Then we define the $m$-th {\em T-continuant} $T_m$ and the $m$-th
{\em T-semicontinuant} $S_m$ as:
$$
T_m = T_m(x_1, \dots, x_m) = \sum_{I \in I_m} x_I, \quad S_m = S_m(x_1, \dots, x_m) = \sum_{J \in J_m} x_J,
$$
where for any subset $I$ of $\{1, \dots, m\}$ we write $x_I$ for the monomial $\prod_{i \in I} x_i$ and in particular,
$x_\emptyset = 1$.
\end{definition}

Here are the first few polynomials $T_m$ and $S_m$:
\begin{align*}
S_0 & = 1, & T_0 & = 2, \\
S_1 & = 1, & T_1 & = 2 + x_1, \\
S_2 & = 1 + x_1, & T_2 & = 2 + x_1 + x_2 + x_1x_2, \\
S_3 & = 1 + x_2 + x_1x_2, & T_3 & = 2 + x_1 + x_2 + x_3 + x_1x_2 + x_2x_3 + x_1x_2x_3,
\end{align*}
\begin{align*}
S_4 = & \ 1 + x_1 + x_3 + x_2x_3 + x_1x_2x_3, \\
T_4 = & \ 2 + x_1 + x_2 + x_3 + x_4 + x_1x_2 + x_1x_4 + x_2x_3 + x_3x_4 + x_1x_2x_3 + x_2x_3x_4 + x_1x_2x_3x_4, \\
S_5 = & \ 1 + x_2 + x_4 + x_1x_2 + x_1x_4 + x_3x_4 + x_2x_3x_4 + x_1x_2x_3x_4, \\
T_5 = & \ 2 + x_1 + x_2 + x_3 + x_4 + x_5 + x_1x_2 + x_1x_4 + x_2x_3 + x_2x_5 + x_3x_4 + x_4x_5, \\
& + x_1x_2x_3 + x_1x_2x_5 + x_1x_4x_5 + x_2x_3x_4 
+ x_3x_4x_5 + x_1x_2x_3x_4 + x_2x_3x_4x_5 + x_1x_2x_3x_4x_5.
\end{align*}

We have seen in Corollary \ref{tcontfrac} that given the LE $c_1, \dots, c_m$ of a pair $(e, k)$, we have:
$$
\frac{e}{k} = [c_m + 1, c_{m - 1}, \dots, c_2, c_1 + 1] = \frac{K_m(c_m + 1, c_{m - 1}, \dots, c_2, c_1 + 1)}{K_{m - 1}(c_{m- 1},
\dots, c_2, c_1 + 1)}.
$$
In the following Lemma we will give a similar representation in terms of T-(semi-)continuants and explain the
relation between continuants and T-continuants.

\begin{lemma}\label{polynomiallemma}
Let $c_1, \dots, c_m$ be the LE of the pair $(e, k)$ with $k \leq  e - k$. Then:
\begin{enumerate}[(i)]
\item\label{polynomiallemmai}
$k = S_m(c_1, \dots, c_{m - 1})$ and $e = T_m(c_1, \dots, c_m) = S_m(c_1, \dots, c_{m - 1}) + S_{m + 1}(c_1, \dots,
c_m)$.
\item\label{polynomiallemmaiv}
For $m > 1$ it follows that $T_m(c_1, \dots, c_m) = K_m(c_m + 1, c_{m - 1}, \dots, c_2, c_1 + 1)$ and
$S_m =$ $K_{m - 1}(c_{m - 1},$ $\dots, c_2, c_1 + 1)$.
\item\label{polynomiallemmaii}
$T_m(c_1, \dots, c_m) = T_m(c_m, \dots, c_1)$.
\item\label{polynomiallemmaiii}
$c_m, \dots, c_1$ is the LE of $(e, k')$, where $k' = S_m(c_m, \dots, c_2)$ and moreover
$k k' \equiv (-1)^{m + 1} \mod e$.
\end{enumerate}
\end{lemma}

\begin{proof}
(\ref{polynomiallemmai})
Starting with $c_1 + 2 = T_1(c_1)$ and $1 = S_1$, the assertion follows from Lemma \ref{lelemma} by induction.

(\ref{polynomiallemmaiv})
Follows from (\ref{polynomiallemmai}) and the Corollary \ref{tcontfrac}.

(\ref{polynomiallemmaii})
Follows from  (\ref{polynomiallemmaiv}) and the corresponding symmetry of continuants. Alternatively,
one can conclude form Lemma \ref{evenoddlemma} (\ref{evenoddlemmaii}) that the sets $I_m$ are invariant when
reversing the ordering of the set $\{1, \dots, m\}$.

(\ref{polynomiallemmaiii})
Follows from  (\ref{polynomiallemmaiv}) and the corresponding property of continuants.
\end{proof}

By Lemma \ref{polynomiallemma}, the standard identities of continuants such as Cassini's carry over in a natural
way to T-continuants. The following Lemma adds to this some additional relations between T-continuants and T-semicontinuants.

\begin{lemma}\label{productformulas}
\begin{enumerate}[(i)]
\item\label{productformulasi}
Let $m, n > 0$, $c_1, \dots, c_{m + n} > 0$, then
$$T_{m + n}(c_1, \dots, c_{m + n}) = S_{m + 1}(c_1, \dots, c_m) S_{n + 1}(c_{m + n}, \dots, c_{m + 1})
+ S_m(c_1, \dots, c_{m - 1}) S_n(c_{m + n}, \dots, c_{m + 2}).$$
\item\label{productformulasii}
Let $m > 0$, $c_1, \dots, c_m > 0$, $1 \leq i \leq m$, and $c_i + d > 0$, then:
$$ T_m(c_1, \dots, c_{i - i}, c_i + d, c_{i + 1}, \dots, c_m) = T_m(c_1, \dots, c_m) +
d S_i(c_1, \dots, c_{i - 1}) S_{m - i + 1}(c_m, \dots, c_{i + 1}), $$
in particular:
\begin{align*}
T_m(c_1, \dots, c_{m - 1}, c_m + d) & = T_m(c_1, \dots, c_m) + d S_m(c_1, \dots, c_{m - 1})\\
T_m(c_1 + d, c_2, \dots, c_m) & = T_m(c_1, \dots, c_m) + d S_m(c_m, \dots, c_2).
\end{align*}
\item\label{productformulasiii}
\begin{align*}
T_m(c_1, \dots, c_{i - 1}, c_i + 2, c_{i + 1}, \dots, c_m) = & \ S_{m - i + 1}(c_m, \dots, c_{i + 1} )T_i(c_1, \dots, c_i)\\
& + S_i(c_1, \dots, c_{i - 1}) T_{m - i}(c_{i + 1}, \dots, c_m).
\end{align*}
\item\label{productformulasiv}
Let $m \geq  i > 0$, $c_1, \dots, c_{m - 1} > 0$. Then
$$
S_m(c_1, \dots, c_{m - 1}) + (-1)^i S_{m - i + 1}(c_1, \dots, c_{m - i}) =
\sum_{j = 1}^{i - 1} (-1)^j T_{m - j}(c_1, \dots, c_{m - j})
$$
and in particular, for $i = m$:
$$
S_m(c_1, \dots, c_{m - 1}) = \sum_{j = 1}^{m - 1} (-1)^j T_{m - j}(c_1, \dots, c_{m - j}) - (-1)^m.
$$
\item\label{productformulasv}
$$
T(c_1, \dots, c_m) - T(c_1, \dots, c_i) = \sum_{j = i + 1}^m c_j S_j(c_1, \dots, c_{j - 1}).
$$
\end{enumerate}
\end{lemma}

\begin{proof}
(\ref{productformulasi})
Follows by induction with the formulas $S_{i} = S_{i - 2} + c_{i - 1} S_{i - 1}$ and
$T_i = S_i + S_{i + 1}$ for any $i$.

(\ref{productformulasii})
It follows immediately from the definition that
$S_{m + 1}(c_1, \dots, c_{m - 1}, c_m + d) = S_{m + 1}(c_1, \dots, c_m) + d S_m(c_1, \dots, c_{m - 1})$.
With this the assertion follows by direct computation from (\ref{productformulasi}).

(\ref{productformulasiii})
Direct computation using (\ref{productformulasii}) and \ref{polynomiallemma}  (\ref{polynomiallemmaii}).

(\ref{productformulasiv})
Induction using $S_m = T_{m - 1} - S_{m -1}$.
\end{proof}

\section{Local Hausdorff dimension of Cantor sets}\label{cantorappendix}

In this appendix we collect some material on Hausdorff dimension of Cantor sets that will be needed
in Section \ref{cantor} (for general reference on Hausdorff dimension, we refer to \cite{Falconer97},
but any source on the subject will suffice). In particular, we will introduce the {\em local} Hausdorff
dimension, where we follow some ideas of Baek (see \cite{Baek01}, \cite{Baek04}). For our purposes, we
will present a modified version of arguments in \cite{Baek04}, in order to allow for more convenient boundedness
conditions.

Recall that for any $s \geq 0$ the {\em Hausdorff measure} of a set $\mathfrak{S} \subseteq \R^n$ is defined as
$$
\mathcal{H}^s(\mathfrak{S}) :=
\underset{\delta \rightarrow 0}{\lim} \left(\inf \{\sum_{i = 1}^\infty |B_i|^s \mid
\{B_i\}_{i = 0}^\infty \text{ is a $\delta$-cover}\} \right),
$$
where a $\delta$-cover is a countable collection of sets $B_i$ with diameters $0 < |B_i| \leq \delta$
that cover $\mathfrak{S}$. The {\em Hausdorff dimension} of $\mathfrak{S}$ is defined as
$$
\dim_H \mathfrak{S} :=  \underset{s \geq 0}{\inf} \{\mathcal{H}^s(\mathfrak{S}) = 0\}.
$$

Now let $\mathfrak{C} \subset \R$ be any Cantor set, where our only condition is that $\mathfrak{C}$
is totally disconnected, i.e. $\mathfrak{C}$ does not contain a proper interval of $\R$.

As in Section \ref{cantor}, we denote $\mathfrak{P}_f$, $\mathfrak{P}$ the sets of finite and infinite
paths, respectively in a binary tree, i.e. finite and infinite words built from the letters $L$ and $R$.
Then $\mathfrak{C}$ is constructed iteratively by starting with an initial closed interval $I_\emptyset$,
where we denote $\emptyset$ the empty word in $\mathfrak{P}_f$. Then, inductively, for any word $W \in \mathfrak{P}_f$,
a middle open subinterval of $I_W$ is removed, leaving two new intervals $I_{WL}$ and $I_{WR}$.
Then there is a one-to-one correspondence between elements $p \in \mathfrak{P}$ and points
$x_p = \bigcap_{i \geq 0} I_{p |_i}$, where $p |_i \in \mathfrak{P}_f$ denotes the prefix of
length $i$ of $p$.

Now, following \cite{Baek01}, \cite{Baek04}, for any $W \in \mathfrak{P}_f$ we set
$d_{WL} = |I_{WL}| / |I_W|$ and $d_{WR} = |I_{WR}| / |I_W|$. With this, we define the {\em local Hausdorff measure}
$$
h^s(x_p) = \liminf_{k \rightarrow \infty} \prod_{i = 0}^k (d_{p|_i L}^s + d_{p|_i R}^s)
$$
and the {\em local Hausdorff dimension}
$$
h(x_p) = \inf_{s \geq 0}\{h^s(x_p) = 0\}
$$
for any $p \in \mathfrak{P}$.

\begin{lemma}\label{localdimlemma}
Let $\mu$ be a finite Borel measure with support on $\mathfrak{C}$ and $x_p = \bigcap_{i \geq 0} I_{p|_i} \in \mathfrak{C}$.
Assume that there exists a subsequence $\{p|_{i_k} \mid k \geq 0\}$ such that
$\lim_{k \rightarrow \infty} \log \mu(I_{p|_{i_k}}) / \log |I_{p|_{i_k}}| =$
$\liminf_{i \rightarrow \infty}$ $\log \mu(I_{p|_i}) / \log |I_{p|_i}|$ and
$2|I_{p|_{i_k}}| + |I_{p|_{{i_k} - 1}X_{i_k}}| \leq |I_{p|_{{i_k} - 1}}|$ for all $k$ (where $X_{i_k} = L$
if the last letter in $p|_i$is $R$, and vice versa). Then
$$
\liminf_{\epsilon \rightarrow 0} \frac{\log \mu(B_\epsilon(x_p))}{\log \epsilon} \leq
\liminf_{i \rightarrow \infty} \frac{\log \mu(I_{p|_i})}{\log |I_{p|_i}|},
$$
where $B_\epsilon(x)$ denotes the open $n$-ball of radius $\epsilon$ with center $x$.
\end{lemma}

\begin{proof}
The conditions imply that for every $k \geq 0$, we can find a ball $B_{\epsilon_k}(x_p)$ such that
$B_{\epsilon_k}(x_p) \cap \mathfrak{C} = I_{p_k} \cap \mathfrak{C}$. Hence, for any measure $\mu$ that
is supported on $\mathfrak{C}$ we have $\mu(B_{\epsilon_k}) = \mu(I_{p|_{i_k}})$. Moreover,
$\epsilon_k \geq |I_{p|_{i_k}}|$ for all $k$, hence $\log \mu(B_{\epsilon_k}(x_p)) / \log \epsilon_k \leq
\log \mu(I_{p|_{i_k}}) / \log |I_{p|_{i_k}}|$ for $k \gg 0$, and the statement follows.
\end{proof}

\begin{theorem}\label{hausdorfflocalglobal}
Let $\mathfrak{C}$ be a Cantor set that satisfies the conditions of Lemma \ref{localdimlemma} for every $x \in \mathfrak{C}$.
Then $\dim_H \mathfrak{C} \leq \sup_{x \in \mathfrak{C}} h(x)$.
\end{theorem}

\begin{proof}
For any $s \geq 0$ and any $W \in \mathfrak{P}_f$ of length $n$, we define
$$
\mu_s(I_W \cap \mathfrak{C}) = \frac{|I_W|^s}{\prod_{k = 0}^n (d^s_{W|_k L} + d^s_{W|_k R})},
$$
which in a natural way extends to a finite Borel measure on $\mathfrak{C}$ (see \cite{Baek01}, \cite{Baek04}).
If $h^s(x_p) = 0$ for some $x_p \in \mathfrak{C}$, we get by Lemma \ref{localdimlemma}:
$$
\liminf_{\epsilon \rightarrow 0} \frac{\log \mu(B_\epsilon(x_p))}{\log \epsilon} \leq
s - \limsup_{i \rightarrow \infty} \frac{\log \prod_{k = 0}^{i - 1}(d^s_{p|_k L} + d^s_{p|_k R})}{\log |I_{p|_i}|}.
$$
Now, $\log|I_{p|_i}| < 0$ for $i \gg 0$, and from $h^s(x_p) = 0$ it follows that
$\log \prod_{k = 0}^{i - 1}(d^s_{p|_k L} + d^s_{p|_k R}) < 0$ for infinitely many $i$.
Hence $\liminf_{\epsilon \rightarrow 0} \log \mu(B_\epsilon(x_p)) / \log \epsilon \leq s$.

Moreover, it follows that 
$\liminf_{\epsilon \rightarrow 0} \log \mu(B_\epsilon(x_p)) / \log \epsilon \leq h(x)$.
Using \cite[Proposition 10.1]{Falconer97}, we conclude the proof.
\end{proof}

The following statement is needed in Section \ref{cantor}:

\begin{corollary}\label{hausdorffzerodim}
Let $\mathfrak{C}$ be a Cantor set that satisfies the conditions of Lemma \ref{localdimlemma} for every $x \in \mathfrak{C}$.
If $h(x) = 0$ for every $x \in \mathfrak{C}$, then $\dim_H \mathfrak{C} = 0$.
\end{corollary}

In \cite{juergensenstaiger95}, the local Hausdorff dimension at a point $x \in \mathfrak{C}$ was introduced as
$$
l^{\dim_H}_\mathfrak{C}(x) = \underset{\varepsilon \rightarrow 0}{\lim} \dim_H(B_\varepsilon(x) \cap \mathfrak{C}),
$$
and it is a natural question how $h(x)$ relates to $l^{\dim_H}_\mathfrak{C}$. The following corollary provides a partial answer.

\begin{corollary}
Let $x \in \mathfrak{C}$. If $h(x) = \lim_{\epsilon \rightarrow 0 }\sup_{y \in B_\epsilon(x) \cap \mathfrak{C}} h(y)$, then
$$
l^{\dim_H}_\mathfrak{C}(x) \leq h(x).
$$
\end{corollary}

\newpage

\section{The first few levels of the Markoff Tree}\label{firstfewlevels}

\begin{figure}[ht]
\tikzstyle{level 1}=[level distance=1cm, sibling distance=8cm]
\tikzstyle{level 2}=[level distance=1cm, sibling distance=4cm]
\tikzstyle{level 3}=[level distance=1cm, sibling distance=2cm]
\tikzstyle{level 4}=[level distance=3cm, sibling distance=1cm]
\tikzstyle{level 5}=[level distance=5cm, sibling distance=0.5cm]
\begin{tikzpicture}[grow=right, sloped]
  \node[align=left] {$1, 5, 2$}
    child {node[align=left]{$1, 13, 5$}
      child {node[align=left] {$1, 34, 13$}
        child {node[align=left] {$1, 89, 34$}
           child {node[align=left] {$1, 233, 89$}
               child {node[align=left] {$1, 610, 233$}}
               child {node[align=left] {$233, 62210, 89$}}
           }
           child {node[align=left] {$89, 9077, 34$}
               child {node[align=left] {$89, 2423525, 9077$}}
               child {node[align=left] {$9077, 925765, 34$}}
           }
        }
        child {node[align=left] {$34, 1325, 13$}
           child {node[align=left] {$34, 135137, 1325$}
               child {node[align=left] {$34, 13782649, 135137$}}
               child {node[align=left] {$135137, 537169541, 1325$}}
           }
           child {node[align=left] {$1325, 51641, 13$}
               child {node[align=left] {$1325, 205272962, 51641$}}
               child {node[align=left] {$51641, 2012674, 13$}}
           }
        }
      }
      child {node[align=left] {$13, 194, 5$}
        child {node[align=left] {$13, 7561, 194$}
          child {node[align=left] {$13, 294685, 7561$}
            child {node[align=left] {$13, 11485154, 294685$}}
            child {node[align=left] {$294685, 6684339842, 7561$}}
          }
          child {node[align=left] {$7561, 4400489, 194$}
            child {node[align=left] {$7561, 99816291793, 4400489$}}
            child {node[align=left] {$4400489, 2561077037, 194$}}
          }
        }
        child {node[align=left] {$194, 2897, 5$}
          child {node[align=left] {$194, 1686049, 2897$}
            child {node[align=left] {$194, 981277621, 1686049$}}
            child {node[align=left] {$1686049, 14653451665, 2897$}}
          }
          child {node[align=left] {$2897, 43261, 5$}
            child {node[align=left] {$2897, 375981346, 43261$}}
            child {node[align=left] {$43261, 646018, 5$}}
          }
        }
      }
    }
    child {node[align=left] {$5, 29, 2$}
      child {node[align=left] {$5, 433, 29$}
        child {node[align=left] {$5, 6466, 433$}
         child {node[align=left] {$5, 96557, 6466$}
           child {node[align=left] {$5, 1441889, 96557$}}
           child {node[align=left] {$96557, 1873012681, 6466$}}
         }
         child {node[align=left] {$6466, 8399329, 433$}
           child {node[align=left] {$6466, 162930183509, 8399329$}}
           child {node[align=left] {$8399329, 10910721905, 433$}}
         }
        }
        child {node[align=left] {$433, 37666, 29$}
         child {node[align=left] {$433, 48928105, 37666$}
           child {node[align=left] {$433, 63557570729, 48928105$}}
           child {node[align=left] {$48928105, 5528778008357, 37666$}}
         }
         child {node[align=left] {$37666, 3276509, 29$}
           child {node[align=left] {$37666, 370238963953, 3276509$}}
           child {node[align=left] {$3276509, 285018617, 29$}}
         }
        }
      }
      child {node[align=left] {$29, 169, 2$}
        child {node[align=left] {$29, 14701, 169$}
          child {node[align=left] {$29, 1278818, 14701$}
           child {node[align=left] {$29, 111242465, 1278818$}}
           child {node[align=left] {$1278818, 56399710225, 14701$}}
          }
          child {node[align=left] {$14701, 7453378, 169$}
           child {node[align=left] {$14701, 328716329765, 7453378$}}
           child {node[align=left] {$7453378, 3778847945, 169$}}
          }
        }
        child {node[align=left] {$169, 985, 2$}
          child {node[align=left] {$169, 499393, 985$}
           child {node[align=left] {$169, 253191266, 499393$}}
           child {node[align=left] {$499393, 1475706146, 985$}}
          }
          child {node[align=left] {$985, 5741, 2$}
           child {node[align=left] {$985, 16964653, 5741$}}
           child {node[align=left] {$5741, 33461, 2$}}
          }
        }
      }
    };
\end{tikzpicture}
\end{figure}

\begin{figure}[ht]
\small
\tikzstyle{level 1}=[level distance=1cm, sibling distance=11cm]
\tikzstyle{level 2}=[level distance=1cm, sibling distance=5.5cm]
\tikzstyle{level 3}=[level distance=1cm, sibling distance=2.75cm]
\tikzstyle{level 4}=[level distance=1.5cm, sibling distance=1.35cm]
\tikzstyle{level 5}=[level distance=3cm, sibling distance=0.7cm]
\tikzstyle{level 6}=[level distance=4cm, sibling distance=0.35cm]
\begin{tikzpicture}[grow=right, sloped]
  \node[align=left] {$5$}
    child {node[align=left]{$13$}
      child {node[align=left] {$34$}
        child {node[align=left] {$89$}
           child {node[align=left] {$233$}
             child {node[align=left] {$610$}
               child {node[align=left] {$1597$}}
               child {node[align=left] {$426389$}}
             }
             child {node[align=left] {$62210$}
               child {node[align=left] {$43484701$}}
               child {node[align=left] {$16609837$}}
             }
           }
           child {node[align=left] {$9077$}
             child {node[align=left] {$2423525$}
               child {node[align=left] {$647072098$}}
               child {node[align=left] {$65995009186$}}
             }
             child {node[align=left] {$925765$}
               child {node[align=left] {$25209506681$}}
               child {node[align=left] {$94418953$}}
             }
           }
        }
        child {node[align=left] {$1325$}
           child {node[align=left] {$135137$}
             child {node[align=left] {$13782649$}
               child {node[align=left] {$1405695061$}}
               child {node[align=left] {$5587637513705$}}
             }
             child {node[align=left] {$537169541$}
               child {node[align=left] {$217774440785026$}}
               child {node[align=left] {$2135248790338$}}
             }
           }
           child {node[align=left] {$51641$}
             child {node[align=left] {$205272962$}
               child {node[align=left] {$815959972309$}}
               child {node[align=left] {$31801503090601$}}
             }
             child {node[align=left] {$2012674$}
               child {node[align=left] {$311809494089$}}
               child {node[align=left] {$78442645$}}
             }
           }
        }
      }
      child {node[align=left] {$194$}
        child {node[align=left] {$7561$}
          child {node[align=left] {$294685$}
            child {node[align=left] {$11485154$}
              child {node[align=left] {$447626321$}}
              child {node[align=left] {$10153507819457$}}
           }
           child {node[align=left] {$6684339842$}
             child {node[align=left] {$5909324059011749$}}
             child {node[align=left] {$151620880341401$}}
           }
         }
         child {node[align=left] {$4400489$}
           child {node[align=left] {$99816291793$}
             child {node[align=left] {$2264132942340130$}}
             child {node[align=left] {$1317721482167652770$}}
           }
           child {node[align=left] {$2561077037$}
             child {node[align=left] {$33809973988413085$}}
             child {node[align=left] {$1490542435045$}}
           }
         }
        }
        child {node[align=left] {$2897$}
          child {node[align=left] {$1686049$}
            child {node[align=left] {$981277621$}
              child {node[align=left] {$571101889373$}}
              child {node[align=left] {$4963446454828093$}}
           }
           child {node[align=left] {$14653451665$}
             child {node[align=left] {$74119312578961858$}}
             child {node[align=left] {$127353146734466$}}
           }
          }
          child {node[align=left] {$43261$}
            child {node[align=left] {$375981346$}
              child {node[align=left] {$3267653834825$}}
              child {node[align=left] {$48795987025021$}}
            }
            child {node[align=left] {$646018$}
              child {node[align=left] {$83842154089$}}
              child {node[align=left] {$9647009$}}
            }
          }
        }
      }
    }
    child {node[align=left] {$29$}
      child {node[align=left] {$433$}
        child {node[align=left] {$6466$}
          child {node[align=left] {$96557$}
            child {node[align=left] {$1441889$}
              child {node[align=left] {$21531778$}}
              child {node[align=left] {$417673428514$}}
            }
            child {node[align=left] {$1873012681$}
              child {node[align=left] {$542557456311485$}}
              child {node[align=left] {$36332699889481$}}
            }
          }
          child {node[align=left] {$8399329$}
            child {node[align=left] {$162930183509$}
              child {node[align=left] {$3160519691308253$}}
              child {node[align=left] {$4105512645967389917$}}
            }
            child {node[align=left] {$10910721905$}
              child {node[align=left] {$274928228722804802$}}
              child {node[align=left] {$14173019355266$}}
            }
          }
        }
        child {node[align=left] {$37666$}
          child {node[align=left] {$48928105$}
            child {node[align=left] {$63557570729$}
              child {node[align=left] {$82561235448866$}}
              child {node[align=left] {$9329254482520315202$}}
            }
            child {node[align=left] {$5528778008357$}
              child {node[align=left] {$811537892743746482789$}}
              child {node[align=left] {$624740857339396181$}}
            }
          }
          child {node[align=left] {$3276509$}
            child {node[align=left] {$370238963953$}
              child {node[align=left] {$41836262445484585$}}
              child {node[align=left] {$3639273892628002565$}}
            }
            child {node[align=left] {$285018617$}
              child {node[align=left] {$2801598191304130$}}
              child {node[align=left] {$24793343170$}}
            }
          }
        }
      }
      child {node[align=left] {$169$}
        child {node[align=left] {$14701$}
          child {node[align=left] {$1278818$}
            child {node[align=left] {$111242465$}
              child {node[align=left] {$9676815637$}}
              child {node[align=left] {$426776599819081$}}
            }
            child {node[align=left] {$56399710225$}
              child {node[align=left] {$216374893891527449$}}
              child {node[align=left] {$2487396418774357$}}
            }
          }
          child {node[align=left] {$7453378$}
            child {node[align=left] {$328716329765$}
              child {node[align=left] {$14497376284172417$}}
              child {node[align=left] {$7350141181533573809$}}
            }
            child {node[align=left] {$3778847945$}
              child {node[align=left] {$84495546415824461$}}
              child {node[align=left] {$1915868454737$}}
            }
          }
        }
        child {node[align=left] {$985$}
          child {node[align=left] {$499393$}
            child {node[align=left] {$253191266$}
              child {node[align=left] {$128367472469$}}
              child {node[align=left] {$379325837704445$}}
            }
            child {node[align=left] {$1475706146$}
              child {node[align=left] {$2210871958107149$}}
              child {node[align=left] {$4360711162037$}}
            }
          }
          child {node[align=left] {$5741$}
            child {node[align=left] {$16964653$}
              child {node[align=left] {$50130543874$}}
              child {node[align=left] {$292182217634$}}
            }
            child {node[align=left] {$33461$}
              child {node[align=left] {$576298801$}}
              child {node[align=left] {$195025$}}
            }
          }
        }
      }
    };
\end{tikzpicture}
\end{figure}

\newpage

\ 

\newpage

\section{The 300 Smallest Markoff Numbers}\label{threehundred}

We present a computer generated list of the smallest 300 Markoff numbers together with their weights and coweights.

\setlength{\tabcolsep}{1mm}

{\footnotesize

\begin{longtable}{l | p{.17\textwidth} | p{.17\textwidth} | p{.17\textwidth} | p{.17\textwidth} | p{.17\textwidth}}
\caption{The 300 smallest Markoff numbers and their weights}\\
$n$ & $m_n$ & $r_{m_n}$ & $s_{m_n}$ & $w_{m_n}$ & $v_{m_n}$ \\ \hline \hline \endhead
1 & 1 & 0 & 1 & -1 & 10 \\
 & 2 & 1 & 1 & 1 & 5 \\
 & 5 & 2 & 1 & 1 & 2 \\
 & 13 & 5 & 2 & 2 & 1 \\
 & 29 & 12 & 5 & 7 & 2 \\
 & 34 & 13 & 5 & 5 & 1 \\
 & 89 & 34 & 13 & 13 & 2 \\
 & 169 & 70 & 29 & 41 & 10 \\
 & 194 & 75 & 29 & 31 & 5 \\
10 & 233 & 89 & 34 & 34 & 5 \\
 & 433 & 179 & 74 & 104 & 25 \\
 & 610 & 233 & 89 & 89 & 13 \\
 & 985 & 408 & 169 & 239 & 58 \\
 & 1325 & 507 & 194 & 196 & 29 \\
 & 1597 & 610 & 233 & 233 & 34 \\
 & 2897 & 1120 & 433 & 463 & 74 \\
 & 4181 & 1597 & 610 & 610 & 89 \\
 & 5741 & 2378 & 985 & 1393 & 338 \\
 & 6466 & 2673 & 1105 & 1553 & 373 \\
20 & 7561 & 2923 & 1130 & 1208 & 193 \\
 & 9077 & 3468 & 1325 & 1327 & 194 \\
 & 10946 & 4181 & 1597 & 1597 & 233 \\
 & 14701 & 6089 & 2522 & 3566 & 865 \\
 & 28657 & 10946 & 4181 & 4181 & 610 \\
 & 33461 & 13860 & 5741 & 8119 & 1970 \\
 & 37666 & 15571 & 6437 & 9047 & 2173 \\
 & 43261 & 16725 & 6466 & 6914 & 1105 \\
 & 51641 & 19760 & 7561 & 7639 & 1130 \\
 & 62210 & 23763 & 9077 & 9079 & 1325 \\
30 & 75025 & 28657 & 10946 & 10946 & 1597 \\
 & 96557 & 39916 & 16501 & 23191 & 5570 \\
 & 135137 & 51709 & 19786 & 19990 & 2957 \\
 & 195025 & 80782 & 33461 & 47321 & 11482 \\
 & 196418 & 75025 & 28657 & 28657 & 4181 \\
 & 294685 & 113922 & 44041 & 47081 & 7522 \\
 & 426389 & 162867 & 62210 & 62212 & 9077 \\
 & 499393 & 206855 & 85682 & 121172 & 29401 \\
 & 514229 & 196418 & 75025 & 75025 & 10946 \\
 & 646018 & 249755 & 96557 & 103247 & 16501 \\
40 & 925765 & 353702 & 135137 & 135341 & 19786 \\
 & 1136689 & 470832 & 195025 & 275807 & 66922 \\
 & 1278818 & 529673 & 219385 & 310201 & 75245 \\
 & 1346269 & 514229 & 196418 & 196418 & 28657 \\
 & 1441889 & 596067 & 246410 & 346312 & 83177 \\
 & 1686049 & 651838 & 252005 & 269465 & 43066 \\
 & 2012674 & 770133 & 294685 & 297725 & 44041 \\
 & 2423525 & 925943 & 353770 & 354304 & 51797 \\
 & 2922509 & 1116300 & 426389 & 426391 & 62210 \\
 & 3276509 & 1354498 & 559945 & 786985 & 189026 \\
50 & 3524578 & 1346269 & 514229 & 514229 & 75025 \\
 & 4400489 & 1701181 & 657658 & 703054 & 112325 \\
 & 6625109 & 2744210 & 1136689 & 1607521 & 390050 \\
 & 7453378 & 3087111 & 1278649 & 1807955 & 438553 \\
 & 8399329 & 3472225 & 1435394 & 2017346 & 484525 \\
 & 9227465 & 3524578 & 1346269 & 1346269 & 196418 \\
 & 9647009 & 3729600 & 1441889 & 1541791 & 246410 \\
 & 11485154 & 4440035 & 1716469 & 1834951 & 293165 \\
 & 13782649 & 5273811 & 2017978 & 2038784 & 301585 \\
 & 16609837 & 6344632 & 2423525 & 2424059 & 353770 \\
60 & 16964653 & 7026989 & 2910674 & 4116314 & 998785 \\
 & 20031170 & 7651227 & 2922509 & 2922511 & 426389 \\
 & 21531778 & 8901089 & 3679649 & 5171489 & 1242085 \\
 & 24157817 & 9227465 & 3524578 & 3524578 & 514229 \\
 & 38613965 & 15994428 & 6625109 & 9369319 & 2273378 \\
 & 43484701 & 16610303 & 6344810 & 6346208 & 926173 \\
 & 48928105 & 20226717 & 8361658 & 11752046 & 2822725 \\
 & 63245986 & 24157817 & 9227465 & 9227465 & 1346269 \\
 & 78442645 & 30015427 & 11485154 & 11603636 & 1716469 \\
 & 94418953 & 36074136 & 13782649 & 13803455 & 2017978 \\
70 & 111242465 & 46075462 & 19083973 & 26983921 & 6545450 \\
 & 137295677 & 52442283 & 20031170 & 20031172 & 2922509 \\
 & 144059117 & 55694245 & 21531778 & 23023618 & 3679649 \\
 & 165580141 & 63245986 & 24157817 & 24157817 & 3524578 \\
 & 205272962 & 78545995 & 30054973 & 30365023 & 4491749 \\
 & 225058681 & 93222358 & 38613965 & 54608393 & 13250218 \\
 & 253191266 & 104875077 & 43440605 & 61433965 & 14906249 \\
 & 285018617 & 117825755 & 48708778 & 68458648 & 16443089 \\
 & 298045301 & 113843800 & 43484701 & 43486099 & 6344810 \\
 & 321534781 & 132920268 & 54948325 & 77226023 & 18548098 \\
80 & 375981346 & 145356973 & 56196005 & 60089573 & 9603553 \\
 & 433494437 & 165580141 & 63245986 & 63245986 & 9227465 \\
 & 447626321 & 173047443 & 66898250 & 71516008 & 11425913 \\
 & 537169541 & 205543262 & 78649345 & 79460245 & 11754074 \\
 & 576298801 & 238710779 & 98877242 & 139833536 & 33929305 \\
 & 647072098 & 247223313 & 94455265 & 94597841 & 13829605 \\
 & 780291637 & 298046521 & 113844266 & 113847926 & 16610905 \\
 & 941038565 & 359444748 & 137295677 & 137295679 & 20031170 \\
 & 981277621 & 379368596 & 146666477 & 156828167 & 25064338 \\
 & 1134903170 & 433494437 & 165580141 & 165580141 & 24157817 \\
90 & 1311738121 & 543339720 & 225058681 & 318281039 & 77227930 \\
 & 1405695061 & 537877013 & 205813970 & 207935978 & 30758713 \\
 & 1475706146 & 611256455 & 253190281 & 358063219 & 86879945 \\
 & 1873012681 & 774290566 & 320086397 & 449859017 & 108046858 \\
 & 2151239746 & 831684075 & 321534781 & 343812479 & 54948325 \\
 & 2561077037 & 990084419 & 382755826 & 409176220 & 65372957 \\
 & 2971215073 & 1134903170 & 433494437 & 433494437 & 63245986 \\
 & 3057250481 & 1169831520 & 447626321 & 452244079 & 66898250 \\
 & 3778847945 & 1565159188 & 648272521 & 916629619 & 222345506 \\
 & 4434764269 & 1693992981 & 647072098 & 647214674 & 94455265 \\
100 & 4801489937 & 1984902931 & 820545226 & 1153218856 & 276979385 \\
 & 5348189873 & 2042828390 & 780291637 & 780295297 & 113844266 \\
 & 6449974274 & 2463670947 & 941038565 & 941038567 & 137295677 \\
 & 6684339842 & 2584092721 & 998982001 & 1067938321 & 170621525 \\
 & 7645370045 & 3166815962 & 1311738121 & 1855077841 & 450117362 \\
 & 7778742049 & 2971215073 & 1134903170 & 1134903170 & 165580141 \\
 & 9629807441 & 3679208170 & 1405695061 & 1407817069 & 205813970 \\
 & 9676815637 & 4008035521 & 1660086266 & 2347290926 & 569378905 \\
 & 10910721905 & 4510417602 & 1864575701 & 2620530901 & 629397602 \\
 & 14001740009 & 5348193067 & 2042829610 & 2042839192 & 298048097 \\
110 & 14653451665 & 5665123983 & 2190175426 & 2341920284 & 374286601 \\
 & 17445941365 & 6744410242 & 2607315281 & 2787289361 & 445317442 \\
 & 19577194573 & 8109139505 & 3358915562 & 4750223942 & 1152597601 \\
 & 20365011074 & 7778742049 & 2971215073 & 2971215073 & 433494437 \\
 & 24793343170 & 10249486187 & 4237103741 & 5955115391 & 1430359717 \\
 & 25209506681 & 9631659149 & 3679915642 & 3685470766 & 538792565 \\
 & 30395743789 & 11610578034 & 4435013113 & 4435990313 & 647393602 \\
 & 32124537073 & 12419566880 & 4801489937 & 5134163567 & 820545226 \\
 & 44208781349 & 16886251875 & 6449974274 & 6449974276 & 941038565 \\
 & 44560482149 & 18457556052 & 7645370045 & 10812186007 & 2623476242 \\
120 & 50130543874 & 20764750117 & 8601040685 & 12163706477 & 2951409337 \\
 & 53316291173 & 20365011074 & 7778742049 & 7778742049 & 1134903170 \\
 & 56399710225 & 23360168307 & 9675536650 & 13680794696 & 3318530233 \\
 & 63557570729 & 26274489812 & 10861787305 & 15265898707 & 3666717602 \\
 & 65995009186 & 25214353799 & 9633510857 & 9648052211 & 1410484105 \\
 & 71700814274 & 29640623697 & 12253230065 & 17221056817 & 4136142677 \\
 & 83842154089 & 32413953163 & 12531457130 & 13399705400 & 2141549281 \\
 & 95969331289 & 36657026960 & 14001740009 & 14001749591 & 2042829610 \\
 & 99816291793 & 38587888548 & 14917656385 & 15947373851 & 2547867970 \\
 & 119154326114 & 45593413853 & 17445941365 & 17625915445 & 2607315281 \\
130 & 128367472469 & 53171457184 & 22024301053 & 31146899083 & 7557438842 \\
 & 139583862445 & 53316291173 & 20365011074 & 20365011074 & 2971215073 \\
 & 143367113573 & 54858181515 & 20991006962 & 21207430972 & 3137087141 \\
 & 162930183509 & 67354220371 & 27843772738 & 39132477604 & 9398815925 \\
 & 172765826641 & 66007698628 & 25219201985 & 25257269243 & 3692452738 \\
 & 208333239010 & 79576653333 & 30395743789 & 30396720989 & 4435013113 \\
 & 251250963713 & 95969339651 & 36657030154 & 36657055240 & 5348197193 \\
 & 259717522849 & 107578520350 & 44560482149 & 63018038201 & 15290740090 \\
 & 292182217634 & 121025831139 & 50130538133 & 70895275783 & 17202073997 \\
 & 303011495165 & 115740092172 & 44208781349 & 44208781351 & 6449974274 \\
140 & 311809494089 & 119311314754 & 45653484253 & 46124450173 & 6822963842 \\
 & 328716329765 & 136150856363 & 56392256818 & 79736239324 & 19341502949 \\
 & 365435296162 & 139583862445 & 53316291173 & 53316291173 & 7778742049 \\
 & 370238963953 & 153055564992 & 63272665105 & 88927731023 & 21359559946 \\
 & 417673428514 & 172663323955 & 71377831109 & 100316543351 & 24093964765 \\
 & 479716816349 & 185461819125 & 71700814274 & 76668641026 & 12253230065 \\
 & 571101889373 & 220791871034 & 85359637609 & 91273723729 & 14587401650 \\
 & 665048316673 & 275472032399 & 114104251874 & 161367780524 & 39154389145 \\
 & 679944086914 & 262858951995 & 101618397709 & 108632769071 & 17355954325 \\
 & 815959972309 & 312220310365 & 119468510114 & 120700958786 & 17854701145 \\
150 & 841771717954 & 348653014865 & 144408421169 & 204187326641 & 49529419285 \\
 & 956722026041 & 365435296162 & 139583862445 & 139583862445 & 20365011074 \\
 & 982145940029 & 375243159204 & 143367113573 & 143583537583 & 20991006962 \\
 & 1070710724173 & 442624452524 & 182977905749 & 257162633399 & 61765160770 \\
 & 1184065449986 & 452289781295 & 172765826641 & 172803893899 & 25219201985 \\
 & 1427933269321 & 545424970563 & 208334944570 & 208341642368 & 30397947073 \\
 & 1490542435045 & 576227430677 & 222763233074 & 238139856986 & 38046948649 \\
 & 1513744654945 & 627013566048 & 259717522849 & 367296043199 & 89120964298 \\
 & 1722099665665 & 657783551488 & 251250963713 & 251250988799 & 36657030154 \\
 & 1915868454737 & 793532621205 & 328672889498 & 464729408878 & 112728732989 \\
160 & 2076871684802 & 793294393323 & 303011495165 & 303011495167 & 44208781349 \\
 & 2135248790338 & 817034414741 & 312631126589 & 315854453885 & 46722441193 \\
 & 2156735837173 & 891587472514 & 368579316689 & 518026580369 & 124424852290 \\
 & 2504730781961 & 956722026041 & 365435296162 & 365435296162 & 53316291173 \\
 & 3099893879221 & 1184098670126 & 452302470737 & 452402131157 & 66024095098 \\
 & 3267653834825 & 1263297435618 & 488399472989 & 522238472029 & 83464478018 \\
 & 4360711162037 & 1806262617670 & 748177194673 & 1058076690973 & 256730208074 \\
 & 4508515437145 & 1722099687557 & 657783559850 & 657783625526 & 95969350453 \\
 & 4643961467965 & 1776973308747 & 679944086914 & 686958458276 & 101618397709 \\
 & 5528778008357 & 2285578567387 & 944850630610 & 1327957693804 & 318962279525 \\
170 & 5587637513705 & 2138060991308 & 818110478953 & 826545460219 & 122265876434 \\
 & 6557470319842 & 2504730781961 & 956722026041 & 956722026041 & 139583862445 \\
 & 7163627708162 & 2769507719995 & 1070710724173 & 1144895451823 & 182977905749 \\
 & 8115549747397 & 3099980848801 & 1184131890266 & 1184392799006 & 172851666985 \\
 & 8822750406821 & 3654502875938 & 1513744654945 & 2140758220993 & 519435045698 \\
 & 9787184545081 & 3738374837400 & 1427933269321 & 1427939967119 & 208334944570 \\
 & 9925594216162 & 4111315732857 & 1702962733225 & 2408352982409 & 584364418045 \\
 & 10153507819457 & 3925235141920 & 1517453001793 & 1622197606303 & 259173984074 \\
 & 14173019355266 & 5859028992773 & 2422082400205 & 3404067623053 & 817587000473 \\
 & 14235090298445 & 5437320661083 & 2076871684802 & 2076871684804 & 303011495165 \\
180 & 14622039889385 & 5594996637517 & 2140876896154 & 2162950023166 & 319952129669 \\
 & 15988960048321 & 6609726164163 & 2732415356170 & 3840218444168 & 922341268873 \\
 & 17167680177565 & 6557470319842 & 2504730781961 & 2504730781961 & 365435296162 \\
 & 18696424380481 & 7228166198398 & 2794458744005 & 2988074214713 & 477555886138 \\
 & 21246581423810 & 8115777435463 & 3100067821177 & 3100750882579 & 452527201625 \\
 & 22592065572301 & 9357939962069 & 3876185648162 & 5481754313906 & 1330096633345 \\
 & 26500373448281 & 10244754717563 & 3960510195370 & 4233890704408 & 676436901233 \\
 & 30901824368813 & 11803446623878 & 4508515437145 & 4508515502821 & 657783559850 \\
 & 31801503090601 & 12168581182878 & 4656206581885 & 4704240458033 & 695875230298 \\
 & 36332699889481 & 15019688359352 & 6209035912505 & 8726365188575 & 2095892945914 \\
190 & 44945570212853 & 17167680177565 & 6557470319842 & 6557470319842 & 956722026041 \\
 & 46127828641049 & 17623808310363 & 6733432474730 & 6743596290040 & 985871051441 \\
 & 48795987025021 & 18864864025739 & 7293286116482 & 7798605052196 & 1246377918925 \\
 & 51422757785981 & 21300003689580 & 8822750406821 & 12477253282759 & 3027489309890 \\
 & 57850602535042 & 23962504125979 & 9925594182701 & 14036909842895 & 3405925423477 \\
 & 65082055350517 & 26957823917211 & 11166277193266 & 15791416401116 & 3831606586645 \\
 & 67082333460610 & 25623191867697 & 9787196235121 & 9787242142481 & 1427948370517 \\
 & 73224462646361 & 30328804257734 & 12561872555437 & 17761950126841 & 4308490098890 \\
 & 80902026460669 & 30901824426127 & 11803446645770 & 11803446817712 & 1722099715837 \\
 & 82561235448866 & 34130542039071 & 14109453347537 & 19830390668347 & 4763063342273 \\
200 & 93139301545921 & 38503146951361 & 15916936250882 & 22370139308162 & 5372846095693 \\
 & 97568760404309 & 37267950234252 & 14235090298445 & 14235090298447 & 2076871684802 \\
 & 100169256075517 & 38271123030638 & 14622039889385 & 14644113016397 & 2140876896154 \\
 & 106974698806081 & 41357153980800 & 15988960048321 & 17096763136319 & 2732415356170 \\
 & 117669030460994 & 44945570212853 & 17167680177565 & 17167680177565 & 2504730781961 \\
 & 127353146734466 & 49235591884415 & 19034814375361 & 20353628918779 & 3252924806225 \\
 & 145624636022689 & 55623966835967 & 21246581423810 & 21247264485212 & 3100067821177 \\
 & 148135740183017 & 61359829568746 & 25416072313501 & 35943748523221 & 8721413592050 \\
 & 151620880341401 & 58614975076521 & 22659908684642 & 24224044888162 & 3870208044053 \\
 & 180995342924521 & 69256365627280 & 26500373448281 & 26773753957319 & 3960510195370 \\
210 & 187611224490881 & 77557860622531 & 32062163448202 & 45062357376712 & 10823531789513 \\
 & 217774440785026 & 83329499390175 & 31885309605601 & 32214057385499 & 4765230894385 \\
 & 238763690000642 & 98703268009921 & 40803252436801 & 57346114029121 & 13773353872325 \\
 & 262229286072101 & 100188517542107 & 38278482154450 & 38336266554220 & 5604520209509 \\
 & 299713796309065 & 124145519261542 & 51422757785981 & 72722761475561 & 17645500813642 \\
 & 308061521170129 & 117669030460994 & 44945570212853 & 44945570212853 & 6557470319842 \\
 & 316141040381993 & 120759677612784 & 46127828641049 & 46137992456359 & 6733432474730 \\
 & 332380318337465 & 128500489573192 & 49679162421961 & 53121150382111 & 8489842695962 \\
 & 379325837704445 & 157121637984713 & 65081802158266 & 92039076249694 & 22332229220561 \\
 & 381249713544034 & 145625232114511 & 55624194524033 & 55625982799499 & 8116071573265 \\
220 & 426776599819081 & 176766390491736 & 73214784551737 & 103522571656127 & 25111317834298 \\
 & 459789046174429 & 175623808514133 & 67082333460610 & 67082379367970 & 9787196235121 \\
 & 473638416248177 & 181233808565389 & 69347612525386 & 70063009447990 & 10364077584317 \\
 & 542557456311485 & 224289522541113 & 92719746704282 & 130311111311854 & 31298041403345 \\
 & 554510738235029 & 211804254955880 & 80902026460669 & 80902026632611 & 11803446645770 \\
 & 668746232531714 & 255438330978675 & 97568760404309 & 97568760404311 & 14235090298445 \\
 & 686480075504546 & 262279709932717 & 100207782712165 & 100359054293605 & 14671860317729 \\
 & 767465181141553 & 317894486677955 & 131676207785642 & 186218278892312 & 45184131144601 \\
 & 806515533049393 & 308061521170129 & 117669030460994 & 117669030460994 & 17167680177565 \\
 & 867493136119153 & 335363374569595 & 129647818893242 & 138596987589632 & 22143258740761 \\
230 & 971341527703714 & 402319473791747 & 166636506702965 & 235616893671527 & 57153245279917 \\
 & 998123312425769 & 381251274127774 & 145625828206333 & 145630509957553 & 21248121516122 \\
 & 1032834620396045 & 399282575032962 & 154358279221721 & 165013104702841 & 26363683193762 \\
 & 1451727959691893 & 554510738385079 & 211804255013194 & 211804255463344 & 30901824500165 \\
 & 1491304701603697 & 570634798845219 & 218348452400746 & 220599694931960 & 32631980139097 \\
 & 1597456854383053 & 617587801992005 & 238763690000642 & 255306551592962 & 40803252436801 \\
 & 1746860020068409 & 723573111879672 & 299713796309065 & 423859315570607 & 102845515571962 \\
 & 1797110092720441 & 686612067374626 & 262330133793197 & 262726109403437 & 38408892611458 \\
 & 1965217522626434 & 814019750849633 & 337178020826785 & 476841729922465 & 115701204969701 \\
 & 2111485077978050 & 806515533049393 & 308061521170129 & 308061521170129 & 44945570212853 \\
240 & 2210871958107149 & 915771584495037 & 379324361998130 & 536442795377962 & 130161709120097 \\
 & 2264132942340130 & 875289074233103 & 338377199123297 & 361734280359179 & 57793289051185 \\
 & 2487396418774357 & 1030253502313948 & 426720192655565 & 603364088167487 & 146357138790754 \\
 & 2613117102565793 & 998127398083769 & 381252834718834 & 381265091685514 & 55628226532685 \\
 & 2801598191304130 & 1158171440067873 & 478784248488001 & 672916128899489 & 161627787288901 \\
 & 3151440817820989 & 1203743419591266 & 459789126299113 & 459789440952809 & 67082436965410 \\
 & 3160519691308253 & 1306537163284433 & 540113502136330 & 759091798545046 & 182318230828625 \\
 & 3243440684655313 & 1241075655154880 & 474887297648177 & 479786280809327 & 70972432559626 \\
 & 3565466389961309 & 1473939293984652 & 609316371195845 & 856351491992647 & 205677966816002 \\
 & 4169218794693154 & 1611848648289613 & 623151768456005 & 666327150175685 & 106492821059521 \\
250 & 4583654867317685 & 1750800366616467 & 668746232531714 & 668746232531716 & 97568760404309 \\
 & 4704586218916261 & 1797455638914941 & 686744088327362 & 687780697828562 & 100549180372873 \\
 & 4963446454828093 & 1918902125751537 & 741860600638090 & 793259922426518 & 126779106061681 \\
 & 5032254317345729 & 2084427882680008 & 863398255353985 & 1221029330694295 & 296271319451546 \\
 & 5527939700884757 & 2111485077978050 & 806515533049393 & 806515533049393 & 117669030460994 \\
 & 5909324059011749 & 2284480090460732 & 883155032892925 & 944116212370447 & 150838812283682 \\
 & 6369686478515453 & 2638257317407993 & 1092738503901850 & 1545085473708526 & 374789109184145 \\
 & 7054174412588354 & 2699221286155173 & 1032834620396045 & 1043489445877165 & 154358279221721 \\
 & 8487613404424009 & 3247711593052213 & 1242708649541930 & 1255521374732630 & 185721691988101 \\
 & 9950291461929757 & 3800673140690600 & 1451727959691893 & 1451727960142043 & 211804255013194 \\
260 & 10181446324101389 & 4217293152016490 & 1746860020068409 & 2470433131948081 & 599427592618130 \\
 & 10216281973762705 & 3903279305965872 & 1491304701603697 & 1493555944134911 & 218348452400746 \\
 & 10881286283486669 & 4206780083103874 & 1626370100606233 & 1739053965824953 & 277936690319522 \\
 & 11454127114764514 & 4744454795874015 & 1965217522431409 & 2779237272857531 & 674356041403105 \\
 & 12315957481333442 & 4705490811168293 & 1797801251550925 & 1800514952171437 & 263223878282009 \\
 & 12885900008113189 & 5337505423958395 & 2210863357068434 & 3126616263761996 & 758637677978725 \\
 & 14472334024676221 & 5527939700884757 & 2111485077978050 & 2111485077978050 & 308061521170129 \\
 & 14497376284172417 & 6004661215090278 & 2487067701165605 & 3516607361098417 & 853018304121194 \\
 & 14851330019499401 & 5672916816810603 & 2166942971989610 & 2167420430932408 & 316315866542273 \\
 & 15423171693415289 & 5962429596483661 & 2305010110741498 & 2464117096035694 & 393685111186805 \\
270 & 16320019794869474 & 6746642286687683 & 2789039640676885 & 3919907065193575 & 941522840835341 \\
 & 17910542369308357 & 6841223909613610 & 2613117102565793 & 2613129359532473 & 381252834718834 \\
 & 18410741231768629 & 7610874151194525 & 3146283173290594 & 4421881221814946 & 1062044884216825 \\
 & 20769646571311849 & 8586029106829510 & 3549405406115549 & 4988440749176681 & 1198120585374730 \\
 & 21600295969140418 & 8250579033871701 & 3151440817820989 & 3151441132474685 & 459789126299113 \\
 & 23854878116939714 & 9222459875899275 & 3565466389961309 & 3812501510758111 & 609316371195845 \\
 & 26050201244948621 & 9950291462322593 & 3800673140840650 & 3800673142019158 & 554510738578913 \\
 & 26071224093240493 & 10799054607088409 & 4473114879063674 & 6325939728024734 & 1534930362283105 \\
 & 28399179102482729 & 10979317867599065 & 4244679763551394 & 4538774500314466 & 725389768850885 \\
 & 31416837838692077 & 12000164235336588 & 4583654867317685 & 4583654867317687 & 668746232531714 \\
280 & 32996473695239650 & 13667563554568793 & 5661280512684809 & 8006216968466729 & 1942616981990173 \\
 & 33809973988413085 & 13070566784642598 & 5052938405996713 & 5401726365514709 & 863018934527914 \\
 & 37889062373143906 & 14472334024676221 & 5527939700884757 & 5527939700884757 & 806515533049393 \\
 & 40254049821997474 & 15561775671567955 & 6016012379451749 & 6431277192706391 & 1027507207655485 \\
 & 41836262445484585 & 17294972731611518 & 7149684610974845 & 10048655749349969 & 2413587554589082 \\
 & 46890399112653461 & 17910553065691297 & 6841227995271610 & 6841260084420430 & 998132675950169 \\
 & 48306462250737473 & 18484066814863809 & 7072774740632834 & 7145738193853954 & 1057034027250125 \\
 & 53243232159418993 & 22010386141759859 & 9098942315500874 & 12787926265860584 & 3071396148367705 \\
 & 58122604880389733 & 22240110226042298 & 8509985123419585 & 8597725797737161 & 1271809634912210 \\
 & 59341817924539925 & 24580185800219268 & 10181446324101389 & 14398739476117879 & 3493720040136818 \\
290 & 70014236857598818 & 26749958796218487 & 10220211312947065 & 10235639531056643 & 1496385896811481 \\
 & 74119312578961858 & 28655029879238381 & 11078229260495189 & 11845777058753285 & 1893196647988273 \\
 & 84408473716542145 & 32242381632832033 & 12315957481333442 & 12318671181953954 & 1797801251550925 \\
 & 84495546415824461 & 34997169175011122 & 14495460438077785 & 20495961109208905 & 4971675308457794 \\
 & 99194853094755497 & 37889062373143906 & 14472334024676221 & 14472334024676221 & 2111485077978050 \\
 & 101791412246620601 & 38881073241687600 & 14851330019499401 & 14851807478442199 & 2166942971989610 \\
 & 107246981290506205 & 44335547834263417 & 18328169036663258 & 25759662212284046 & 6187215614895025 \\
 & 120987879709637005 & 50015557712772738 & 20676087714928729 & 29058793428681209 & 6979322867359138 \\
 & 122760633575886146 & 46890427116134699 & 17910563762074237 & 17910647772517951 & 2613144737746085 \\
 & 148050629787671833 & 56550309487554483 & 21600296518323130 & 21600298674991616 & 3151441527253105 \\
300 & 152098457523687709 & 58199154485574821 & 22269401267979938 & 22499005933036754 & 3328142022058225 \\
\end{longtable}
}

\newpage

\section{Some numerical experimentation}\label{experiments}

For funsies we have done some low-effort numerical investigation of the uniqueness conjecture
and the growth of Markoff numbers. The highest published numerically confirmed bound that we
are aware of is uniqueness of Markoff numbers up to $10^{140}$ \cite{Baragar96}. We have checked
uniqueness for all $215596025$ Markoff numbers up to $10^{15000}$. The computation was done by
a straightforward implementation of the mutation algorithm in Python and took about 29 hours on a laptop.
Following an argument of Button \cite [p. 84]{Button01}, we further conclude that the uniqueness
conjecture is true for all Markoff numbers of the form $N p^n$, where $p$ is a prime number,
$n \geq 0$, and $N \leq 10^{3750}$.

Zagier estimated the number $M(n)$ of Markoff numbers $\leq n$ with
$C \cdot (\log n)^2 +$ $O(\log n (\log \log n)^2)$ with
$C \approx 0.180717104711507$ (\cite{Zagier82}, see also \cite{BaragarUmeda}).
Figure \ref{zagierdev1} shows the deviation of $M(n)$ from $C \log(n)^2$. In the chosen scale
the deviation appears to follow a linear growth with slope close to $1/log(3) \approx 0.910239$
(only every 100-th value is taken into account).
The gray line represents $-2.038389 + 0.914755 \log_{10} n$ obtained from quadratic regression.

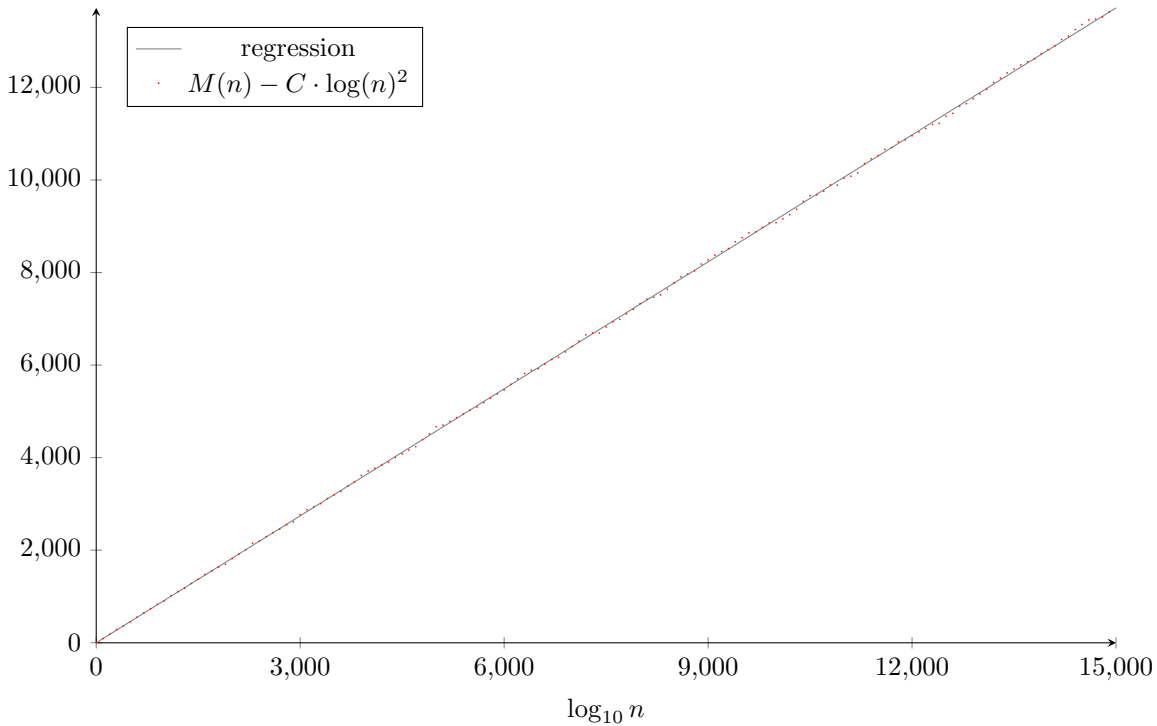
\begin{figure}
\begin{tikzpicture}
\begin{axis}[
axis lines = left,
xlabel = $\log_{10} n$,
scaled ticks=false,
xtick distance=3000,
legend pos = north west,
width=15cm,
height=10cm
]

\addplot [
domain=0:15000, 
samples=120, 
color=gray
]
{0.9147551564680976 * x - 2.038389099852793};

\addplot[mark=none, mark size=0.1, only marks, color=red] table[x = n, y = diff] {
n diff
0 1.0
100 88.56323998934204
200 186.25295995736815
300 285.0691599040583
400 362.0118398294726
500 448.0809997334436
600 552.2766396162333
700 646.598759477667
800 735.0473593178904
900 827.6224391363794
1000 903.3239989337744
1100 1018.1520387104247
1200 1103.1065584649332
1300 1176.1875581985805
1400 1279.3950379106682
1500 1374.7289976011962
1600 1475.1894372715615
1700 1553.7763569196686
1800 1631.4897565455176
1900 1694.3296361514367
2000 1817.2959957350977
2100 1914.3888352988288
2200 2011.6081548416987
2300 2147.953954361379
2400 2196.4262338597327
2500 2296.024993338622
2600 2377.750232794322
2700 2456.601952230558
2800 2550.5801516426727
2900 2611.684831033461
3000 2767.9159904047847
3100 2871.273629758507
3200 2939.757749086246
3300 3004.3683483917266
3400 3116.1054276786745
3500 3194.9689869415015
3600 3270.9590261820704
3700 3386.075545405969
3800 3472.318544605747
3900 3614.688023785129
4000 3710.1839829403907
4100 3770.8064220827073
4200 3840.555341195315
4300 3903.43074028939
4400 4004.432619366795
4500 4080.5609784163535
4600 4162.815817445517
4700 4237.197136454284
4800 4387.704935438931
4900 4515.339214403182
5000 4668.099973354489
5100 4701.987212274224
5200 4785.000931177288
5300 4864.1411300599575
5400 4946.407808922231
5500 5024.800967756659
5600 5092.320606570691
5700 5189.966725364327
5800 5283.739324133843
5900 5372.638402897865
6000 5461.663961619139
6100 5584.8160003349185
6200 5704.094519034028
6300 5817.499517686665
6400 5889.030996344984
6500 5922.68895496428
6600 6021.473393566906
6700 6118.384312145412
6800 6168.421710714698
6900 6283.585589244962
7000 6400.875947766006
7100 6518.29278627038
7200 6653.8361047282815
7300 6694.505903199315
7400 6690.302181623876
7500 6823.2249400392175
7600 6934.2741784229875
7700 6989.449896812439
7800 7105.752095140517
7900 7203.180773489177
8000 7327.735931761563
8100 7425.41757003963
8200 7465.225688330829
8300 7518.160286538303
8400 7635.2213647812605
8500 7777.408922985196
8600 7907.72296115756
8700 7969.163479298353
8800 8033.730477467179
8900 8183.423955544829
9000 8273.243913665414
9100 8375.190351739526
9200 8453.263269782066
9300 8518.462667807937
9400 8665.788545817137
9500 8752.240903794765
9600 8857.819741755724
9700 8879.525059714913
9800 8976.356857612729
9900 9071.315135538578
10000 9078.399893417954
10100 9160.61113126576
10200 9249.948849096894
10300 9370.413046911359
10400 9539.003724709153
10500 9661.720882475376
10600 9680.56452023983
10700 9755.53463794291
10800 9893.631235688925
10900 9887.854313358665
11000 10038.203871026635
11100 10081.679908663034
11200 10150.282426282763
11300 10354.011423930526
11400 10458.86690145731
11500 10527.848859012127
11600 10660.957296535373
11700 10699.192214086652
11800 10822.553611591458
11900 10869.041489064693
12000 10954.655846476555
12100 11035.39668393135
12200 11112.264001339674
12300 11199.257798731327
12400 11223.378076076508
12500 11379.62483343482
12600 11438.99807074666
12700 11593.497788071632
12800 11649.123985379934
12900 11755.876662611961
13000 11857.75581985712
13100 11961.761457055807
13200 12107.893574267626
13300 12203.152171432972
13400 12313.537248581648
13500 12398.048805713654
13600 12485.686842799187
13700 12556.451359957457
13800 12619.34235700965
13900 12732.359834045172
14000 12815.503791064024
14100 12898.774228066206
14200 13036.171145051718
14300 13112.69454202056
14400 13252.344419002533
14500 13359.12077587843
14600 13457.02361279726
14700 13479.052929669619
14800 13522.208726495504
14900 13629.491003334522
15000 13697.89976015687
};
\legend{regression, $M(n) - C \cdot \log(n)^2$}
\end{axis}
\end{tikzpicture}
\caption{Deviation of $M(n)$ from Zagier's estimate.}\label{zagierdev1}
\end{figure}

Zagier also conjectured an improved formula, $M(n) = C \cdot (\log 3n)^2 + o(\log n)$.
Figure \ref{zagierdev2} shows the differences $M(n) - C \cdot (\log 3n)^2$
for every $n = 10^k$, $0 \leq k \leq 15000$.

The counts are attached to this document here:
\attachfile[color=0 0 1, icon=Paperclip, subject=Counts of Markoff numbers]{counts1.txt}

\begin{figure}
\includegraphics{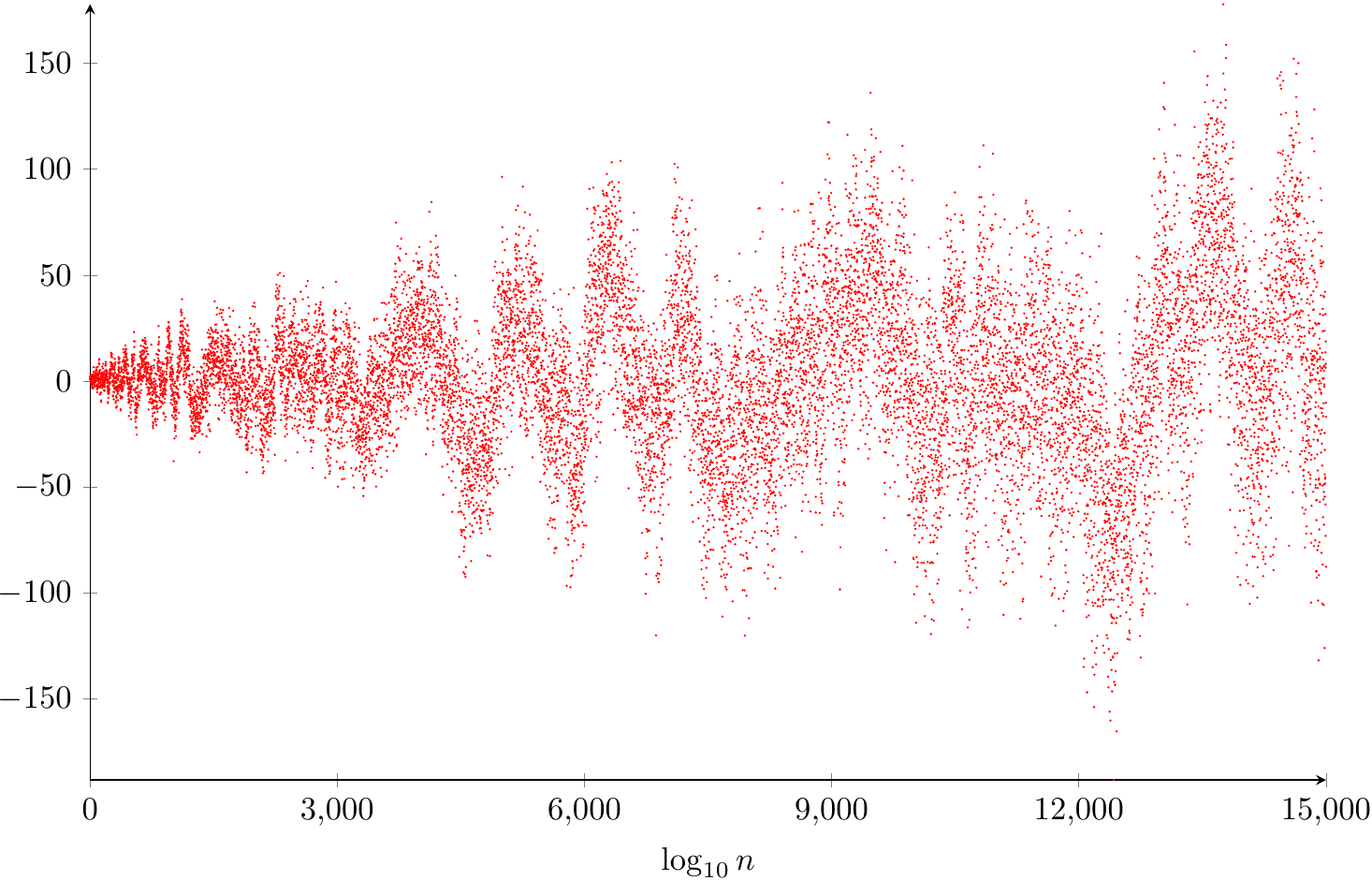}
\caption{Deviation of $M(n)$ from $C \cdot (\log 3n)^2$.}\label{zagierdev2}
\end{figure}


\end{document}